\documentclass{amsart}

\usepackage{epsfig}
\usepackage{amsmath}
\usepackage{amssymb}
\usepackage[all]{xy}

\def\comment#1{{\sf{[#1]}}}

\def\Z{{\mathbb Z}}
\def\N{{\mathbb N}}
\def\Q{{\mathbb Q}}
\def\R{{\mathbb R}}
\def\C{{\mathbb C}}
\def\A{{\mathbb A}}
\def\D{{\mathbb D}}
\def\H{{\mathbb H}}
\def\P{{\mathbb P}}
\def\V{{\mathbb V}}
\def\L{{\mathbb L}}

\def\B{{\mathcal B}}
\def\E{{\mathcal E}}
\def\M{{\mathcal M}}
\def\cN{{\mathcal N}}
\def\O{{\mathcal O}}

\def\cH{{\mathcal H}}
\def\cL{{\mathcal L}}
\def\cP{{\mathcal P}}
\def\cV{{\mathcal V}}

\def\bP{{\pmb \cP}}

\def\t{\mathbf{t}}
\def\a{\mathbf{a}}
\def\b{\mathbf{b}}
\def\bw{\mathbf{w}}

\def\g{{\mathfrak g}}
\def\h{{\mathfrak h}}
\def\p{{\mathfrak p}}

\def\Pminus{{\P^1-\{0,1,\infty\}}}

\def\l{{\ell}}
\def\d{{\epsilon}}		
\def\w{{\omega}}

\def\ehat{{\hat{\pmb{\d}}}}

\def\Mbar{\overline{\M}}
\def\Ebar{{\overline{E}}}
\def\cEbar{{\overline{\E}}}

\def\cLbar{\overline{\cL}}
\def\Hbar{\overline{\cH}}
\def\bPbar{{\overline{\bP}}}

\def\Ahat{\widehat{A}}
\def\Shat{\widehat{S}}
\def\That{\widehat{T}}

\def\G{{\Gamma}}

\def\Mtilde{\widetilde{M}}

\def\xbar{{\bar{x}}}

\def\alphatilde{{\tilde{\alpha}}}

\def\adual{{\check{\a}}}
\def\bdual{{\check{\b}}}

\def\v{{\vec{v}}}

\def\sl{\mathfrak{sl}}
\def\gl{\mathfrak{gl}}

\def\SL{{\mathrm{SL}}}

\def\SLtilde{\widetilde{\SL}}
\def\Gm{{\mathbb{G}_m}}

\def\un{\mathrm{un}}
\def\top{\mathrm{top}}
\def\zag{{\mathrm{Zag}}}
\def\norm{{\mathrm{norm}}}
\def\bch{{\mathrm{BCH}}}

\def\canon{{\mathrm{can}}}
\def\KZB{{\mathrm{KZB}}}
\def\dR{{\mathrm{dR}}}

\def\Ql{{\Q_\l}}

\def\dot{{\bullet}}
\def\bs{\backslash}
\def\bbs{{\bs\negthickspace \bs}}
\def\blank{\phantom{x}}

\def\shalf{\scriptstyle{\frac{1}{2}}}

\def\ll{\langle\langle}
\def\rr{\rangle\rangle}
\def\comptensor{\hat{\otimes}}

\newcommand\im{\operatorname{im}}               
\newcommand\id{\operatorname{id}}
\newcommand\ad{\operatorname{ad}}
\newcommand\Ad{\operatorname{Ad}}

\newcommand\Spec{\operatorname{Spec}}
\newcommand\Hom{\operatorname{Hom}}
\newcommand\End{\operatorname{End}}
\newcommand\Aut{\operatorname{Aut}}
\newcommand\Der{\operatorname{Der}}

\newcommand\Bl{\operatorname{Bl}}
\newcommand\Gr{\operatorname{Gr}}
\newcommand\Diff{\operatorname{Diff}}

\newcommand\Res{\operatorname{Res}}

\renewcommand\div{\operatorname{div}}
\renewcommand\Im{\operatorname{Im}}

\numberwithin{equation}{section}

\newtheorem{theorem}{Theorem}[section]
\newtheorem{lemma}[theorem]{Lemma}
\newtheorem{proposition}[theorem]{Proposition}
\newtheorem{corollary}[theorem]{Corollary}

\theoremstyle{definition}
\newtheorem{definition}[theorem]{Definition}
\newtheorem{example}[theorem]{Example}

\theoremstyle{remark}
\newtheorem{remark}[theorem]{Remark}

\begin{document}

\title{Notes on the Universal Elliptic KZB Connection}

\dedicatory{To Eduard Looijenga, friend and colleague}

\author{Richard Hain}

\address{Department of Mathematics\\ Duke University, Box 90320\\
Durham, NC 27708}

\email{hain@math.duke.edu}

\date{\today}

\thanks{Supported in part by the NSF through grants DMS-0706955 and
DMS-1005675.}



\maketitle

\tableofcontents


\section*{\sc Introduction}

The universal elliptic KZB\footnote{For Knizhnik--Zamolodchikov--Bernard.}
connections generalize the connections defined by the physicists Knizhnik and
Zamolodchikov \cite{kz} in genus 0 and Bernard \cite{kzb} in genus 1. For each
$n\ge 1$, the universal elliptic KZB connection is an integrable connection on a
bundle of pronilpotent Lie algebras over $\M_{1,1+n}$, the moduli space of
$(n+1)$-pointed smooth projective curves of genus 1, regarded as a stack over
$\C$. The fiber of the connection over the point corresponding to the
$(n+1)$-pointed genus 1 curve $(E;0,x_1,\dots,x_n)$ is the Lie algebra of the
unipotent completion of  $\pi_1^\un(C_n(E',(x_1,\dots,x_n)))$, where $E' :=
E-\{0\}$ and
$$
C_n(E') = (E')^n - \text{fat diagonal}
$$
is the configuration space of $n$ points in $E'$. Explicit constructions of the
universal elliptic KZB connection were given by Calaque, Enriquez and Etingof
(for all $n\ge 1$) in \cite{cee} and, independently, by Levin and Racinet (for
$n=1$ only) in \cite{levin-racinet}. We will generally drop the adjectives
``universal'' and ``elliptic''. Since we consider only universal elliptic KZB
connections, there should be no confusion.

In mathematics, KZB connections play a role in representation theory \cite{cee}
and in the study of periods of mixed elliptic motives
\cite{enriquez,hain-matsumoto:mem}. In this paper, we focus on the KZB
connection over $\M_{1,2}$, the $n=1$ case. This is the most important in the
theory of mixed elliptic motives.

In this paper, we give a complete exposition of the construction of the KZB
connection in the $n=1$ case. We use it to compute the limit mixed Hodge
structure (MHS) on the Lie algebra of the unipotent fundamental group of the
first order Tate curve $E_{\partial/\partial q}$ (i.e., the restriction of the
universal elliptic curve over the $q$-disk to the tangent vector
$\partial/\partial q$ at $q=0$) with its identity removed and with a canonical
tangential base point $\partial/\partial w$ at its identity.\footnote{Limit
mixed Hodge structures are reviewed in Section~\ref{sec:hodge}. Tangential base
points and their relationship to limit MHSs are explained in
Section~\ref{sec:pause}.} In particular, we show that its periods are multiple
zeta values. We also use it to derive certain formulas which relate this limit
MHS to the MHS on the unipotent fundamental group of $\Pminus$, which we regard
as the nodal cubic with its singular point and identity element removed. We also
show that, when restricted to $\M_{1,\vec{1}}$, the moduli space of elliptic
curves with a non-zero abelian differential (equivalently, a non-zero tangent
vector at the identity), the elliptic KZB connection is defined over $\Q$ and we
give an explicit formula for this connection in terms of the coordinates on
$\M_{1,\vec{1}/\Q}$.

This paper grew out of notes from a seminar at Duke University during the summer
of 2007 in which we read the paper of Levin and Racinet \cite{levin-racinet}.
Because this paper is derived from lecture notes, the style is sometimes a
little expansive and background which might otherwise be omitted is included.

The paper is in four parts. The first contains some background material. The
second part is a complete exposition of the elliptic KZB equation. This
exposition follows the approach of Levin and Racinet, which expresses the
elliptic KZB connection in terms of Kronecker's Jacobi form, $F(\xi,\eta,\tau)$,
\cite{kronecker} and Eisenstein series. This function $F$ was rediscovered by
Zagier in \cite{zagier} and can be expressed in terms of classical theta
functions. Zagier \cite{zagier} showed that this Jacobi form is a generating
function for the periods of modular forms of level 1, a fact whose relevance is
still not completely understood in the context of mixed elliptic motives.

During the seminar, we were unable to verify some of the computations in the
Levin-Racinet paper without modifying several factors of automorphy. Such
differences may have arisen because of differing conventions. This paper uses
the modified factors of automorphy. Because of this, and because it is not
likely that the paper of Levin and Racinet will be published, complete proofs of
the modular behaviour and integrability of the elliptic KZB connection are given
in Part 2.

Levin and Racinet define Hodge and weight filtrations on the fibers of the
elliptic KZB connection. In Part 3, we prove that with these filtrations, the
KZB connection is an admissible variation of MHS isomorphic to the canonical
variation of MHS whose fiber over $[E,x]$ is the Lie algebra of
$\pi_1^\un(E-\{0\},x)$ with its canonical MHS. This allows us to explicitly
compute the limit MHS on the fiber associated to a tangent vector at the
identity of the nodal cubic. In particular, we prove that its periods are
multiple zeta values. The explicit formula for the KZB connection allows us to
compute a formula for the canonical map of Lie algebras induced by the
homomorphism
$$
\pi_1^\un(\Pminus,\partial/\partial w) \to
\pi_1^\un(E'_{\partial/\partial q},\partial/\partial w)
$$
and also for the logarithm of the monodromy action
$$
\pi_1^\un(E'_{\partial/\partial q},\partial/\partial w) \to
\pi_1^\un(E'_{\partial/\partial q},\partial/\partial w).
$$
Here $w$ is the parameter in $\Pminus$ and $q$ is the coordinate $\exp(2\pi i
\tau)$ in the $q$-disk.

For applications to elliptic and modular motives, it is important to know that
the elliptic KZB connection is defined over $\Q$. Levin and Racinet
\cite{levin-racinet} state this as a result and sketch a proof of it. In Part 4
we elaborate on their computations and give an explicit formula for the
restriction of the KZB connection to $\M_{1,\vec{1}}$ and show that its
canonical extension to $\Mbar_{1,\vec{1}}$ is also defined over $\Q$. The story
for $\Mbar_{1,2}$ is more complicated and has been verified by Ma Luo. It will
appear in his Duke PhD thesis.

Background material on the topology of moduli spaces of elliptic
curves (viewed as orbifolds) and their associated mapping class groups is not
included. It can be found, for example, in \cite{hain:elliptic}. The books of
Serre \cite{serre} and Silverman \cite{silverman} are excellent references for
background material on modular forms.

\bigskip

\noindent{\em Acknowledgments:} I am indebted to all participants in the
seminar, particularly Aaron Pollack, whose help in checking the factors of
automorphy was invaluable.  I am grateful to Makoto Matsumoto, Francis Brown and
Benjamin Enriquez for their constructive comments, and to Ma Luo and Jin Cao for
their numerous corrections. Finally, I would like to thank the referee for
his/her thorough reading of the manuscript and for their constructive comments
on the exposition.

\subsection{Some Conventions:}

We use the topologist's convention for path multiplication: if $\alpha,\beta :
[0,1] \to X$ are paths in a topological space with $\alpha(1)=\beta(0)$, then
$\alpha\beta : [0,1]\to X$ is the path obtained by first traversing $\alpha$ and
then $\beta$.

The adjoint action of an element $u$ of the enveloping algebra of a Lie algebra
$\g$ on an element $x$ of $\g$ will often be denoted by $u\cdot x$. This will be
extended to power series $u$ of elements of $\g$ when it makes sense. For
example if $\t \in \g$, then
$$
e^\t \cdot x = \sum_{n=0}^\infty \ad_\t^n(x)/n!
$$
If $\delta$ is a derivation of $\g$, then $\delta (f\cdot u)= \delta(f)\cdot u +
f \cdot \delta(u)$.

We will be sloppy and denote the generic element of $\SL_2(\Z)$ by
$$
\gamma =
\begin{pmatrix}
a & b \cr c & d
\end{pmatrix}
$$
So, unless otherwise mentioned, the entries of $\gamma$ are $a$, $b$, $c$ and
$d$.

We will use the terms ``local system'' and ``locally constant sheaf''
interchangeably. Local systems of vector spaces over a smooth manifold
correspond to vector bundles with a flat (i.e., integrable) connection.
Sometimes we will abuse terminology and refer to such a local system as a ``flat
bundle''.

\part{Background}

In this part, we present the background needed to understand the universal
elliptic KZB connection. In parts 3 and 4, the reader will also need to be
familiar with the basics of Deligne's theory of mixed Hodge structures.
Introductory references are listed in Section~\ref{sec:hodge}.

\section{The Universal Elliptic Curve}
\label{sec:univ_curve}

The material in this section is standard. We will assume that the reader is
familiar with the construction of $\M_{1,1}$ as the orbifold quotient of the
upper half plane
$$
\h := \{\tau \in \C : \Im \tau > 0 \}
$$
by $\SL_2(\Z)$, the construction of its Deligne-Mumford compactification
$\Mbar_{1,1}$ (as an orbifold), the construction of the standard line bundle
$\cL$ over $\M_{1,1}$, and its extension $\cLbar$ to $\Mbar_{1,1}$. Denote their
$k$th powers by $\cL_k$ and $\cLbar_k$, respectively. In particular, their
inverses will be denoted by $\cL_{-1}$ and $\cLbar_{-1}$. This material is
classical and can be found, for example, in the first four sections of
\cite{hain:elliptic}.

The group $\SL_2(\Z)$ acts on $\Z^2$ by right multiplication:
$$
\begin{pmatrix} a & b \cr c & d\end{pmatrix} :
\begin{pmatrix} m & n \end{pmatrix} \mapsto
\begin{pmatrix}m & n\end{pmatrix}\begin{pmatrix} a & b \cr c & d\end{pmatrix}.
$$
Denote the corresponding semi-direct product $\SL_2(\Z) \ltimes \Z^2$ by $\G$.
This is the set $\SL_2(\Z)\times \Z^2$ with multiplication:
$$
(\gamma_1,v_1)(\gamma_2,v_2) = (\gamma_1\gamma_2,v_1\gamma_2+v_2)
$$
where $\gamma_1,\gamma_2 \in \SL_2(\Z)$ and $v_1,v_2 \in \Z^2$.

The group $\G$ acts on $X := \C\times \h$ on the left:
$$
(m,n) : (\xi,\tau)
\mapsto \bigg(\xi + \begin{pmatrix} m & n \end{pmatrix}
\begin{pmatrix} \tau \cr 1 \end{pmatrix}, \tau \bigg)
$$
and
$$
\gamma : (\xi,\tau) \mapsto
\big((c\tau+d)^{-1}\xi,\gamma \tau\big)
$$
where $\gamma \in \SL_2(\Z)$.

The quotient $\G\bs X$ is the universal elliptic curve $\E$; the map $\G\bs X
\to \SL_2(\Z)\bs \h$ induced by the projection $X \to \h$ is the projection $\E
\to \M_{1,1}$.

The universal elliptic curve can be compactified using the Tate curve to obtain
a proper orbifold map $\cEbar \to \Mbar_{1,1}$ whose fiber over $q=0$ is the
nodal cubic. Its pullback to the $q$-disk $\D$, with the double point removed,
is the quotient of $\C^\ast \times \D$ by the group action $\Z\times \C^\ast
\times \D \to \C^\ast\times \D$ defined by
$$
n : (w,q) \mapsto
\begin{cases}
(q^n w,q) & q \neq 0, \cr
(w,q) & q=0.
\end{cases}
$$
Note that, although this group action is not continuous, the quotient (endowed
with the quotient topology) is Hausdorff and is a complex manifold. The
fiber over $q=0$ is the group $\C^\ast$. The zero section (aka, the identity
section) passes through it at $w=1$.

\begin{proposition}
\label{prop:nbd}
The normal bundle of the zero section of $\cEbar$ is  $\cLbar_{-1}$.
\end{proposition}

\begin{proof}
The line bundle $\cL_{-1}$ is the quotient of $\C\times \h$ by the action
$$
\begin{pmatrix} a & b \cr c & d\end{pmatrix} :
(\xi,\tau) \mapsto \big((c\tau+d)^{-1}\xi,\gamma \tau\big).
$$
The identity $\C\times \h \to \C\times \h$ is equivariant with respect to the
natural inclusion $\SL_2(\Z) \to \SL_2(\Z)\ltimes \Z^2$ and thus induces a
quotient mapping $\cL_{-1} \to \E$ that commutes with the projections to
$\M_{1,1}$. This projection extends over $q=0$. This is well known and follows
from the result of Exercise~47 in \cite[\S5.2]{hain:elliptic}.
\end{proof}

\begin{corollary}
A neighbourhood of the zero section of $\cLbar_{-1}$ is biholomorphic with a
neighbourhood of the identity section of $\cEbar \to \Mbar_{1,1}$. \qed
\end{corollary}

Denote by $L'$ the complex manifold obtained by removing the 0-section from a
holomorphic line bundle $L$.

\begin{corollary}
The moduli space $\M_{1,\vec{1}}'$ of pairs $(E,\v)$, where $E$ is a stable
elliptic curve and $\v$ is a (possible vanishing) tangent vector at the identity
is naturally isomorphic with $\cLbar_{-1}$. In particular, the moduli space of
smooth elliptic curves and a non-zero tangent vector at the identity
$\M_{1,\vec{1}}$ is isomorphic to $\cL_{-1}'$. \qed
\end{corollary}

\subsection{Fundamental Groups}

A non-zero point $x$ of an elliptic curve $E$ determines (and is determined by)
an orbifold map $[E,x] : \C\to \E'$. 

\begin{proposition}
The fundamental group of $\E'$ with respect to the base point $[E,x]$ is an
extension
$$
1\to \pi_1(E',x) \to \pi_1(\E',[E,x]) \to \SL_2(\Z) \to 1.
$$
In particular, it is isomorphic to an extension of $\SL_2(\Z)$ by a free
group of rank $2$.
\end{proposition}

\begin{proof}
The function $\R^2 \times \h \to \C\times \h$ defined by $(u,v,\tau) \mapsto
(u+v\tau,\tau)$ is a homeomorphism. It induces a homeomorphism $(\R/\Z)^2 \times
\h \to \E_\h$ which restricts to give a homeomorphism
$$
\big((\R/\Z)^2 -\{0\}\big) \times \h \to \E_\h',
$$
where $\E_\h$ denotes the universal elliptic curve $\Z^2\bs\big(\C\times\h\big)$
over $\h$ and $\E_\h'$ denotes $\E_\h$ with the 0-section removed. It follows
that $\E_\h'$ is homotopy equivalent to each of its fibers $E_\tau'$. In
particular, the inclusion $(E',x) \to (\E_\h',(E,x))$ induces an isomorphism on
fundamental groups.

The result follows from covering space theory as the covering $\E_\h' \to \E'$
is Galois with Galois group $\SL_2(\Z)$.
\end{proof}

\begin{corollary}
For each point $[E,x]$ of $\E'$, there is a natural action of $\pi_1(\E',[E,x])$
on $\pi_1(E',x)$.
\end{corollary}

\begin{proof}
Since $\pi_1(E',x)$ is a normal subgroup of $\pi_1(\E',[E,x])$, one has the
conjugation action $g : \gamma \mapsto g\gamma g^{-1}$ of $\pi_1(\E',[E,x])$ on
$\pi_1(E',x)$.
\end{proof}

Denote the $\C^\ast$ bundle obtained from $\cL_k$ by removing the 0-section by
$\cL_k'$. Its (orbifold) fundamental group is a central extension
$$
0 \to \Z \to \pi_1(\cL_k',\ast) \to \SL_2(\Z) \to 1.
$$

\begin{remark}
It is well-known that $\pi_1(\cL_{-1}')$ is naturally isomorphic to each of
the following groups:
\begin{enumerate}

\item the braid group $B_3$ on 3-strings;

\item the fundamental group of $\C^2$ with the cusp $x^2=y^3$ removed;

\item the fundamental group of the complement of the trefoil knot;

\item the inverse image $\SLtilde_2(\Z)$ of $\SL_2(\Z)$ in the universal
covering group $\SLtilde_2(\R)$ of $\SL_2(\R)$. 
\end{enumerate}
Details can be found, for example, in \cite{hain:elliptic}.
\end{remark}

Proposition~\ref{prop:nbd} implies that if $E$ is an elliptic curve and $\v$ is
a non-zero tangent vector at $0\in E$, there is a natural homomorphism
$$
\pi_1(\cL_{-1}',[E,\v]) \to \pi_1(\E',[E,\v]).
$$
Composing this with the action above we obtain an action
$$
\pi_1(\cL_{-1}',[E,\v]) \to \Aut \pi_1(E',\v).
$$

Denote the element of $\pi_1(E',\v)$ that corresponds to moving once around  the
identity in the positive direction by $c_o$. Denote by $z_o$ the image  in
$\pi_1(\cL_{-1}',[E,\v]) \cong \SLtilde_2(\Z)$ of the {\em positive}
generator of the fundamental group of the fiber $\cL_{-1,E}'\cong\C^\ast$
over $[E]$ of the projection $\cL_{-1}' \to \M_{1,1}$.

\begin{proposition}
This action of $\pi_1(\cL_{-1}',[E,\v])$ on $\pi_1(E',\v)$ fixes $c_o$.
\end{proposition}

\begin{proof}
Observe that $c_o$ is the image of $z_o$ under the continuous mapping
$\cL_{-1}' \to \E'$. The result follows as $z_o$ is central in
$\SLtilde_2(\Z)$.
\end{proof}

Since $\pi_1(\cL_{-1}',[E,\v])$ acts on $\pi_1(E',\v)$, we can form the
semi-direct product
$$
\pi_1(\cL_{-1}',[E,\v]) \ltimes \pi_1(E',\v).
$$

\begin{lemma}
The element $c_o^{-1}z_o$ is central in $\pi_1(\cL_{-1}',[E,\v]) \ltimes
\pi_1(E',\v)$.
\end{lemma}

\begin{proof}
Note that $z_o$ acts on $\pi_1(E',\v)$ by conjugation by $c_o$. Since $z_o$ is
central in $\pi_1(\cL_{-1}',[E,\v])$ and since each element of
$\pi_1(\cL_{-1}',[E,\v])$ fixes $c_o$, we see that $c_o^{-1}z_o$ commutes
with each element of $\pi_1(\cL_{-1}',[E,\v])$.

If $g\in \pi_1(E',\v)$, then
$$
gc_o^{-1}z_o g^{-1} = gc_o^{-1}\big(z_o g^{-1} z_o^{-1}\big) z_o
= gc_o^{-1}\big(c_o g^{-1} c_o^{-1}\big) z_o = c_o^{-1}z_o.
$$
\end{proof}

This semi-direct product can be realized as the fundamental group of the
pullback $\E'_\cL$ of $\E'$ to $\cL_{-1}'$. This has a (continuous)
section. Since $\E'_\cL$ is a $\C^\ast$ covering of $\E'$, we obtain:

\begin{proposition}
\label{prop:cent-extn}
The kernel of the natural homomorphism
$$
\pi_1(\E_\cL',[E,\v])\cong
\pi_1(\cL_{-1}',[E,\v]) \ltimes \pi_1(E',\v) \to \pi_1(\E',[E,\v])
$$
is the infinite cyclic subgroup generated by $c_o^{-1}z_o$. This homomorphism
induces an isomorphism
$$
\big(\pi_1(\cL_{-1}',[E,\v]) \ltimes
\pi_1(E',\v)\big)/\langle c_o^{-1}z_o \rangle
\to \pi_1(\E',[E,\v]).
$$
\qed
\end{proposition}

In mapping class group notation, this result says that there is a natural
isomorphism
$$
\G_{1,2} \cong \big(\G_{1,\vec{1}}\ltimes \pi_1(E',\v)\big)/\Z.
$$
In the Hodge and Galois worlds, the copy of $\Z$ is a copy of $\Z(1)$.

\subsection{The local system $\H$}
\label{sec:H}

This is the local system (i.e., locally constant sheaf) over $\M_{1,1}$ whose
fiber over $[E] \in \M_{1,1}$ is $H_1(E;\C)$. We identify it, via Poincar\'e
duality $H_1(E) \to H^1(E)$, with the local system $R^1\pi_\ast\C$ over
$\M_{1,1}$ associated to the universal elliptic curve $\pi : \E \to \M_{1,1}$.
This has fiber $H^1(E;\C)$ over $[E] \in \M_{1,1}$.

We consider two ways of framing (i.e., trivializing) the pullback of $\H$ to
$\h$. Denote the universal elliptic curve over $\h$ by $\E_\h \to \h$. It is the
quotient of $\C\times \h$ by the standard action of $\Z^2$ given above. The
first homology of $E_\tau := \C/(\Z\oplus \tau\Z)$ is naturally isomorphic to
$\Lambda_\tau := \Z\oplus\tau\Z$. Let $\a,\b$ be the basis of $H_1(E_\tau;\Z)$
that corresponds to the basis $1,\tau$ of $\Lambda_\tau$.

Denote the dual basis of $H^1(E_\tau;\C)\cong \Hom(H_1(E_\tau),\C)$ by
$\adual,\bdual$. Then, under Poincar\'e duality,
$$
\adual = -\b \text{ and } \bdual = \a.
$$

Denote the element $d\xi$ of $H^1(E_\tau,\C)$ by $w_\tau$. Then
$$
w_\tau = \adual + \tau \bdual = \tau \a - \b.
$$
The two framings $\a,\b$ and $2\pi i \bdual, \w_\tau$ of $\H$ over $\h$ are
related
by
$$
\begin{pmatrix} 2\pi i \bdual & w_\tau \end{pmatrix} = 
\begin{pmatrix} \bdual & \adual \end{pmatrix}
\begin{pmatrix} 2\pi i & \tau \cr 0 & 1 \end{pmatrix} =
\begin{pmatrix} \a & \b \end{pmatrix}
\begin{pmatrix} 2\pi i & \tau \cr 0 & -1 \end{pmatrix}.
$$

\begin{remark}
The local system $\H$ underlies a polarized variation of Hodge structure over
$\h$ of weight $-1$. The Hodge subbundle $F^0\cH$ of the corresponding flat
bundle $\cH =\H\otimes_\Q\O_{\h}$ is $\O(\h)\w$.
\end{remark}

\section{Unipotent Completion}

Suppose that $\pi$ is a discrete group and that $R$ is a commutative ring.
Denote the group algebra of $\pi$ over $R$ by $R\pi$. This is an $R$-algebra.
The {\em augmentation} is the homomorphism $\epsilon : R\pi \to R$ that takes
each $\gamma \in \pi$ to 1. Its kernel, denoted $J$, is called the {\em
augmentation ideal}. The powers of $J$ define a topology on $R\pi$. A base of
neighbourhoods of 0 consist of the powers of $J$:
$$
R\pi \supseteq J \supseteq J^2 \supseteq J^3 \supseteq \cdots
$$
The completion of $R\pi$ in this topology is called the $J$-adic completion
of $\pi$ and is denoted by $R\pi^\wedge$. In concrete terms:
$$
R\pi^\wedge = \varinjlim_n R\pi/J^n.
$$
Denote its augmentation ideal by $J^\wedge$.

The group algebra also has a ``coproduct''
$$
\Delta : R\pi \to R\pi \otimes R\pi.
$$
This is an augmentation preserving algebra homomorphism, which is continuous
in the $J$-adic topology. It thus induces a ring homomorphism
$$
\Delta : R\pi^\wedge \to R\pi^\wedge \hat{\otimes} R\pi^\wedge.
$$

Now suppose that $R$ is a field $F$ of characteristic zero. Note that
each element of $1+J^\wedge$ is a unit. Define
$$
\cP(F)=\{x\in F\pi^\wedge : \epsilon(x) = 1\text{ and } \Delta{x}=x\otimes x\}
$$
and
$$
\p = \{x \in F\pi^\wedge : \Delta x = x\otimes 1 + 1 \otimes x\}.
$$
Elements of $\p$ are said to be {\em primitive}; elements of $\cP$ are said to
be {\em group-like}.

\begin{proposition}
\label{prop:exp}
\begin{enumerate}

\item $\cP(F)$ is a subgroup of the group $1+J^\wedge$;

\item $\p$ is a Lie algebra, with bracket $[u,v] = uv-vu$, which lies in
$J^\wedge$;

\item The logarithm and exponential mappings
$$
\xymatrix{
J^\wedge \ar@/^/[rr]|\exp && 1 + J^\wedge \ar@/^/[ll]|\log
}
$$
are continuous bijections, which induce continuous bijections
$$
\xymatrix{
\p \ar@/^/[rr]|\exp && \cP(F) \ar@/^/[ll]|\log
}.
$$
\end{enumerate}
\end{proposition}

The third part implies that the exponential map
$$
\exp : (\p,\bch) \to \cP
$$
is a group isomorphism, where the multiplication on $\p$ is defined using the
Baker-Campbell-Hausdorff formula \cite{serre:la-lg}:
$$
\bch(u,v) := \log(e^u e^v) = u + v + \frac{1}{2}[u,v] + \cdots
$$

\begin{proof}
The first two assertions are easily verified, as is the first part of the third
assertion. To prove the last assertion, note that since $\exp$ is continuous,
$\exp \Delta (x) = \Delta \exp(x)$ for all $x\in J^\wedge$. Now, $x \in
J^\wedge$ is primitive if and only if
$$
\Delta x = x\otimes 1 + 1\otimes x.
$$
Since $x\otimes 1$ and $1\otimes x$ commute, this holds if and only if
$$
\Delta\exp(x) = \exp(\Delta(x))
= \exp(x\otimes 1)\exp(1\otimes x) = \exp(x)\otimes\exp(x).
$$
That is, $x \in J^\wedge$ is primitive if and only if $\exp x$ is group-like.
\end{proof}

Since $\epsilon(\gamma) = 1$ for all $\gamma \in \pi$, there is a homomorphism
$\pi \to 1 + J^\wedge$. By the definition of the coproduct $\Delta$, the image
of this homomorphism lands in $\cP(F)$. Thus, the inclusion $\pi \to F\pi$
induces a natural homomorphism $\pi \to \cP(F)$

\begin{definition}
Suppose that $H_1(\pi;F)$ is finite dimensional (e.g., $\pi$ is finitely
generated). The homomorphism $\pi \to \cP(F)$ is called the {\em unipotent (or
Malcev) completion of $\pi$ over $F$.} The prounipotent group $\cP$ is denoted
$\pi^\un$. The Lie algebra of the unipotent completion is the Lie algebra $\p$.
It is also called the Malcev Lie algebra associated to $\pi$.
\end{definition}

Unipotent completion can be viewed as a functor from the category of groups
to the category of prounipotent groups over $F$:
$$
\xymatrix{\pi \ar@{~>}[r] & \cP(F)}
$$
There is also the functor $\xymatrix{\pi \ar@{~>}[r] & \p}$ that assigns to a
group, the Lie algebra of its unipotent completion over $F$. There are therefore
natural homomorphism
$$
\Aut\pi \to \Aut \cP \text{ and } \Aut \pi \to \Aut \p.
$$

\begin{remark}
When $\pi$ is the fundamental group of an algebraic variety, $\p$ carries
additional structure: If $\pi$ is the fundamental group of a complex algebraic
variety and $F=\Q$, then $\p$ has a natural mixed Hodge structure; if $\pi$ is
the fundamental group of a smooth algebraic variety defined over $\Q$ with
$\Q$-rational base point, then the absolute Galois group $G_\Q$ acts on
$\p\otimes \Ql$.
\end{remark}

\subsection{The unipotent completion of a free group}

Suppose that $\pi$ is the free group $\langle x_1,\dots,x_n\rangle$ generated by
the set $\{x_1,\dots,x_n\}$.

Consider the ring
$$
F\ll X_1,\dots,X_n\rr
$$
of formal power series in the non-commuting indeterminants $X_j$. Define
an augmentation
$$
\epsilon : F\ll X_1,\dots,X_n\rr \to F
$$
by sending a power series to its constant term. The augmentation ideal
$\ker\epsilon$ is the maximal ideal $I=(X_1,\dots,X_n)$.

Define a coproduct
$$
\Delta : F\ll X_1,\dots,X_n\rr \to
F\ll X_1,\dots,X_n\rr\hat{\otimes} F\ll X_1,\dots,X_n\rr
$$
by defining each $X_j$ to be primitive:
$$
\Delta X_j := X_j \otimes 1 + 1 \otimes X_j.
$$
There is a unique group homomorphism
$$
\pi \to F\ll X_1,\dots,X_n\rr
$$
that takes $x_j$ to $\exp(X_j)$. This extends to a ring homomorphism
$$
\theta : F\pi \to F\ll X_1,\dots,X_n\rr.
$$
Since $\epsilon (x_j) = 1 = \epsilon(\exp(X_j))$, $\theta$ is augmentation
preserving, and therefore extends to a continuous homomorphism
$$
\hat{\theta} : F\pi^\wedge \to F\ll X_1,\dots,X_n\rr
$$

As in the case of completed group algebras, one can define primitive and
group-like elements of $F\ll X_1,\dots,X_n\rr$. As there, an element of $1+I$ is
group-like if and only if it is the exponential of a primitive element. Since
$\exp(X_j)$ is group-like, it is easy to check that $\hat{\theta}$ preserves
both the product and the coproduct. (One says that it is a homomorphism of
complete Hopf algebras.)

It is easy to use universal mapping properties to prove:

\begin{proposition}
The homomorphism $\hat{\theta}$ is an isomorphism of complete Hopf algebras.
\qed
\end{proposition}

\begin{corollary}
The restriction of $\hat{\theta}$ induces a natural isomorphism
$$
d\theta : \p \to \L(X_1,\dots,X_n)^\wedge
$$
of topological Lie algebras.
\end{corollary}

\begin{proof}
This follows immediately from the fact that $\hat{\theta}$ induces an
isomorphism on primitive elements and the well-known fact that the set of
primitive elements of the power series algebra $F\ll X_1,\dots,X_n\rr$ is the
completed free Lie algebra $\L(X_1,\dots,X_n)^\wedge$.
\end{proof}

There is a weaker version of the construction of the unipotent completion of a
free group, which will be relevant later. Suppose that
$$
\theta : \pi \to F\ll X_1,\dots,X_n\rr
$$ 
is a homomorphism that satisfies $\theta(x_j) = \exp(U_j)$, where $U_j \in
J^\wedge$ and $U_j \equiv X_j \bmod (J^\wedge)^2$. Then it is not difficult to
show that $\theta$ induces a continuous isomorphism
$$
\hat{\theta} : F\pi^\wedge \to F\ll X_1,\dots,X_n\rr
$$
and, by restriction, a Lie algebra isomorphism
$$
d\theta : \p \to \L(X_1,\dots,X_n)^\wedge
$$
and a group isomorphism
$$
\cP \to \exp\L(X_1,\dots,X_n)^\wedge.
$$

\section{Factors of Automorphy}
\label{sec:automorphy}

Suppose that $G$ is a group that acts on a space (or set) $X$ on the left.
Suppose that $V$ is a left $G$-module (or left $G$-space, etc.). A function $M :
G\times X \to \Aut V$ (written $(g,x) \mapsto M_g(x)$) is a {\em factor of
automorphy} if the function
$$
V\times X \to V \times X,\qquad g : (v,x) \mapsto (M_g(x)v,gx)
$$
is an action. This is equivalent to the condition
$$
M_{gh}(x) = M_g(hx)M_h(x)\quad\text{ all } g,h \in G,\ x \in X.
$$
Note that the projection $V \times X \to X$ is $G$-equivariant; $G$-equivariant
sections of this projection correspond to functions $f: X \to V$ satisfying
$f(gx) = M_g(x)f(x)$ and, by definition, to sections of the ``bundle'' $G\bs
(X\times V) \to G\bs X$.\footnote{More precisely, $G$-invariant sections of
$V\times X \to X$ correspond to section of the stack bundle $G\bbs(V\times X)
\to G\bbs X$.} Such bundles are flat in the sense that they give rise to a
locally constant sheaf. An open set in $G\bs X$ corresponds to a $G$-invariant
open set $U$ in $X$. The set of constant sections of the bundle over this set
is, by definition, the set of $G$-invariant locally constant sections of
$V\times X \to X$. When $V$ is a real or complex vector space, this bundle has a
natural flat connection $\nabla$ which is characterized by the property that a
local section $s$ is constant if and only if $\nabla s = 0$. In such cases, we
will refer to the bundle $G\bs(V\times X) \to X$ as being a {\em flat bundle}.

\bigskip

Three examples that will be generalized and combined to form $\cP$ are:

\begin{example}
Fix $k\in \Z$. Let $G = \SL_2(\Z)$, $X = \h$, $V=\C$, and $A_\gamma(\tau) =
(c\tau + d)^k$. The (orbifold) quotient of $\C\times \h \to \h$ is the line
bundle $\cL_k \to \M_{1,1}$.
\end{example}

Note that the fibered product $\E\times_{\M_{1,1}}\E \to \M_{1,1}$ of the
universal elliptic curve is the quotient of $\C\times\C\times\h$ by the
$\SL_2(\Z)\ltimes(\Z^2\oplus\Z^2)$-action
$$
\big((m,n),(r,s)\big) : (\xi,\eta,\tau)
= \big(\xi + m\tau+n,\eta+r\tau+s, \tau \big)
$$
and
$$
\gamma : (\xi,\tau) \mapsto
\big((c\tau+d)^{-1}\xi,(c\tau+d)^{-1}\eta,\gamma \tau\big)
$$
where $\gamma \in \SL_2(\Z)$.

\begin{example}
\label{ex:linebdle}
Suppose that $G = \SL_2(\Z)\ltimes(\Z^2\oplus\Z^2)$ and that $X=\C\times\C
\times \h$, where the $G$-action is the one defined above.
Let $V=\C$. Define
$$
A_\gamma(\xi,\eta,\tau) =
\begin{cases}
(c\tau+d)e\big(c\xi\eta/(c\tau+d)\big) & \gamma \in \SL_2(\Z),\cr
e(\tau)^{-mr} e(\xi)^{-r} e(\eta)^{-m} & \gamma = \big((m,n),(r,s)\big)
\end{cases}
$$
where $e(u) = \exp(2\pi i u)$. This is a well-defined factor of automorphy. The
quotient
$$
G\bs\big(\C\times X) \to G\bs X
$$
is a line bundle
$$
\cN \to \E\times_{\M_{1,1}}\E
$$
over the self product over $\M_{1,1}$ of the universal elliptic curve. The 
restriction of $\cN$ to the zero section $\M_{1,1}$ is the line bundle $\cL =
\cL_1$. (Just look at the factor of automorphy when $\xi=\eta=0$.)
\end{example}

\begin{remark}
Later (Prop.~\ref{prop:poincare}) we will see that the restriction of $\cN$ to
the fiber $E^2$ over $[E]$ is the pullback of the Poincar\'e line bundle
over $E\times E$ to $E\times E$ along the map $(\xi,\eta) \mapsto (\xi,-\eta)$.
\end{remark}

The next example gives an alternative description of the local system $\H$.

\begin{example}
\label{ex:hodge}
Let $G=\SL_2(\Z)$, $X=\h$ and $V = \C^2$. Then
$$
M_\gamma(\tau) =
\begin{pmatrix}
(c\tau + d)^{-1} & 0 \cr 2\pi i c & c\tau + d
\end{pmatrix}
$$
is a factor of automorphy. The resulting bundle is the vector bundle associated
to the local system $\H \to \M_{1,1}$ defined in Section~\ref{sec:H}. To see
this, we set
$$
\t = \w_\tau/2\pi i \in H^1(E_\tau,\C).
$$
Then $\a$ and $\t$ comprise a framing of the pullback $\H_\h$ of $\H$ to $\h$,
which gives an isomorphism $\C^2 \times \h \to \H_\h$ via
\begin{equation}
\label{eqn:coords}
(u,v,\tau) \mapsto
\begin{pmatrix} (\a,\tau) & (\t,\tau)\end{pmatrix}
\begin{pmatrix} u \cr v\end{pmatrix}
\end{equation}
Here, $(\a,\tau)$ denotes $\a$ viewed as an element of $H_1(E_\tau)$. Likewise,
$(\t,\tau)$ denotes the element $\w_\tau/2\pi i$ of $H^1(E_\tau)$.

Since $\Lambda_{\gamma\tau} = (c\tau + d) \Lambda_\tau$, multiplication by
$(c\tau + d)$ induces an isomorphism
$$
E_\tau \to E_{\gamma\tau}.
$$
This induces the identification of the fibers of $\H_\h$ over $\tau$ and
$\gamma\tau$. For convenience, set $\a = (\a,\tau) \in H_1(E_\tau)$ and $\a' =
(\a,\gamma\tau) \in H_1(E_{\gamma\tau})$. Similarly with $\b$ and $\b'$, and
with $\t$ and $\t'$. Then
$$
2\pi i \t' = \w_{\gamma\tau} = (c\tau+d)^{-1}\w_\tau = 2\pi i(c\tau+d)^{-1}\t
$$
and
\begin{align*}
\begin{pmatrix}
\a' & \w_{\gamma\tau}
\end{pmatrix}
&=
(c\tau+d)^{-1}
\begin{pmatrix}
\a' & \b'
\end{pmatrix}
\begin{pmatrix}
c\tau+d & {a\tau+b} \cr 0 & -(c\tau + d)
\end{pmatrix}
\cr
&=
(c\tau+d)^{-1}
\begin{pmatrix}
\a & \b
\end{pmatrix}
\begin{pmatrix}
d & b \cr c & a
\end{pmatrix}
\begin{pmatrix}
c\tau+d & {a\tau+b} \cr 0 & -(c\tau + d)
\end{pmatrix}
\cr
&=
(c\tau+d)^{-1}
\begin{pmatrix}
\a & \b
\end{pmatrix}
\begin{pmatrix}
(c\tau+d)d & \tau \cr (c\tau+d)c & -1
\end{pmatrix}
\cr
&=
(c\tau+d)^{-1}
\begin{pmatrix}
\a & \w_\tau
\end{pmatrix}
\begin{pmatrix}
1 & \tau \cr 0 & -1
\end{pmatrix}
\begin{pmatrix}
(c\tau+d)d & \tau \cr (c\tau+d)c & -1
\end{pmatrix}
\cr
&=
\begin{pmatrix}
\a & \w_\tau
\end{pmatrix}
\begin{pmatrix}
c\tau+d & 0 \cr -c & (c\tau+d)^{-1}
\end{pmatrix}
\end{align*}
from which we conclude that
$$
\begin{pmatrix}
\a & \t
\end{pmatrix}
=
\begin{pmatrix}
\a' & \t'
\end{pmatrix}
M_\gamma(\tau).
$$
Equation (\ref{eqn:coords}) now implies that the bundle with factor of
automorphy $M_\gamma(\tau)$ is isomorphic to $\H$ as the following points
correspond:
$$
\big(\begin{pmatrix} u \cr v \end{pmatrix},\tau\big) \leftrightarrow
\begin{pmatrix} \a & \t\end{pmatrix}
\begin{pmatrix} u \cr v \end{pmatrix}
\leftrightarrow
\begin{pmatrix} \a' & \t'\end{pmatrix}
M_\gamma(\tau)
\begin{pmatrix} u \cr v \end{pmatrix}
\leftrightarrow
\big(M_\gamma(\tau)\begin{pmatrix} u \cr v \end{pmatrix},\gamma\tau\big)
$$
Note that $\t$ and $\a$ are both invariant under $\tau \mapsto \tau + 1$. It
follows that $\H$ is trivial over the $q$-disk.
\end{example}

Since the bundle $\H$ exists over $\M_{1,1}$, this computation gives a
conceptual proof that $M_\gamma(\tau)$ is a factor of automorphy.

\begin{remark}
\label{rem:meaning}
It is useful to keep in mind that
$$
\a \in H_1(E_\tau,\Z) \text{ and } \langle \a, \t \rangle =
-(2\pi i)^{-1} \in \Z(-1).
$$
Note that $\t$ spans a line sub-bundle of $\cH := \H\otimes_\C \O_{\M_{1,1}}$.
This line bundle is the {\em Hodge bundle} $F^1 \cH$ and is isomorphic to $\cL$.
The factor of automorphy of $\H$ implies that the quotient of $\cH$ by $F^1$ is
isomorphic to $\cL_{-1}$, so that we have an exact sequence
$$
0 \to \cL \to \cH \to \cL_{-1}\to 0.
$$
Later we will see that this splits, even over $\Mbar_{1,1}$. (Cf.\
Remark~\ref{rem:splitting} and the last paragraph of
Section~\ref{sec:trivializing}.)
\end{remark}

\section{Some Lie theory}

Let $\C\ll \t,\a \rr$ be the completion of the free associative algebra
generated by the indeterminants $\t$ and $\a$. It is a topological algebra.
Denote the closure of the free Lie algebra $\L(\t,\a)$ in $\C\ll \t,\a \rr$ by
$\p$. It is a topological Lie algebra.

Define a continuous action  $\C\ll \t,\a \rr \times \p \to \p$ of $\C\ll \t,\a
\rr$ on $\p$ by
$$
f(\t,\a) : x \mapsto f(\t,\a)\cdot x := f(\ad_\t,\ad_\a)(x).
$$
for all $x\in \p$.

For later use, we record the following fact:

\begin{proposition}
Suppose that $A : [a,b] \to \L(X_1,\dots,X_n)^\wedge$ is smooth.\footnote{That
is, each coefficient of the power series $A(t)$ is a smooth function of $t\in
[a,b]$.} If $X : [a,b] \to \C\ll X_1,\dots,X_n\rr$ satisfies the initial
value problem
$$
X' = AX,\quad X(0)=1,
$$
then $X(t)$ is group-like for all $t\in [a,b]$.
\end{proposition}

\begin{proof}
This follows from standard Lie theory. It can also be proved directly as
follows. Since the diagonal $\Delta$ is linear, since $\Delta$ is an algebra
homomorphism, and since $A$ is primitive, we have
$$
\big(\Delta X \big)' = \Delta\big(X') = \Delta(AX)
= (\Delta A)(\Delta X) = (A\otimes 1 + 1 \otimes A)\Delta X.
$$
On the other hand,
$$
(X\otimes X)' = X'\otimes X + X\otimes X' = (AX)\otimes X + X\otimes(AX)
= (A\otimes 1 + 1 \otimes A)(X\otimes X).
$$
Thus both $\Delta X$ and $X\otimes X$ satisfy the IVP
$$
Y' = (A\otimes 1 + 1 \otimes A)Y,\quad Y(0) = 1\otimes 1,
$$
where $Y : [a,b] \to \C\ll X_1,\dots,X_n\rr \otimes \C\ll X_1,\dots,X_n\rr$. It
follows that $\Delta X = X\otimes X$ for all $t$.
\end{proof}

\subsection{Two identities}

For later use we recall two standard identities. To avoid confusion, we shall
denote composition of endomorphisms $\phi$ and $\psi$ of $\p$ by $\phi\circ
\psi$.

Recall that if $V$ is a vector space and $u,\phi \in \End V$, then in $\End V$
we have
$$
\exp(\ad \phi)(u) = e^\phi \circ u \circ e^{-\phi}.
$$
Applying this in the case where $V = \p$, we see that for all $\delta \in \Der
\p$ and $\phi \in \C\ll \t,\a\rr$, 
$$
\exp(\phi)\cdot \delta = e^\phi \circ \delta \circ e^{-\phi}.
$$
In particular, if $\w$ is a 1-form on a manifold that takes values in $\p$, then
\begin{equation}
\label{eq:exp_ad}
e(-m\t)\circ \w \circ e(m\t) = e(-m\t)\cdot \w,
\end{equation}
where $e(u) := \exp(2\pi i u)$.

\begin{lemma}
\label{lem:exp_ad}
Suppose that $u\in \C\ll \t,\a \rr$. If $\delta$ is a continuous derivation of
$\C\ll \t, \a\rr$, then
$$
e^{-u} \delta(e^u) = \frac{1-\exp(-\ad_u)}{\ad_u} \delta(u)
\text{ and }
\delta(e^u)e^{-u} = \frac{\exp(\ad_u)-1}{\ad_u}\delta(u).
$$
\end{lemma}

\begin{proof}
The functions $e^{-su} \delta(e^{su})$ and $\frac{1-\exp(-s\ad_u)}{\ad_u}
\delta(u)$ both satisfy the differential equation
$$
X'(s)= \delta(u) - \ad_u(X).
$$
Since both functions vanish when $s=0$, they are equal for all $s\in \C$. In
particular, they are equal when $s=1$. This proves the first identity. The
second is proved similarly using the differential equation $Y'=
\delta(u)+\ad_u(Y)$.
\end{proof}

\section{Connections and Monodromy}

Suppose that $\G$ is a discrete group, $G$ is a Lie (or proalgebraic) group and
that $X$ is a topological space. Suppose that $\G$ acts on $X$ on the left.
(Think of this action as being discontinuous and fixed point free, but it does
not have to be.) Suppose that the action of $\G$ lifts to the trivial right
principal $G$-bundle $G\times X \to X$:
$$
\gamma : (g,x) \mapsto \big(M_\gamma(x)g,\gamma x\big)
$$
where $M_\gamma : X \to G$ is a factor of automorphy.

\subsection{Connections}

Denote the Lie algebra of $G$ by $\g$. Sections of the bundle $G\times X \to X$
will be identified with functions $X \to G$ in the obvious way. A Lie algebra
valued 1-form
$$
\w \in E^1(X)\otimes \g
$$
defines a connection on the trivial bundle $G\times X \to X$ by the formula
$$
\nabla f = df + \w f
$$
where $f$ is a locally defined function $X \to G$.

\begin{proposition}
The connection $\nabla$ is $\G$-invariant if and only if for all $\gamma \in \G$,
$$
\gamma^\ast \w = \Ad(M_\gamma)\w - dM_\gamma M_\gamma^{-1}.
$$
The connection $\nabla$ is flat if and only if $\w$ satisfies
$$
d\w + \frac{1}{2} [\w,\w] = 0. \qed
$$
\end{proposition}

\begin{example}
\label{ex:connection_H}
The sections $\a$ and $\b$ of the Hodge bundle $\H_\h$ over $\h$  are flat.
Since they give local framings of the associated vector bundle $\cH :=
\H\otimes_\C \O$, there is a flat connection on $\cH$, which is characterized by
the property that $\nabla \a = \nabla \b = 0$. Since $\t = \w_\tau/2\pi i =
(\tau \a - \b)/2\pi i$, we have
$$
2 \pi i \nabla \t = \nabla (\tau \a - \b) = \a d\tau.
$$
It follows that, in terms of the framing $\a$, $\t$ of $\cH$, the connection
is given by
$$
\nabla = d + (2\pi i)^{-1} \a \frac{\partial}{\partial \t}\otimes d\tau.
$$
\end{example}

\subsection{Parallel transport}

Every path $\alpha : [0,1] \to X$ has a horizontal lift $\alphatilde : [0,1] \to
G$ that starts at $1\in G$. In other words, the section
$$
t \mapsto \big(\alphatilde(t),\alpha(t)\big) \in G\times X
$$
is a flat section of the bundle that projects to $\alpha$ and begins at
$(1,\alpha(0))$.\footnote{Note that this does not require the connection to be
flat.}

The function $\alphatilde$ is the unique solution of the ODE
$$
d\alphatilde = -(\alpha^\ast\w)\alphatilde,\quad \alphatilde(0) = 1.
$$
Note that the uniqueness of solutions of ODEs implies that the horizontal lift
of $\alpha$ that begins at $g\in G$ is $t\mapsto \alphatilde(t)g$.

Denote the value of the lift $\alphatilde$ at $t=1$ by $T(\alpha)$. The function
$$
T : \alpha \mapsto T(\alpha)
$$
is called the (parallel) transport function associated to $\nabla$. When
$\nabla$ is flat, $T(\alpha)$ depends only on the homotopy class of $\alpha$
relative to its endpoints.

An immediate consequence of the uniqueness of solutions to IVPs:

\begin{lemma}
If $\alpha$ and $\beta$ are composable paths, then $T(\alpha\beta) =
T(\beta)T(\alpha)$. \qed
\end{lemma}

To make the transport multiplicative, we will work with $T(\alpha)^{-1}$. A
formula for the transport and the inverse transport can be given using Chen's
iterated integrals. First a basic fact from ODE.

\begin{lemma}
Suppose that $R$ is a topological algebra (such as $\C\ll \t,\a \rr$ or
$\gl_n(\C)$) and that $A: [a,b] \to R$ is a smooth function. A function
$X : [a,b] \to R^\times$ is a solution of the IVP
$$
X' = -AX,\quad X(0) = 1
$$
if and only $Y = X^{-1}(t)$ is a solution of the IVP
$$
Y' = YA,\quad Y(0) = 1.
$$
\end{lemma}

\begin{proof}
Suppose that $X$ satisfies $X'=-AX$ and $X(0)=I$. Then
$$
0 = X^{-1}\big(X X^{-1}\big)'
= X^{-1}\big(X(X^{-1})'\big) - X^{-1}\big(AXX^{-1}\big) = (X^{-1})' - X^{-1}A. 
$$
The opposite direction is proved similarly.
\end{proof}

Recall that if $\w_1,\dots,\w_r$ are 1-forms on a manifold $X$ taking values in
an associative algebra $A$, and if $\gamma$ is a piecewise smooth path in $X$
then one defines the iterated integral
\begin{equation}
\label{eqn:it_int}
\int_\gamma \w_1 \w_2 \dots \w_r =
\int_{0\le t_1 \le \cdots \le t_r \le 1}
f_1(t_1) f_2(t_2) \dots f_r(t_r) dt_1dt_2\dots dt_r
\end{equation}
where $\gamma^\ast \w_j = f_j(t) dt$. See \cite{chen,hain:bowdoin} for more
background.

\begin{corollary}[Transport Formula]
The inverse transport is given by
$$
T(\alpha)^{-1} = 1 + \int_\alpha \w + \int_\alpha \w\w
+ \int_\alpha \w\w\w + \cdots
$$
\end{corollary}

\begin{proof}
This follows from Chen's transport formula (cf.\ \cite{chen,hain:bowdoin}) and
the previous lemma.
\end{proof}

For future use, we record the following standard fact.

\begin{proposition}
If the connection $\nabla$ is $\G$-invariant, then for all paths $\alpha :
[0,1]\to X$ and all $\gamma\in \G$
$$
T(\gamma\circ \alpha)
= M_\gamma\big(\alpha(1)\big) T(\alpha) M_\gamma\big(\alpha(0)\big)^{-1}.
$$
\end{proposition}

\subsection{Monodromy}

Suppose that the connection $d+\w$ is $\G$-invariant and flat.  Our task in this
section is to explain how to compute the associated monodromy representation
from the transport function $T$ of $\w$ and the factor of automorphy $M$. 

By covering space theory, the choice of a point $x_o \in X$ determines a
surjective homomorphism
$$
\rho : \pi_1(\G\bs X,\xbar_o) \to \G
$$
whose kernel is $\pi_1(X,x_o)$, where $\xbar_o$ denotes the image of $x_o$ in
$\G\bs X$.

To each $\gamma \in \pi_1(\G\bs X)$, let $c_\gamma$ be its lift to a path in $X$
that begins at $x_o$. Note that its end point is $\rho(\gamma)\cdot x_o$ and
that the homotopy class of $c_\gamma$ depends only upon $\gamma$.

\begin{lemma}
If $\gamma,\mu\in \pi_1(\G\bs X,\xbar_o)$, then $c_{\gamma\mu} = c_\gamma \cdot
\big(\rho(\gamma)\circ c_\mu\big)$.
\end{lemma}

Here $\cdot$ denotes path multiplication and $\circ$ denotes composition. The
proof is best given by the picture Figure~\ref{fig:cocycle}.
\begin{figure}[!ht]
\epsfig{file=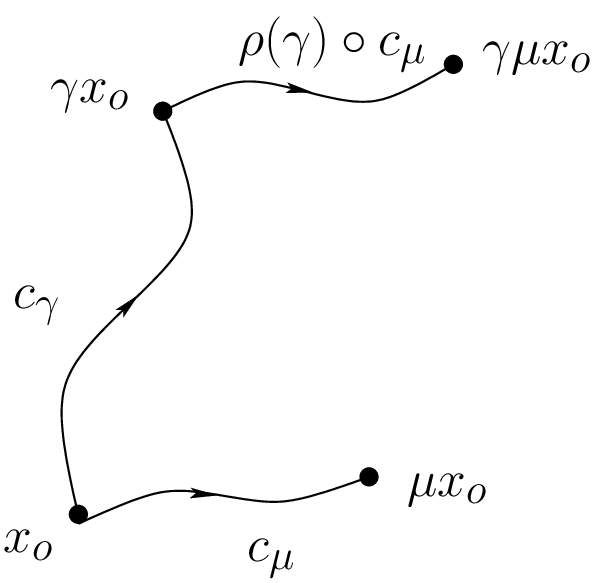, width=1.5in}
\caption{The cocycle relation $c_{\gamma\mu} = c_\gamma \cdot
\big(\rho(\gamma)\circ c_\mu\big)$}
\label{fig:cocycle}
\end{figure}

To obtain a homomorphism (instead of an anti-homomorphism), we need to take
inverses. Define
$$
\Theta_{x_o} : \pi_1(\G\bs X,\xbar_o) \to G
$$
by
$$
\Theta_{x_o}(\gamma) = T(c_\gamma)^{-1}M_{\rho(\gamma)}(x_o).
$$
Note that $\Theta_{x_o}(\gamma)^{-1}$ is the element of the fiber $G$ over $x_o$
that is identified with the point $T(c_\gamma)$ in the fiber $G$ over
$\rho(\gamma)\cdot x_o$. It is thus the result of parallel transporting $1\in G$
about the loop $\gamma$.

\begin{proposition}
The monodromy representation $\pi_1(\G\bs X,\xbar_o) \to G$ of the flat bundle
$\G\bs(G\times X) \to \G\bs X$ with respect to the identification above is
$$
\Theta_{x_o} : \pi_1(\G\bs X,\xbar_o) \to G.
$$ 
\end{proposition}

\begin{proof}
Just trace through the identifications. But to reassure the reader, we show that
$\Theta_{x_o}$ is a group homomorphism. (We'll drop $\rho$ and the $x_o$ below.)
If $\gamma,\mu\in \pi_1(\G\bs X,\xbar_o)$, then
\begin{align*}
\Theta(\gamma\mu) &= T(c_{\gamma\mu})^{-1} M_{\gamma\mu}(x_o) \cr
&= T(c_\gamma\cdot(\gamma\circ c_\mu))^{-1}M_{\gamma\mu}(x_o) \cr
&= T(c_\gamma)^{-1}T(\gamma\circ c_\mu)^{-1}M_\gamma(\mu\cdot x_o)M_\mu(x_o)\cr
&= T(c_\gamma)^{-1}M_\gamma(x_o)T(c_\mu)^{-1}
M_\gamma(\mu\cdot x_o)^{-1}M_\gamma(\mu\cdot x_o)M_\mu(x_o) \cr
&= \Theta(\gamma)\Theta(\mu).
\end{align*}
\end{proof}

Combining this with the transport formula above, we obtain a formula for the
monodromy in terms of $\w$ and the factor of automorphy.

\begin{corollary}
\label{cor:chen-dyson}
For all $x\in X$ and $\gamma \in \pi_1(\G\bs X,\xbar)$,
$$
\Theta_x(\gamma)
= \bigg(1 + \int_{c_\gamma}\w + \int_{c_\gamma}\w\w
+ \int_{c_\gamma} \w\w\w + \cdots\bigg)M_\gamma(x). \qed
$$
\end{corollary}

\part{The Universal Elliptic KZB Connection}

\section{The Bundle $\bP$ over $\E'$}

Before we define the universal elliptic KZB connection, we need to define the
bundle $\bP$ over $\E'$ on which it lives.

\subsection{The flat bundle $\bP^\top$}
\label{sec:bPtop}

The bundle $\bP$ with the KZB connection will be the de Rham realization of a
topological local system $\bP^\top$. To provide context, we first construct it.

Denote by $Y$ the universal covering space of $\E'$. This is also the universal
covering space of $\E_\h' = \big(\C\times\h) - \Lambda_\h$. Choose a base point
$[E_o,x_o]$ of $\E'$ and a lift $y_o$ of it to $Y$. This determines an
isomorphism of $\Aut(Y/\E')$ with $\pi_1(\E',[E_o,x_o])$.

Denote the unipotent completion of $\pi_1(E_o',x_o)$ over $\C$ by $\cP_o$. The
natural action
$$
\pi_1(\E',[E_o,x_o]) \times \pi_1(E_o',x_o) \to \pi_1(E_o',x_o),
\quad (g,\gamma)\mapsto g\gamma g^{-1}
$$
determines a left action of $\pi_1(\E',[E_o,x_o])$ on $\cP_o$. We can
therefore form the quotient
$$
\pi_1(\E',[E_o,x_o])\bs \big( \cP_o\times Y \big)
$$
by the diagonal $\pi_1(\E',[E_o,x_o])$-action. This is a flat right principal
$\cP_o$-bundle which we shall denote by $\bP^\top \to \E'$. Its fiber over
$[E,x]$ is naturally isomorphic to the unipotent completion of $\pi_1(E',x)$.

Since the Lie algebra $\p_o$ of $\cP_o$ can be viewed as a group with
multiplication defined by the Baker-Campbell-Hausdorff formula, we can (and
will) view $\bP^\top$ as a local system of Lie algebras. (Cf.\ the comment
following Prop.~\ref{prop:exp}.)

\subsection{The bundle $\bP$}
\label{sec:bP}

Here we construct a bundle $\bP$ over $\E$ on which the universal elliptic KZB
connection lives. Its fiber over each point of $\E'$ is the Lie algebra
$$
\p := \L(\t,\a)^\wedge.
$$
Denote the corresponding group $\exp\p$ by $\cP$. It is prounipotent. The
universal elliptic KZB connection on it is constructed in
Section~\ref{sec:connection}. It is flat. In Section~\ref{sec:rigidity} we will
prove that it is isomorphic to the flat bundle $\bP$.

The bundle $\bP$ will be constructed as the quotient of $\p\times\C\times\h$ by
a lift of the action of $\SL_2(\Z)\ltimes \Z^2$ on $\C\times \h$ to
$\p\times\C\times\h$.

The (completed) universal enveloping algebra of $\p$ is the power series algebra
$\C\ll \t,\a \rr$. The adjoint action defines the ring homomorphism
$$
\C\ll \t,\a \rr \to \End \p
$$
that takes $f(\t,\a)$ to $f(\ad_\t,\ad_\a)\in \End\p$.  This restricts to a
homomorphism $\C\ll \t,\a \rr^\times \to \Aut \p$.

We use the notation of Section~\ref{sec:automorphy}. Take $G = \G :=
\SL_2(\Z)\ltimes \Z^2$, $X=\C\times \h$, and $V = \p$. Note that $\G$ acts on
$\p\times X$ via the projection $\G \to \SL_2(\Z)$ using the factor of
automorphy $M_\gamma(\tau)$ defined in (\ref{eqn:def}).

The bundle $\bP$ is defined using factors of automorphy
$\Mtilde_\gamma(\xi,\tau)$ which live in the group $\SL_2(\R) \ltimes \C\ll
\t,\a \rr^\times$, where $\SL_2(\R)$ acts on $\C\ll \t,\a \rr$ via its left
action on the generators via the factor of automorphy $M_\gamma(\tau)$  defined
in Example~\ref{ex:hodge}. Specifically, $M_\gamma(\tau)$ is defined by
\begin{equation}
\label{eqn:def}
M_\gamma(\tau) : 
\begin{cases}
\a &\mapsto (c\tau + d)^{-1}\a + 2\pi i c\t \cr
\t &\mapsto (c\tau+d)\t.
\end{cases}
\end{equation}
The general factor of automorphy is defined by
\begin{equation}
\label{eqn:M_tilde}
\Mtilde_\gamma(\xi,\tau) =
\begin{cases}
M_\gamma(\tau)\circ e\big(\frac{c\xi \t}{c\tau + d} \big) &
\gamma \in \SL_2(\Z);\cr
e(-m\t) & (m,n) \in \Z^2.
\end{cases}
\end{equation}
Here $e(u) := \exp(2\pi iu)$. This is a factor of automorphy for $\G$ as
$$
e(c\xi \t - m\t)\circ M_\gamma(\tau) =
\Mtilde_\gamma\big(\xi + (m,n)\gamma(\tau,1)^T,\tau\big)\circ e(-m\t),
$$
where $\gamma \in \SL_2(\Z)$ and $(m,n)\in\Z^2$.\footnote{Denote left
multiplication by $\phi \in \C\ll \t,\a\rr$ by $L_\phi$. For $M \in \Aut H$, we
have $M\circ L_{M^{-1}\phi} = L_\phi\circ M$. In particular,
$M_\gamma(\tau)\circ L_{\t/(c\tau+d)} = L_\t\circ M_\gamma(\tau)$.}

\begin{proposition}
This is a well-defined factor of automorphy.
\end{proposition}

\begin{proof}
The first task is to show that $\Mtilde$ is well-defined on $\Gamma\times
\C\times \h$ --- that is, it is compatible with the relation in $\G =
\SL_2(\Z)\ltimes \Z^2$. This relation is
$$
(m,n)\circ \gamma = \gamma \circ (m,n)(\cdot \gamma),
$$
where $\gamma \in \SL_2(\Z)$, $\circ$ denotes composition in $\G$ and $\cdot$
denotes the right action of $\SL_2(\Z)$ on $\Z^2$. If
$$
\gamma = \begin{pmatrix}a & b \cr c & d\end{pmatrix}
$$
then we have to show that
$$
e(-m\t)\circ\Mtilde_\gamma(\xi,\tau) = 
\Mtilde_\gamma(\xi + (m,n)\gamma (\tau,1)^T,\tau)\circ e\big(-(ma+nc)\t\big).
$$
Since
$$
M_\gamma(\tau) \circ e(\phi) = e(M_\gamma(\tau)\cdot\phi) \circ M_\gamma(\tau)
$$
for all $\phi \in \C\ll \t,\a \rr$, and since $M_\gamma(\tau) : \t\mapsto (c\tau
+ d)\t$ (see above), we have
$$
e(-m\t)\circ M_\gamma(\tau)
=M_\gamma(\tau)\circ e\bigg(\frac{-m\t}{c\tau+d}\bigg).
$$
Thus, the left-hand side expands to
\begin{align*}
e(-m\t)\circ \Mtilde_\gamma(\xi,\tau)
&= e(-m\t)\circ M_\gamma(\tau)\circ e\bigg(\frac{c\xi\t}{c\tau+d}\bigg) \cr
&= M_\gamma(\tau)\circ e\bigg(\frac{c\xi \t - m\t}{c\tau + d}\bigg)
\end{align*}

The right-hand side expands to
\begin{align*}
&\phantom{=} \Mtilde_\gamma(\xi + (ma+nc)\tau
+ (mb+nd),\tau)\circ e\big(-(ma+nc)\t\big) \cr
&= M_\gamma(\tau)\circ
e\bigg(\frac{c\big(\xi+(ma+nc)\tau + (mb+nd)\big)\t}{c\tau + d}\bigg)
\circ e\big(-(ma+nc)\t\big)\cr
&= M_\gamma(\tau)\circ e\bigg(\frac{c\xi \t - m\t}{c\tau + d}\bigg),
\end{align*}
which equals the left-hand side. It follows that $\Mtilde$ is a well-defined
function on $\G \times \C \times \h$.

Since $\Mtilde$ defines a homomorphism $\Z^2 \to \Q\ll \t \rr^\times$, to
complete the proof we need only check that the restriction of $\Mtilde$ to
$\SL_2(\Z)\times \C\times\h$ is a factor of automorphy. We will use the fact
that $M_\gamma(\tau)$ is a factor of automorphy.

Let
$$
\gamma_1 = \begin{pmatrix} a & b \cr c & d \end{pmatrix},
\quad
\gamma_2 = \begin{pmatrix} p & q \cr r & s \end{pmatrix}
\text{ and }
\gamma_1\gamma_2 = \begin{pmatrix} e & f \cr g & h \end{pmatrix}.
$$
Set $(\xi',\tau') = \gamma_2(\xi,\tau) = (\xi/(r\tau+s),\gamma_2 \tau)$.
Then
\begin{align*}
\Mtilde_{\gamma_1}(\xi',\tau') \Mtilde_{\gamma_2}(\xi,\tau)
&= M_{\gamma_1}(\gamma_2 \tau)\circ e\bigg(\frac{c\xi'\t}{c\tau' + d}\bigg)
\circ M_{\gamma_2}(\tau)\circ e\bigg(\frac{r\xi \t}{r\tau + s}\bigg) \cr
&= M_{\gamma_1}(\gamma_2 \tau) M_{\gamma_2}(\tau) \circ
e\bigg(\frac{c\xi'\t}{(c\tau' + d)(r\tau+s)}\bigg)
\circ e\bigg(\frac{r\xi \t}{r\tau + s}\bigg) \cr
&= M_{\gamma_1\gamma_2}(\tau) \circ
e\bigg(
\frac{c\xi\t}{(r\tau+s)\big(c(p\tau+q)+d(r\tau+s)\big)}+\frac{r\xi\t}{r\tau+s}
\bigg) \cr
&= M_{\gamma_1\gamma_2}(\tau) \circ
e\bigg(\frac{\big(c + r(g\tau+h)\big)\xi\t}{(r\tau+s)(g\tau+h)}\bigg)\cr
&= M_{\gamma_1\gamma_2}(\tau) \circ
e\bigg(\frac{g \xi\t}{g\tau + h}\bigg) \cr
&= \Mtilde_{\gamma_1\gamma_2}(\xi,\tau).
\end{align*}
\end{proof}

\begin{remark}
\label{rem:abelianization}
The bundle $\bP$ is a bundle of free Lie algebras. Its quotient by the
commutator subalgebra of each fiber is the bundle over  $\h$ with framing
$\t,\a$ and factor of automorphy $M_\gamma(\tau)$. So it is isomorphic to $\H$
by Example~\ref{ex:hodge}, as it should be.
\end{remark}

\section{Eisenstein Series and Bernoulli Numbers}

Define the Bernoulli numbers $B_n$ by
$$
\frac{x}{e^x-1} = \sum_{n=0}^\infty B_n \frac{x^n}{n!}.
$$
Recall that $B_0 = 1$, $B_1 = -1/2$ and that $B_{2k+1} = 0$ when $k>0$.

There are several ways to normalize Eisenstein series $G_{2k} : \h \to \C$. We
will use the normalization used by Zagier \cite{zagier}:
$$
G_{2k}(\tau) = \frac{1}{2}\frac{(2k-1)!}{(2\pi i)^{2k}}
\sum_{\substack{\lambda \in \Z\oplus \Z\tau\cr\lambda \neq 0}}
\frac{1}{\lambda^{2k}}
=
-\frac{B_{2k}}{4k} + \sum_{n=1}^\infty \sigma_{2k-1}(n)q^n,
$$
(properly summed when $k=1$), where $q=e(\tau)$ and $\sigma_k(n)=\sum_{d|n}d^k$.
In particular
$$
G_{2k}|_{q=0} = -\frac{B_{2k}}{4k} = \frac{(2k-1)!}{(2\pi i)^{2k}}\,\zeta(2k).
$$
When $k>1$, $G_{2k}$ is a modular form for $\SL_2(\Z)$ of weight $2k$:
$$
G_k(\gamma \tau) = (c\tau + d)^k G_k(\tau),\qquad \gamma \in \SL_2(\Z).
$$
And $G_2$ satisfies
$$
G_2(\gamma\tau) = (c\tau+d)^2 G_2(\tau) + ic(c\tau+d)/4\pi.
$$
(Cf.\ \cite[p.~457]{zagier}, bottom of page, and \cite[p.~459]{zagier},
near bottom of page.)

The role of $G_2(\tau)$ in this work should be clarified by the following
result, which follows from the transformation law for $G_2$ above.

\begin{lemma}
\label{lem:G2}
If $\SL_2(\Z)$ acts on $\C\times \h$ by $\gamma : (\xi,\tau) \mapsto
\big(\xi/(c\tau+d),\gamma\tau\big)$, then the form
$$
\frac{d\xi}{\xi} - 2\cdot 2\pi i\, G_2(\tau)\, d\tau
$$
is $\SL_2(\Z)$-invariant.
\end{lemma} 

This 1-form represents a generator of $H^1(\cL_{-1}',\Z(1)) \cong \Z$.

\subsection{Some useful identities}
\label{sec:b_identities}

The following well-known identities are used later in the paper. Since
$$
\frac{1}{2} \coth\big(u/2\big)
=  \frac{1}{2}\frac{e^{u/2}+e^{-u/2}}{e^{u/2}-e^{-u/2}}
= \frac{1}{2} \frac{e^u+1}{e^u-1}
= \frac{1}{2} + \frac{1}{u} \frac{u}{e^u-1}
=\sum_{m=0}^\infty \frac{B_{2m}}{(2m)!}u^{2m-1} ,
$$
we have
\begin{equation}
\label{eqn:coth}
\frac{1}{u} - \frac{u/4}{\sinh^2(u/2)}
= \frac{1}{u} + \frac{u}{2}\frac{d}{du}\coth\big(u/2)
= \sum_{m=1}^\infty (2m-1)\frac{B_{2m}}{(2m)!}u^{2m-1}.
\end{equation}
Rearranging gives the useful alternative form
\begin{equation}
\label{eqn:coth_alt}
\sum_{m=0}^\infty (2m-1)\frac{B_{2m}}{(2m)!}\, u^{2m-1}
= -\frac{u/4}{\sinh^2(u/2)}.
\end{equation}

\section{The Jacobi Form $F(\xi,\eta,\tau)$}

There are two versions of the function $F(u,v,\tau)$, one used by Levin-Racinet
\cite{levin-racinet}, the other by Zagier
\cite{zagier}.\footnote{Calaque-Enriquez-Etingof \cite{cee} do not explicitly
use the Jacobi form  $F$. However, their connection is expressed in terms of the
function $k(z,x|\tau)$, which is $F^\zag(z,x,\tau) - 1/x$. See their
Section~1.2.}  Denote them by $F(\xi,\eta,\tau)$ and $F^\zag(u,v,\tau)$,
respectively. Zagier's function is defined by
$$
F^\zag(u,v,\tau) :=
\frac{\theta'(0,\tau)\theta(u+v,\tau)}{\theta(u,\tau)\theta(v,\tau)},
$$
where $\theta$ is the classical theta function
$$
\theta(u,\tau) := \sum_{n\in \Z} (-1)^n q^{\shalf(n+\shalf)^2}e^{(n+\shalf)u},\quad
q = e(\tau)
$$
and $\theta'$ is its derivative with respect to $u$.

Their periodicity properties imply that $u = 2 \pi i \xi$, $v = 2 \pi i \eta$.
Since
$$
F(\xi,\eta,\tau) = \frac{1}{\xi} + \frac{1}{\eta}
\bmod \text{ holomorphic functions }
$$
and
$$
F^\zag(u,v,\tau) = \frac{1}{u} + \frac{1}{v}
\bmod \text{ holomorphic functions }
$$
near the origin, it follows that
$$
F(\xi,\eta,\tau) = 2\pi i F^\zag(2\pi i \xi, 2\pi i \eta,\tau).
$$
It satisfies the {\em symmetry condition}
$$
F(\xi,\eta,\tau) = F(\eta,\xi,\tau) = - F(-\xi,-\eta,\tau).
$$

\subsection{Expansions}

We use the formulas in \cite{zagier}, but write them using $F$ in place of
$F^\zag$.\footnote{Note the conflict in notation: Zagier sets $\xi = \exp u$ and
$\eta = \exp v$, which conflicts with the variables $(\xi,\eta)$ used by
Levin-Racinet: $2\pi i(\xi,\eta) = (u,v)$. See \cite[p.~455]{zagier}.}

Set $q = \exp(2\pi i \tau)$. Then
\begin{equation}
\label{eqn:qexp}
F(\xi,\eta,\tau) = \pi i\big[\coth(\pi i \xi) + \coth(\pi i \eta)\big]
+ 4 \pi \sum_{n=1}^\infty \bigg(
\sum_{d|n}\sin \big[2 \pi \big(\frac{n}{d}\xi + d\eta\big)\big]
\bigg)q^n.
\end{equation}

\begin{equation}
\label{eqn:def_F}
F(\xi,\eta,\tau) =
\frac{1}{\xi} + \frac{1}{\eta} - 2 \sum_{r,s = 0}^\infty
(2\pi i)^{1+\max\{r,s\}}
\bigg(\frac{\partial}{\partial \tau}\bigg)^{\min\{r,s\}} G_{|r-s|+1}(\tau)\,
\frac{\xi^r}{r!}\frac{\eta^s}{s!}.
\end{equation}

\subsection{Derivatives}

Differentiating these with respect to $\eta$ yields:
\begin{align}
\label{eqn:derivative}
\frac{1}{\eta} + \eta \frac{\partial F}{\partial\eta}(\xi,\eta,\tau) &=
 - 2 \sum_{\substack{r\ge 0\cr s\ge 1}}
(2\pi i)^{1+\max\{r,s\}}
\bigg(\frac{\partial}{\partial \tau}\bigg)^{\min\{r,s\}} G_{|r-s|+1}(\tau)\,
\frac{\xi^r}{r!}\frac{\eta^s}{(s-1)!}\cr
&=\bigg(\frac{1}{\eta} - \frac{(\pi i)^2\eta}{\sinh^2(\pi i \eta)}\bigg)
+ 8\pi^2\eta 	
\sum_{n=1}^\infty \bigg(
\sum_{d|n}d\cos \big[2 \pi \big(\frac{n}{d}\xi + d\eta\big)\big]
\bigg)q^n.
\end{align}

Comparing the result of differentiating this with respect to $\xi$ and
(\ref{eqn:def_F}) with respect to $\tau$, we obtain the {\em heat equation}:
$$
2 \pi i \frac{\partial F}{\partial \tau}(\xi,\eta,\tau)
= \frac{\partial^2 F}{\partial\xi\partial\eta}(\xi,\eta,\tau).
$$

\subsection{Elliptic and modularity properties}

The {\em elliptic property}, \cite[p.~456]{zagier} is:
\begin{equation}
\label{eqn:w}
F(\xi+m\tau + n,\eta,\tau) =
e(-m\eta)F(\xi,\eta,\tau)\qquad (m,n)\in \Z^2.
\end{equation}
Here, as previously, $e(x) = \exp(2\pi i x)$. Zagier states a more general form
of this, which follows from this one using the symmetry property of
$F(\xi,\eta,\tau)$.

The {\em modularity property} is:
\begin{equation}
\label{eqn:q}
F(\xi/(c\tau + d),\eta/(c\tau + d),\gamma\tau)
= (c\tau+d)e\big(c\xi\eta/(c\tau+d)\big) F(\xi,\eta,\tau).
\end{equation}
In particular
$$
F(\xi,\eta,\tau+1) = F(\xi,\eta,\tau) = F(\xi+1,\eta,\tau).
$$

\begin{proposition}
\label{prop:poincare}
The function $F$ induces a meromorphic section of the line bundle $\mathcal{N}
\to \cEbar\times_{\Mbar_{1,1}}\cEbar$ that was constructed in
Example~\ref{ex:linebdle}. The divisor of the section is $[\G_\iota] -[0_1] -
[0_2]$, where $\G_\iota$ is the graph in $E\times E$ of the involution $\iota$
that takes a point of $E$ to its inverse.
\end{proposition}

\begin{proof}
Since $F(\xi,\eta,\tau)$ is meromorphic on $\C^2$ for each $\tau \in \h$, $\div
F$ has no vertical components over $\M_{1,1}$. The polar locus of $F$ over
$\M_{1,1}$ contains $[0_1]+[0_2]$ with multiplicity one on every fiber over
$\M_{1,1}$. The zero divisor of $F$ over $\M_{1,1}$ contains $[\G_\iota]$ with
multiplicity one on the generic fiber over $\M_{1,1}$. Since the class of $\div
F$ in $H^2(E_\tau^2)$ is constant and since $\div F$ has no vertical components,
it suffices to show that the class of $\div F$ is exactly $[\G_\iota] -[0_1] -
[0_2]$ on an open set of fibers and also over $q=0$.

Identity (\ref{eqn:qexp}) implies that
$$
\frac{1}{\pi i}F(\xi,\eta)|_{q=0} = \frac{w+1}{w-1} + \frac{u+1}{u-1}
$$
where $w = \exp(2\pi i\xi)$ and $u=\exp(2\pi i\eta)$. The coordinates on the
normalization of $E_0\times E_0$ are $(w,u)$. The identity sections are $w=1$
and $u=1$. The involution is given by $w = 1/w$. It is easily checked that
$$
\frac{w+1}{w-1} + \frac{u+1}{u-1} = 0
$$
implies that $wu=1$. It follows that the restriction of $\div F$ to $E_0\times
E_0$ is $[\G_\iota]-[0_1]-[0_2]$. But this implies that the divisor of $F$ is
$[\G_\iota]-[0_1]-[0_2]$ on all nearby fibers. The result follows.
\end{proof}

This result can also be proved using the formula for $F$ in terms of theta
functions.

\subsection{The Weierstrass $\wp$ function}
\label{sec:wp}

Recall that
$$
\wp(z,\tau) = \frac{1}{z^2} +
\sum_{\substack{\lambda \in \Z\oplus \Z\tau\cr \lambda \neq 0}}
\bigg[\frac{1}{(z-\lambda)^2} - \frac{1}{\lambda^2} \bigg].
$$
The next result follows from the standard identity
\begin{align*}
\wp(z,\tau)
&= \frac{1}{z^2} + \sum_{m=2}^\infty (2m-1)\bigg(
\sum_{\substack{\lambda \in \Z\oplus \Z\tau\cr\lambda \neq 0}}
\frac{1}{\lambda^{2m}}\bigg)z^{2m-2} \cr
&= \frac{1}{z^2} +
\sum_{m=2}^\infty \frac{2(2\pi i)^{2m}}{(2m-2)!}G_{2m}(\tau) z^{2m-2} \cr
&= \frac{1}{z^2} +
\sum_{m=1}^\infty \frac{2(2\pi i)^{2m+2}}{(2m)!}G_{2m+2}(\tau) z^{2m}.
\end{align*}

\begin{lemma}
\label{lem:wp}
Suppose that $x,y$ are commuting indeterminants. Then
\begin{multline*}
\frac{1}{2}\frac{xy}{x+y}
\bigg(
\big(\wp(x,\tau) - \frac{1}{x^2}\big)
-\big(\wp(y,\tau) - \frac{1}{y^2}\big)
\bigg)
\cr
=
\sum_{m\ge 1}\frac{(2\pi i)^{2m+2}}{(2m)!}G_{2m+2}(\tau)
\sum_{\substack{j+k=2m+1\cr j,k>0}} (-1)^{j}x^jy^k
\end{multline*}
in the ring $\O(\h)[[x,y]]$ of formal power series with coefficients in
$\O(\h)$. \qed
\end{lemma}

\subsection{The addition formula}

The following identity is used in the proof of the integrability of the elliptic
KZB connection.

\begin{proposition}[Addition Formula]
\label{prop:addition}
\begin{multline*}
F(\xi,\eta_1,\tau)\frac{\partial F}{\partial \eta}(\xi,\eta_2,\tau)
- F(\xi,\eta_2,\tau)\frac{\partial F}{\partial \eta}(\xi,\eta_1,\tau) \cr
= F(\xi,\eta_1+\eta_2,\tau)\big(\wp(\eta_1,\tau) - \wp(\eta_2,\tau)\big).
\end{multline*}
\end{proposition}

\section{The Universal Elliptic KZB Connection}
\label{sec:connection}

This is a $\G$-invariant flat connection constructed by Calaque, Enriquez and
Etingof \cite{cee} and by Levin and Racinet \cite{levin-racinet} on the bundle
\begin{equation}
\label{eqn:bdle}
\p \times \C \times \h \to \C \times \h.
\end{equation}
So it descends to a flat connection on the bundle $\bP \to \E'$. It has regular
singularities along the universal lattice:
$$
\Lambda_\h := \{(m\tau+n,\tau) \in \C\times \h\}.
$$
It therefore descends to a meromorphic connection on the bundle $\bP \to \E$
with regular singularities along the zero-section. In Section~\ref{sec:tate} we
show that the natural extension of this connection to the $q$-disk has regular
singularities along the nodal cubic.

In this section we will follow Levin-Racinet (with modifications).

\subsection{Derivations}

We have already explained the algebra homomorphism
$$
\C\ll \t,\a \rr \to
\End\p,\qquad f(\t,\a) \mapsto \{x \mapsto f(\t,\a)\cdot x\},
$$
where $f(\t,\a)\cdot x := f(\ad_\t,\ad_\a)(x)$. We will view $\p$ as a Lie
subalgebra of $\Der \p$ via the adjoint action $\ad : \p \to \Der \p$, which is
an inclusion as $\p$ has trivial center. Every derivation $\delta$ can be
written uniquely in the form
$$
\delta =
\delta(\a)\frac{\partial}{\partial \a} + \delta(\t)\frac{\partial}{\partial \t}.
$$
Consequently, there is a linear isomorphism
$$
\p \frac{\partial}{\partial \t} \oplus \p \frac{\partial}{\partial \a}
\overset{\simeq}{\longrightarrow} \Der \p.
$$

\subsection{The formula}

The connection is defined by a 1-form
$$
\w \in \Omega^1(\C\times\h,\log \Lambda)\hat{\otimes} \End \p.
$$
via the formula
$$
\nabla f = df + \w f
$$
where $f : \C\times \h \to \p$ is a (locally defined) section of
(\ref{eqn:bdle}). Specifically,
$$
\w = \frac{1}{2\pi i} d\tau\otimes \a\frac{\partial}{\partial \t} + \psi + \nu
$$
where
$$
\psi =
\sum_{m\ge 1}\bigg(\frac{(2\pi i)^{2m+1}}{(2m)!}G_{2m+2}(\tau)d\tau \otimes
\sum_{\substack{j+k=2m+1\cr j,k > 0}} (-1)^j[\ad_\t^j(\a),\ad_\t^k(\a)]
\frac{\partial}{\partial \a}\bigg)
$$
and
$$
\nu = \t F(\xi,\t,\tau)\cdot \a\,d\xi + \frac{1}{2\pi i}\bigg(
\frac{1}{\t} + \t\frac{\partial F}{\partial \t}(\xi,\t,\tau)
\bigg)\cdot \a\,d\tau.
$$
Note that each term takes values in $\Der \p$. Later we will show that its
restriction to a punctured first order neighbourhood of the identity section
takes values in a smaller subalgebra.

\begin{remark}
\label{rem:abelianization2}
Each term of the lower central series of $\bP$ is preserved by the connection.
The connection thus induces a connection on the bundle of abelianizations, which
is isomorphic to $\H$ (cf.\ Remark~\ref{rem:abelianization}).
Example~\ref{ex:connection_H} implies that this induced connection on $\H$ is
the natural connection.
\end{remark}

\subsection{Modularity}

Recall that $\G = \SL_2(\Z)\ltimes \Z^2$. In this section we shall prove:

\begin{proposition}
The universal elliptic KZB connection is $\SL_2(\Z)\ltimes \Z^2$-invariant. That
is,
$$
\gamma^\ast \w =
\Ad\big(\Mtilde_\gamma\big)\cdot\w - d\Mtilde_\gamma \Mtilde_\gamma^{-1}
$$
for all $\gamma \in \SL_2(\Z)\ltimes \Z^2$.
\end{proposition}

It suffices to check that the connection is invariant under $\Z^2$ and
$\SL_2(\Z)$. These are proved in the two following subsections.

\subsubsection{Ellipticity: invariance under $\Z^2$}

\begin{lemma}
For all $\delta \in \Der \p$ we have
$$
e(-m\t)\cdot \delta
= \delta + \frac{1-e(-m\ad_\t)}{\ad_\t}\delta(\t).
$$
\end{lemma}

\begin{proof}
For all $x \in \p$
\begin{align*}
\big(e(-m\t)\cdot \delta\big)(x)
&= e(-m\t)\delta \big(e(m\t)(x)\big)  &\text{(Equation~\ref{eq:exp_ad})} \cr
&= e(-m\t)\delta(e(m\t))(x) + e(-m\t)e(m\t)\delta(x) \cr
&= \delta(x) + \frac{1-e(-m\ad_\t)}{2\pi i m\ad_\t} \delta(2\pi i m\t)\cdot x
&\text{(Lemma~\ref{lem:exp_ad})}
\cr
&= \delta(x) + \frac{1-e(-m\ad_\t)}{\ad_\t} \delta(\t)\cdot x \cr
&= \bigg(\delta + \frac{1-e(-m\ad_\t)}{\ad_\t}\delta(\t)\bigg)(x).
\end{align*}
\end{proof}

\begin{corollary}
If $(m,n) \in \Z^2$, then
$$
(m,n)^\ast \bigg(\frac{1}{2\pi i}\a\frac{\partial}{\partial\t}d\tau\bigg)
- e(-m\t)\cdot \bigg(\frac{1}{2\pi i}\a\frac{\partial}{\partial\t}d\tau\bigg)
= -\frac{1}{2\pi i}\frac{1-e(-m\t)}{\t}(\a) d\tau.
$$
\end{corollary}

\begin{proof}
Apply the previous lemma with $\delta = \a\frac{\partial}{\partial\t}$.
\end{proof}

\begin{corollary}
If $a,b\in \N$, then
$$
e(-m\t)\cdot [\t^a\cdot \a,\t^b \cdot \a]\frac{\partial}{\partial \a}
= [\t^a\cdot \a,\t^b \cdot \a]\frac{\partial}{\partial \a}.
$$
\end{corollary}

\begin{proof}
This follows directly from the previous lemma as the derivation
$$
[\t^a\cdot \a,\t^b \cdot \a]\frac{\partial}{\partial \a}
$$
annihilates $\t$.
\end{proof}

\begin{corollary}
If $(m,n)\in \Z^2$, then $(m,n)^\ast \psi = e(-m\t)\cdot \psi = \psi$. \qed

\end{corollary}

\begin{lemma}
For all $(m,n)\in\Z^2$
$$
(m,n)^\ast \nu - e(-m\t)\cdot \nu
= \frac{1}{2\pi i}\frac{1-e(-m\ad_\t)}{\ad_\t}(\a)d\tau.
$$
\end{lemma}

\begin{proof}
Write $\nu = \nu_1 + \nu_2$, where
$$
\nu_1 = \t F(\xi,\t,\tau)\cdot \a\,d\xi
\text{ and }
\nu_2 = \frac{1}{2\pi i}\bigg(
\frac{1}{\t} + \t\frac{\partial F}{\partial\t}(\xi,\t,\tau)
\bigg)\cdot \a\,d\tau.
$$
Then
\begin{align*}
&\phantom{=} (m,n)^\ast\nu_1 - e(-m\t)\cdot \nu_1 \cr
&= \t F(\xi+m\tau+n,\t)\cdot\a\,d(\xi+m\tau+n)-\t e(-mt)F(\xi,\t)\cdot\a
\,d\xi \cr
&= \t e(-m\t)F(\xi,\t)\cdot\a\,(d\xi+md\tau)
-\t e(-m\t)F(\xi,\t)\cdot\a\,d\xi \cr
&= m\t e(-m\t)F(\xi,\t)\cdot\a\,d\tau.
\end{align*}
Note that
\begin{align*}
\frac{\partial F}{\partial\t}(\xi+m\tau+n,\t)
&= \frac{\partial}{\partial\t} \big(e(-m\t) F(\xi,\t)\big) \cr
&= e(-m\t) \frac{\partial F}{\partial\t}(\xi,\t) -2\pi i m e(-m\t) F(\xi,\t).
\end{align*}
Thus
\begin{align*}
&\phantom{=} 2 \pi i\big((m,n)^\ast \nu_2 - e(-m\t)\cdot \nu_2\big) \cr
&=\bigg(
\frac{1}{\t}+\t\frac{\partial F}{\partial\t}(\xi+m\tau+n,\t,\tau)
\bigg)\cdot\a\, d\tau
-e(-m\t)\bigg(
\frac{1}{\t}+\t\frac{\partial F}{\partial\t}(\xi,\t,\tau)\bigg)\cdot\a\,d\tau \cr
&= -2\pi i m \t e(-m\t) F(\xi,\t)\cdot\a\,d\tau
+ \frac{1}{\t}\big(1-e(-m\t)\big)\cdot\a\,d\tau.
\end{align*}
\end{proof}

If $(m,n)\in \Z^2$, then the results above imply that
$$
(m,n)^\ast\w = e(-m\t)\cdot\w(\xi,\tau).
$$
Since $e(-m\t)$ does not depend on $(\xi,\tau)$, $de(-m\t) = 0$ and $\w$ is
invariant under $\Z^2$.

\subsubsection{Modularity: invariance under $\SL_2(\Z)$}
Let
$$
\gamma = \begin{pmatrix}a & b \cr c & d \end{pmatrix} \in \SL_2(\Z).
$$
Recall that $M_\gamma(\tau)$ is defined by
\begin{align}
\label{eqn:automorphy}
\a &\mapsto (c\tau + d)^{-1}\a + 2\pi i c\t \cr
\t &\mapsto (c\tau+d)\t.
\end{align}
Its inverse is the linear map
\begin{align}
\a &\mapsto (c\tau + d)\a - 2\pi i c\t \cr
\t &\mapsto (c\tau+d)^{-1}\t.
\end{align}

\begin{lemma}
If $\gamma \in SL_2(\Z)$ and $a,b\in\N$, then
$$
e(c\xi\t/(c\tau+d))\cdot [\ad_\t^a(\a),\ad_\t^b(\a)]\frac{\partial}{\partial \a}
= [\ad_\t^a(\a),\ad_\t^b(\a)]\frac{\partial}{\partial\a}.
$$
\end{lemma}

\begin{proof}
This follows from (\ref{eq:exp_ad}) as the derivation $\delta =
[\ad_\t^a(A),\ad_\t^b(A)]\frac{\partial}{\partial \a}$ vanishes on $\t$.
\end{proof}

\begin{lemma}
If $\gamma \in \SL_2(\Z)$ and $a,b\in \N$, then
$$
\Ad(M_\gamma(\tau)) [\ad_\t^a(A),\ad_\t^b(\a)]\frac{\partial}{\partial\a}
= (c\tau+d)^{a+b-1}[\ad_\t^a(A),\ad_\t^b(A)]\frac{\partial}{\partial\a}.
$$
\end{lemma}

\begin{proof}
Set $\delta = [\ad_\t^a(\a),\ad_\t^b(\a)]\frac{\partial}{\partial\a}$. Since
$M_\gamma(\tau)^{-1}(\t) = (c\tau + d)^{-1}\t$, $\delta\circ
M_\gamma^{-1}(\t)=0$. Consequently, $\Ad(M_\gamma(\tau))\delta$ is of the form
$f(\t,\a)\frac{\partial}{\partial\a}$. The coefficient $f(\t,\a)$ is computed as
follows:
\begin{align*}
&\phantom{=}\Ad(M_\gamma(\tau))\delta(\a) \cr
&= M_\gamma(\tau)\circ\delta\circ M_\gamma(\tau)^{-1}(\a) \cr
&= M_\gamma(\tau)\circ \delta \big((c\tau+d)\a -2\pi i c\t\big)\cr
&= (c\tau+d) M_\gamma(\tau)\big([\ad_\t^a(\a),\ad_\t^b(\a)]\big)\cr
&= (c\tau+d)^{a+b+1}[\ad_\t^a\big((c\tau+d)^{-1}\a+2\pi i c\t\big),
\ad_\t^b\big((c\tau+d)^{-1}\a+2\pi i c\t\big)] \cr
&= (c\tau+d)^{a+b-1}[\ad_\t^a(\a),\ad_\t^b(\a)].
\end{align*}
\end{proof}

\begin{corollary}
If $\gamma \in \SL_2(\Z)$, then $\gamma^\ast \psi = \Ad(\Mtilde_\gamma)\psi$.
\end{corollary}

\begin{proof}
This follows as, for each $k\ge 1$, the expression
$$
G_{2k+2}(\tau)d\tau \otimes \sum_{\substack{a+b=2k+1\cr a,b>0}}
[\ad_\t^a(\a),\ad_\t^b(\a)]\frac{\partial}{\partial\a}
$$
is multiplied by $(c\tau+d)^{2k}$ by both $\gamma^\ast$ and
$\Mtilde_\gamma(\xi,\tau)$.
\end{proof}

\begin{lemma}
Set $\nu_1 = \t F(\xi,\t,\tau)\cdot \a d\xi$. Then
$$
\gamma^\ast \nu_1 - \Mtilde_\gamma(\xi,\tau) \nu_1
= - 2\pi i c\t\,d\xi
- \frac{c\xi \t}{c\tau + d}e(c\xi\t)F(\xi,(c\tau+d)\t,\tau)\cdot \a\, d\tau.
$$
\end{lemma}

\begin{proof}
First,
\begin{align*}
\Mtilde_\gamma(\xi,\tau)\nu_1
&= \Mtilde_\gamma(\xi,\tau)\big[\t F(\xi,\t,\tau)\cdot\a\big]d\xi\cr
&= e(c\xi\t)
(c\tau+d)\t F(\xi,(c\tau+d)\t,\tau)
\cdot\big((c\tau+d)^{-1}\a+2\pi ic\t\big)d\xi\cr
&= e(c\xi\t)\t F(\xi,(c\tau + d)\t,\tau)\cdot\a\, d\xi + 2\pi i c\t\,d\xi.
\end{align*}
as the value of $\t F(\xi,(c\tau+d)\t,\tau)$ at $\t=0$ is $(c\tau+d)^{-1}$. This
and the modular property of $F(\xi,\t,\tau)$ then yield:
\begin{align*}
\gamma^\ast \nu_1 &= \t F(\xi/(c\tau+d),\t,\gamma\tau)
\cdot\a \gamma^\ast d\xi \cr
&= (c\tau+d)\t e(c\xi\t)F(\xi,(c\tau+d)\t,\tau)\cdot\a
\bigg(\frac{d\xi}{c\tau+d}-\frac{c\xi d\tau}{(c\tau+d)^2}\bigg) \cr
&= \t e(c\xi\t)F(\xi,(c\tau+d)\t,\tau)\cdot\a
\bigg(d\xi-\frac{c\xi d\tau}{c\tau+d}\bigg) \cr
&= \Mtilde_\gamma(\xi,\tau) \nu_1 - 2\pi i c\t\,d\xi
- \frac{c\xi\t}{c\tau + d}e(c\xi\t)F(\xi,(c\tau+d)\t,\tau)\cdot\a\, d\tau.
\end{align*}
\end{proof}

As a special case of the general formula, we have:

\begin{lemma}
In $\Der \p$ we have:
$$
e(c\xi \ad_\t)\bigg(\frac{1}{2\pi i}\a \frac{\partial}{\partial\t}\bigg)
= \frac{1}{2\pi i}\a \frac{\partial}{\partial\t} +
\frac{1}{2\pi i}\frac{1-e(c\xi\t)}{\t}\cdot\a.
$$
\end{lemma}

\begin{lemma}
Set
$$
\nu_2 = \frac{1}{2\pi i}\bigg(
\frac{1}{\t} + \t\frac{\partial F}{\partial\t}(\xi,\t,\tau)
\bigg)\cdot\a\,d\tau.
$$
Then
\begin{multline*}
\gamma^\ast \nu_2 - \Mtilde_\gamma(\xi,\tau) \nu_2
= \frac{1}{2\pi i}\frac{1-e(c\xi\t)}{\t}\cdot\a \frac{d\tau}{(c\tau + d)^2}\cr
+ \frac{c\xi\t}{c\tau + d}e(c\xi\t)F(\xi,(c\tau+d)\t,\tau)\cdot \a\,d\tau.
\end{multline*}
\end{lemma}

\begin{proof}
First note that the modularity property of $F(\xi,\t,\tau)$ implies that
\begin{align*}
&\phantom{=}\frac{\partial F}{\partial \t}(\xi/(c\tau + d),\t,\gamma \tau)\cr
&=(c\tau+d)\frac{\partial}{\partial\t}\big[e(c\xi\t)F(\xi,(c\tau+d)\t,\tau)\big]
\cr
&= (c\tau + d)e(c\xi\t)\big[c\xi F(\xi,(c\tau+d)\t,\tau) +
(c\tau+d)\frac{\partial F}{\partial \t}(\xi,(c\tau+d)\t,\tau)\big].
\end{align*}
Thus
\begin{align*}
2\pi i \gamma^\ast\nu_2
&=\bigg(\frac{1}{\t}+\t\frac{\partial F}
{\partial\t}(\xi/(c\tau+d),\t,\gamma\tau)
\bigg)\cdot \a\,d\bigg(\frac{a\tau+b}{c\tau+d}\bigg)\cr
&= \bigg(\frac{1}{\t}
+ (c\tau + d)^2e(c\xi\t)\t\frac{\partial F}{\partial\t}(\xi,(c\tau+d)\t,\tau)\cr
& \qquad + c\xi\t(c\tau + d)e(c\xi\t)F(\xi,(c\tau+d)\t,\tau)
\bigg)\cdot \a\,\frac{d\tau}{(c\tau+d)^2}.
\end{align*}
Since
$$
\Mtilde_\gamma(\xi,\tau)(\t) = (c\tau+d)\t \text{ and }
\Mtilde_\gamma(\xi,\tau)(\a) =  e(c\xi\t)\cdot \a/(c\tau+d)+2\pi i c\t
$$
we have
\begin{align*}
&\phantom{=} 2\pi i \Mtilde_\gamma(\xi,\tau)\nu_2 \cr
&= \bigg(\frac{1}{(c\tau+d)\t} +
(c\tau+d)\t\frac{\partial F}{\partial\t}(\xi,(c\tau+d)\t,\tau)
\bigg)\cdot \big(e(c\xi\t)\cdot \a/(c\tau+d) +2\pi ic\t\big)\,d\tau \cr
&= e(c\xi\t)\bigg(\frac{1}{\t} +
(c\tau+d)^2 \t\frac{\partial F}{\partial\t}(\xi,(c\tau+d)\t,\tau)
\bigg)\cdot \a\frac{d\tau}{(c\tau+d)^2}
\end{align*}
as $\frac{1}{\eta} + \eta \frac{\partial F}{\partial\eta}(\xi,\eta,\tau)$ is
holomorphic in $\eta$ and vanishes at $\eta = 0$ by (\ref{eqn:derivative}).

The previous lemma implies that
$$
\big(e(c\xi \ad_\t)-1\big)
\bigg(\frac{1}{2\pi i} \a \frac{\partial}{\partial \t}\bigg)
=
\frac{1}{2\pi i}\bigg(\frac{1-e(c\xi\t)}{\t}\bigg)\cdot \a.
$$
Now assemble the pieces to obtain the result.
\end{proof}

Combining the last two computations, we obtain:

\begin{corollary}
For all $\gamma \in \SL_2(\Z)$,
$$
\gamma^\ast \nu - \Mtilde_\gamma \nu
= \frac{1-e(c\xi\t)}{2\pi i\t}\cdot \a\frac{d\tau}{(c\tau+d)^2}
- 2\pi i c\t\,d\xi.
$$
\end{corollary}

\begin{lemma}
For all $\gamma\in \SL_2(\Z)$,
$$
d\Mtilde_\gamma \Mtilde_\gamma^{-1}
= e(c\xi\t)\cdot \big(dM_\gamma M_\gamma^{-1}\big) + 2\pi ic\t\,d\xi.
$$
\end{lemma}

\begin{proof}
Since $\Mtilde_\gamma(\xi,\tau) = e(c\xi\t)M_\gamma(\tau)$, we have
\begin{align*}
d\Mtilde_\gamma \Mtilde_\gamma^{-1}
&= d\big(e(c\xi\t)M_\gamma\big)M_\gamma^{-1}e(-c\xi\t) \cr
&= \big(e(c\xi\t)dM_\gamma+2\pi ic\t e(c\xi\t)M_\gamma d\xi\big)
M_\gamma^{-1}e(-c\xi\t) \cr
&= e(c\xi\t)\cdot \big(dM_\gamma M_\gamma^{-1}\big) + 2\pi ic\t\,d\xi.
\end{align*}
\end{proof}

\begin{lemma}
For all $\gamma \in \SL_2(\Z)$, we have
$$
\gamma^\ast\bigg(\frac{1}{2\pi i}\a\frac{\partial}{\partial\t}d\tau\bigg)
- M_\gamma \bigg(\frac{1}{2\pi i}\a\frac{\partial}{\partial\t}d\tau\bigg)
+ dM_\gamma M_\gamma^{-1} = 0.
$$
\end{lemma}

\begin{proof}
This is best done using matrices with respect to the basis $\{\a,\t\}$ of $H$.
We have
$$
\frac{1}{2\pi i}\a\frac{\partial}{\partial \t}d\tau
= \frac{1}{2\pi i}\begin{pmatrix} 0 & 1 \cr 0 & 0 \end{pmatrix} d\tau
$$
and
$$
M_\gamma(\tau) =
\begin{pmatrix}
(c\tau + d)^{-1} & 0 \cr 2\pi i c & c\tau + d
\end{pmatrix}
,\qquad
M_\gamma(\tau)^{-1} =
\begin{pmatrix}
c\tau + d & 0 \cr -2\pi i c & (c\tau + d)^{-1}
\end{pmatrix}.
$$
So
\begin{align*}
dM_\gamma M_\gamma^{-1} &=
\begin{pmatrix}
-c(c\tau + d)^{-2} & 0 \cr 0 & c
\end{pmatrix}
\begin{pmatrix}
c\tau + d & 0 \cr -2\pi i c & (c\tau + d)^{-1}
\end{pmatrix}
d\tau \cr
&= \begin{pmatrix}
-c & 0 \cr -2\pi i c^2(c\tau+d) & c
\end{pmatrix}
\frac{d\tau}{c\tau+d}
\end{align*}
and
\begin{align*}
M_\gamma \bigg(\frac{1}{2\pi i}\a\frac{\partial}{\partial \t}d\tau\bigg)
&= \frac{1}{2\pi i}
\begin{pmatrix}
(c\tau + d)^{-1} & 0 \cr 2\pi i c & c\tau + d
\end{pmatrix}
\begin{pmatrix} 0 & 1 \cr 0 & 0 \end{pmatrix}
\begin{pmatrix}
c\tau + d & 0 \cr -2\pi i c & (c\tau + d)^{-1}
\end{pmatrix}
d\tau\cr
&= \frac{1}{2\pi i}\begin{pmatrix}
-2\pi i c & (c\tau+d)^{-1} \cr -(2 \pi ic)^2 (c\tau + d) & 2\pi ic
\end{pmatrix}
\frac{d\tau}{c\tau+d} \cr
&=
\frac{1}{2\pi i}\begin{pmatrix} 0 & 1 \cr 0 & 0 \end{pmatrix}
\frac{d\tau}{(c\tau+d)^2} +
\begin{pmatrix}
-c & 0 \cr -2 \pi i c^2 (c\tau + d) & c
\end{pmatrix}
\frac{d\tau}{c\tau+d} \cr
&= \gamma^\ast\bigg(\frac{1}{2\pi i}\a\frac{\partial}{\partial\t}d\tau\bigg)
+ dM_\gamma M_\gamma^{-1}.
\end{align*}

\end{proof}

\noindent
{\bf Final computation:} For all $\gamma \in \SL_2(\Z)$, we have:
\begin{align*}
&\phantom{=}\gamma^\ast\w-\Mtilde_\gamma\w
+ d\Mtilde_\gamma\Mtilde_\gamma^{-1}\cr
&= \gamma^\ast\bigg(\nu+\frac{\a}{2\pi i}\frac{\partial}{\partial\t}d\tau\bigg)
- \Mtilde_\gamma\bigg(
\nu + \frac{\a}{2\pi i}\frac{\partial}{\partial\t}d\tau\bigg)
+ e(c\xi\t)\cdot \big(dM_\gamma M_\gamma^{-1}\big) + 2\pi ic\t\,d\xi\cr
&= \gamma^\ast\bigg(\frac{\a}{2\pi i}\frac{\partial}{\partial\t}d\tau\bigg)
+ \frac{1-e(c\xi\t)}{2\pi i\t}\cdot \a\frac{d\tau}{(c\tau+d)^2}
- e(c\xi\t)\cdot M_\gamma\bigg(\frac{\a}{2\pi i}\frac{\partial}{\partial\t}d\tau
\bigg)
\cr & \hspace*{3.25in}
+ e(c\xi\t)\cdot \big(dM_\gamma M_\gamma^{-1}\big)\cr
&= e(c\xi\t)\cdot \bigg(
\gamma^\ast\bigg(\frac{1}{2\pi i}\a\frac{\partial}{\partial\t}d\tau\bigg)
- M_\gamma \bigg(\frac{1}{2\pi i}\a\frac{\partial}{\partial\t}d\tau\bigg)
+ dM_\gamma M_\gamma^{-1}\bigg) \cr
&= 0.
\end{align*}

\subsection{Integrability}

\begin{proposition}
The $1$-form $\w$ is closed.
\end{proposition}

\begin{proof}
It is clear that
$$
d\bigg(\frac{1}{2\pi i} d\tau\otimes \a\frac{\partial}{\partial \t} + \psi\bigg)
= 0
$$
as these terms do not depend upon $\xi$. The heat equation implies that
\begin{align*}
2\pi i\, d\nu &= 2\pi i\, \ad_\t dF(\xi,\ad_\t,\tau)(\a)\wedge d\xi + d\bigg(
\frac{1}{\ad_\t} + \ad_\t\frac{\partial F}{\partial \t}(\xi,\ad_\t,\tau)
\bigg)(\a)\wedge d\tau \cr
&= \ad_\t\bigg(
2\pi i\,\frac{\partial F}{\partial \tau}(\xi,\ad_\t,\tau) - 
\frac{\partial^2 F}{\partial \xi\partial \t}(\xi,\ad_\t,\tau)
\bigg)(\a)d\tau\wedge d\xi\cr
&= 0,
\end{align*}
so that $d\w = 0$.
\end{proof}

The proof of the vanishing of $[\w,\w]$ is quite involved. For this we employ 
the elegant calculus developed by Levin and Racinet in \cite[\S
3.1]{levin-racinet}. 

\subsubsection{The Levin-Racinet calculus}
\label{sec:calc}

For $U,V \in \L(\a,\t)^\wedge$, define
$$
x^r y^s \circ (U,V) = [\t^r\cdot U, \t^s\cdot V].
$$
This extends linearly to an action $f(x,y)\circ (U,V)$ of polynomials and power
series $f(x,y)$ in commuting indeterminants on ordered pairs of elements of
$\L(\t,\a)$. When $U$ and $V$ are equal, one has the identity $f(x,y)\circ (U,U) =
-f(y,x)\circ (U,U)$, so that
\begin{equation}
\label{eqn:skew}
f(x,y)\circ (U,U) = \frac{1}{2}\big(f(x,y)-f(y,x)\big)\circ(U,U).
\end{equation}
As an example of how this notation is used, note that Lemma~\ref{lem:wp} implies
that
\begin{equation}
\label{eqn:psi}
2\pi i \psi =
\frac{1}{2}\frac{xy}{x+y}\big(\wp(x)-\frac{1}{x^2} - \wp(y) + \frac{1}{y^2}\big)
\circ(\a,\a)\frac{\partial}{\partial \a} \otimes d\tau.
\end{equation}

Two more identities will be needed in the proof of the vanishing of $[\w,\w]$.

\begin{lemma}
\label{lem:identities}
Suppose that $U,V\in \L(\t,\a)^\wedge$.
\begin{enumerate}

\item (Jacobi identity) If $f(x,y) \in \C[[x,y]]$, then
$$
\ad_\t \big(f(x,y)\circ(U,V)\big) = (x+y)f(x,y)\circ(U,V).
$$

\item If $\delta$ is a continuous derivation of $\C\ll \t,\a\rr$ and $g(x)\in
\C[[x]]$, then
$$
\delta\big(g(\ad_\t)V\big) =
g(\ad_\t)\delta(V) + \bigg(\frac{g(x+y)-g(y)}{x}\bigg)\circ(\delta(\t),V).
$$

\end{enumerate}
\end{lemma}

\begin{proof}
The first identity encodes the Jacobi identity and is left as an easy exercise.
To prove the second, note that both sides are linear in $g$, so that, by
continuity, it suffices to prove the result when $g$ is a monomial $x^n$. This
holds trivially when $n\le 1$. The general case follows by induction using the
Jacobi identity.
\end{proof}

\subsection{Integrability}

The following computation completes the proof of integrability.

\begin{lemma}
The 2-form $[\w,\w]$ vanishes, so that $\w$ is integrable.
\end{lemma}

\begin{proof}
 Note that
\begin{multline*}
\pi i\,[\w,\w] = 
\bigg[d\tau\otimes \a\frac{\partial}{\partial \t} + 2\pi i\,\psi
\cr
+ 
d\tau \otimes\bigg(\frac{1}{\ad_\t} +
\ad_\t\frac{\partial F}{\partial \t}(\xi,\ad_\t,\tau)\bigg)(\a),
\ad_\t F(\xi,\ad_\t,\tau)(\a)\otimes d\xi\bigg].
\end{multline*}
The expression (\ref{eqn:psi}) implies that the coefficient of $d\tau\wedge
d\xi$ is
\begin{multline*}
[\a\frac{\partial}{\partial \t},\t F(\t)(\a)]
+
\frac{1}{2}\Big[
\frac{xy}{x+y}\big(\wp(x) - \frac{1}{x^2} - \wp(y)
+ \frac{1}{y^2}\big)\circ(\a,\a) \frac{\partial}{\partial\a},\t F(\t)(\a)
\Big]
\cr
+
\Big[\frac{1}{\t} + \t F'(\t)(\a), \t F(\t)(\a) \Big]
\end{multline*}
where $F(z)$ denotes $F(\xi,z,\tau)$ and $F'(z)$ denotes $\partial F/\partial
z(\xi,z,\tau)$. We'll compute these three terms, one at a time.

Since $[\delta,\ad_v] = \ad_{\delta(v)}$, Lemma~\ref{lem:identities} and
equation (\ref{eqn:skew}) imply that the first term is
\begin{align*}
[\a\frac{\partial}{\partial \t},\t F(\t)\cdot\a]
&= \bigg(\frac{(x+y)F(x+y)-yF(y)}{x}\bigg)\circ(\a,\a)
\cr
&= \frac{1}{2}\bigg(\Big(
\frac{y^2-x^2}{xy}\Big)\, F(x+y) - \frac{y}{x}\, F(y) + \frac{x}{y}\, F(x)
\bigg)\circ (\a,\a).
\end{align*}

The identity $[\delta,\ad_v] = \ad_{\delta(v)}$ and the Jacobi identity
(Lemma~\ref{lem:identities}) imply that the second term is
\begin{align*}
&\frac{1}{2}
\ad_{\t F(\t)}\bigg(\frac{xy}{x+y}\Big(\wp(x) - \frac{1}{x^2} - \wp(y)
+ \frac{1}{y^2}\Big)\circ(\a,\a)\bigg) \cr
&=
\frac{1}{2}
{xy}F(x+y)\Big(\wp(x) - \frac{1}{x^2} - \wp(y)
+ \frac{1}{y^2}\Big)\circ(\a,\a) \cr
&= \frac{1}{2}\bigg(
\Big(\frac{x^2-y^2}{xy}\Big)\, F(x+y) + xy(\wp(x)-\wp(y))F(x+y)
\bigg)\circ (\a,\a)
\end{align*}
The addition formula, Proposition~\ref{prop:addition}, implies that the third
term is
\begin{align*}
&\bigg(\Big(\frac{1}{x} + x F'(x)\Big)y F(y)\bigg) \circ(\a,\a) \cr
&= \Big(\frac{y}{x}\, F(y) + xy F'(x) F(y)\Big)\circ (\a,\a)\cr
&= \frac{1}{2}\Big(
\frac{y}{x}\, F(y) - \frac{x}{y}\, F(x) + xy \big(F'(x) F(y) - F'(y)F(x)\big)
\Big) \circ (\a,\a)\cr
&= \frac{1}{2}\Big(
\frac{y}{x}\, F(y) - \frac{x}{y}\, F(x)
-xy\big(\wp(x)-\wp(y)\big)F(x+y)
\Big)\circ (\a,\a).
\end{align*}
These three terms clearly sum to 0.
\end{proof}

\part{Hodge Theory and Applications}

The main goal of this section is to show that the elliptic KZB connection
underlies an admissible variation of mixed Hodge structure (MHS) over $\E' =
\M_{1,2}$ and to show that this variation is isomorphic to the variation of MHS
whose fiber over $[E,x]$ is the Lie algebra of the unipotent fundamental group
of $(E',x)$. We use this to show that the periods of the limit MHS on the fiber
of $\bP$ over $[E_{\partial/\partial q},\partial/\partial w]$ are multizeta
values and to derive an explicit formula for the natural morphism of MHS
$$
\pi_1^\un(\Pminus,\partial/\partial w) \to
\pi_1^\un(E_{\partial/\partial q}',\partial/\partial w).
$$

As preparation, we show that the KZB connection extends to a meromorphic
connection over $\Mbar_{1,2}$ with regular singularities and  pronilpotent
monodromy about the boundary divisors. To fill a gap in the literature, we
prove, in Section~\ref{sec:rigidity}, that the local system associated to the
universal elliptic KZB connection on $\bP$ is the local system $\bP^\top$. It
can be used to prove the analogous result for the KZB connection over
$\M_{1,1+n}$. This complements results in \cite[\S4]{levin-racinet} and \cite[\S
4.3]{cee}.

Throughout, $\p = \L(\t,\a)^\wedge$ and $\cP$ is the corresponding prounipotent
group. Set
$$
\Der^0 \p = \{\delta \in \Der \p : \delta([\t,\a]) = 0\}.
$$
This is the infinitesimal analogue of the mapping class group $\G_{1,\vec{1}}$.

The reader is assumed to be familiar with the basics of Deligne's theory of
mixed Hodge structures \cite{deligne:hodge2}. A good introductory reference is
the book \cite{steenbrink-peters} by Steenbrink and Peters. Another good
introductory reference is Carlson's paper \cite{carlson}. An exposition of the
construction of the mixed Hodge structure on the unipotent fundamental group
of a smooth variety can be found in \cite{hain:bowdoin}.

\section{Extending $\bP$ to $\Mbar_{1,2}$}
\label{sec:extension}

The punctured universal elliptic curve $\E'$ is isomorphic to $\M_{1,2}$ and is
the complement of a normal crossing divisor in $\Mbar_{1,2}$. This divisor has
two components: the zero section and the nodal cubic. The flat bundle $\cP$ over
$\E'$ has prounipotent monodromy about each, and thus extends naturally to a
bundle over $\Mbar_{1,2}$ with regular singularities and (pro) nilpotent
residues along each component.

This extension is easily described. First, the complement of the Tate curve in
$\Mbar_{1,2}$ is the universal elliptic curve $\E$, which is a quotient of
$\C\times \h$.  The bundle $\bP$ defined in Section~\ref{sec:bP} is defined over
$\E$, not just over $\E'$. We take this to be the extension across the zero
section.

To extend $\bP$ across the Tate curve, recall from Example~\ref{ex:hodge} that
the holomorphic vector bundle $\cH := \H\otimes_\C \O_{\D^\ast}$ associated to
$\H$ is trivial on the punctured $q$-disk $\D^\ast$. The framing $\t$, $\a$ of
$\cH$ over $\D^\ast$ determines an extension $\Hbar$ of $\H$ to the entire
$q$-disk $\D$ that is framed by $\t$ and $\a$. The formula for the natural
connection on $\cH$ in Example~\ref{ex:connection_H} implies it extends to a
meromorphic connection on $\Hbar$ with a regular singular point at the cusp
$q=0$ and with nilpotent residue. This implies that $\Hbar$ is Deligne's
canonical extension of $\cH$ to $\Mbar_{1,1}$. (Cf.\ \cite{deligne:ode}.)

Since $\p = \L(\t,\a)^\wedge$, this determines an extension of $\p\times \C
\times \D^\ast \to \C \times \D^\ast$ to $\C\times \D$; its fiber over $q \in
\D$ is the free Lie algebra generated by the fiber of $\Hbar$ over $q$, which is
naturally isomorphic to $\L(\t,\a)^\wedge$. The pullback of the universal
elliptic curve over $\Mbar_{1,1}$ to $\D$ minus the double point $P$ of the
nodal cubic is the quotient of $\C\times \h$ by the subgroup
$$
\G := \begin{pmatrix}
1 & \Z \cr 0 & 1
\end{pmatrix}
\ltimes \Z^2.
$$
of $\SL_2(\Z)\ltimes \Z^2$.

The action of $\G$ on $\p \times \C \times \h$ induces an action of $\G$ on $\p
\times \C \times \D$. The pullback of $\bP$ to $\E_{\D^\ast}$ thus extends to a
bundle $\bPbar$ over $\E_\D$ (minus the double point $P$ of the nodal cubic) as
the quotient of this action.

Formulas (\ref{eqn:qexp}), (\ref{eqn:def_F}) and (\ref{eqn:derivative})
imply that this extension has regular singularities along the identity section
and along the nodal cubic $q=0$.

\begin{proposition}
The meromorphic extension of the elliptic KZB connection defined above has
regular singularities along the two boundary components of $\Mbar_{1,2}$: the
nodal cubic $E_0$ and the identity section. It has pronilpotent residue
at each codimension 1 boundary point. \qed
\end{proposition}

\section{Restriction to $E_\tau'$}

Fix $\tau \in \h$. The first task in proving that the universal elliptic KZB
connection has the expected monodromy is to check that its restriction to the
fiber $E_\tau'$ of $\E' \to \M_{1,1}$ induces an isomorphism of
$\pi_1(E_\tau',x)^\un$ with $\cP$.

The restriction of the universal elliptic KZB connection to $E_\tau'$ is
$$
\nabla = d + \nu_1 = d + \t F(\xi,\t,\tau)\cdot\a\, d\xi.
$$
Identify $\p$ with the image of the adjoint action $\ad : \p \to \Der \p$ which
is injective as $\p$ has trivial center. With this identification, $\nabla$
takes values in $\p$. Fix $x\in \C-\Lambda_\tau$. The associated monodromy
representation $\rho_x : \pi_1(E_\tau',x) \to \cP$ is given by (cf.\
Cor.~\ref{cor:chen-dyson})
$$
\rho_x(\gamma) = \bigg(1 + \int_{c_\gamma}\nu_1 + \int_{c_\gamma}\nu_1\nu_1 +
\int_{c_\gamma} \nu_1\nu_1\nu_1 + \cdots\bigg)e(-m(\gamma)\t)
$$
where $\rho(\gamma) = \big(m(\gamma),n(\gamma)\big) \in \Z^2$. (That is,
the class of $\gamma$ in $H_1(E_\tau')$ is $n(\gamma)\a + m(\gamma)\b$.)

\begin{proposition}
If $[\gamma] = n\a + m\b \in H_1(E_\tau')$, then
$$
\Theta_x(\gamma) \equiv 1 + (m\tau+n)\a - 2\pi i m \t \bmod (\t,\a)^2.
$$
\end{proposition}

\begin{proof}
Observe that
\begin{align*}
\nu_1 &= \t F(\xi,\t,\tau)\cdot \a\,d\xi \cr
&= \t\bigg(\frac{1}{\t} + \frac{1}{\xi} + \text{ holomorphic in }\xi \bigg)
\cdot \a\, d\xi \cr
&\equiv \a\,d\xi \bmod (\t,\a)^2
\end{align*}
and that $\Res_{\xi=0} \nu_1 \equiv [\t,\a] \bmod (\t,\a)^3$. It follows that
$\Theta_x(\a) \equiv 1+\a \bmod (\t,\a)^2$ and that
$$
\Theta_x(\b) \equiv (1+\tau \a)e(-\t)
\equiv 1 + \tau\a - 2\pi i \t \bmod (\t,\a)^2.
$$
\end{proof}

\begin{corollary}
\label{cor:h1}
The universal elliptic KZB connection induces the identification
$$
\C\a \oplus \C\t \to H_1(E_\tau;\C)
$$
that takes $\a$ to $\a$ and $2\pi i \t$ to $\tau \a - \b$, the Poincar\'e
dual of $\w_\tau$.
\end{corollary}

This corresponds to the framing of the bundle $\H$ given in
Example~\ref{ex:hodge}. (This statement can also be deduced from
Remark~\ref{rem:abelianization2}.)

The universal connection induces an isomorphism of the unipotent completion of
$\pi_1(E_\tau',x)$ with $\cP$ for all $(x,\tau) \in \E_\h$. This is a special
case of \cite[Prop.~2.2]{cee}.

\begin{corollary}
\label{cor:isom}
The monodromy of the restriction of the universal elliptic KZB connection to the
fiber $E_\tau'$ of $\E'$ over $[E_\tau]\in \M_{1,1}$ is a homomorphism
$\pi_1(E_\tau',x) \to \cP$ that induces an isomorphism $\pi_1(E_\tau',x)^\un \to
\cP$.
\end{corollary}

\subsection{A better framing of $\H$}

To get rid of the powers of $2\pi i$ in the formulas, we replace $2\pi i\,d\tau$
by $dq/q$ and set
\begin{equation}
\label{eqn:def_AT}
{T = 2 \pi i\, \t \text{ and } A = (2\pi i)^{-1}\a}.
\end{equation}

\begin{remark}
\label{rem:limit_H1}
There is a conceptual reason the basis $A,T$ is a good choice. Denote the
fiber of the universal elliptic curve $\E_\D \to \D$ over $q\in \D$ by $E_q$.
For each nonzero tangent vector $\v$ of $0\in \D$, there is a limit MHS on $H_1$
of the fiber, which we denote by $H_1(E_\v)$ and think of as the homology of the
fiber over $\v$. This MHS is an extension
$$
0 \to \Q(1) \to H_1(E_\v) \to \Q(0) \to 0
$$
which splits when $\v = \partial/\partial q$. In this case, the copy of $\Q(1)$
is spanned by $A$ and the copy of $\Q(1)$ is spanned by $T$.
As will become apparent in Part~\ref{part:Q-DR}, this limit MHS has a $\Q$-DR
form. The basis $A$, $T$ is a $\Q$-DR basis of $H_1(E_{\partial/\partial q})$.
The basis above is the extension of this basis to a framing of the bundle
$\Hbar_\D$ in which $F^0\Hbar$ is trivialized by $T$. The basis $\a$, $\b$
is a $\Q$-Betti basis of $H_1(E_\v)tro$.
\end{remark}

In this frame $[\t,\a] = [T,A]$,
$$
\frac{1}{2\pi i}d\tau \otimes \a \frac{\partial}{\partial \t}
= \frac{dq}{q}\otimes A \frac{\partial}{\partial T}.
$$
and the terms in the KZB connection become:
$$
\psi =
\sum_{m\ge 1}\bigg(\frac{G_{2m+2}(\tau)}{(2m)!}\frac{dq}{q} \otimes
\sum_{\substack{j+k=2m+1\cr j,k > 0}} (-1)^j[\ad_T^j(A),\ad_T^k(A)]
\frac{\partial}{\partial A}\bigg),
$$
and
$$
\nu =
T F^\zag(2\pi i \xi, T,\tau)\cdot A \frac{dq}{q}
+ \bigg(
\frac{1}{T} + T\frac{\partial}{\partial T} F^\zag(2\pi i\xi,T,\tau)
\bigg)\cdot A\frac{dq}{q}
$$

\begin{remark}
The periodicity properties (\ref{eqn:w}) and (\ref{eqn:q}) of $F$ and the
formulas for the factors of automorphy (\ref{eqn:def}) and (\ref{eqn:M_tilde})
imply that the connection is pulled back from a connection on the trivial bundle
$\L(A,T)^\wedge \times \C^\ast\times\D \to \C^\ast\times \D$ along the map
$\C\times \h \to \C^\ast\times \D$ defined by $(\xi,\tau)\mapsto (w,q) :=
(e(\xi),e(\tau))$.
\end{remark}

\section{Restriction to the First-order Tate Curve}
\label{sec:tate}

In this section we compute the restriction of the universal elliptic KZB
connection to  the first order Tate curve. This allows us to see that the
connection has regular singularities along the nodal cubic. Restricting further
to the the regular locus $\Pminus$ of the nodal cubic minus its identity is the
first step in computing the image of $\pi_1(\Pminus,\v)^\un$ in the limit mixed
Hodge structure on the unipotent fundamental group of the first order smoothing
of the Tate curve.

The restriction of the natural extension of the KZB connection to the boundary
divisor $q=0$ is the image of $\w$ under the restriction mapping
$$
\Omega^1_\cEbar(\log(\Mbar_{1,1}\cup \Ebar_0)) \to
\Omega^1_\cEbar(\log(\Mbar_{1,1}\cup \Ebar_0))\otimes_{\O(\Ebar)}\O_{E_0},
$$
where $E_0 \cong \Gm$ is the fiber $\Ebar_0$ of $\cEbar$ over $q=0$ with the
double point removed. The identity section of $\cEbar$ is identified with
$\Mbar_{1,1}$. In concrete terms the restriction mapping is given by
$$
G(\xi,q) \frac{d\xi}{\xi} + H(\xi,q)\frac{dq}{q} \mapsto
G(\xi,0) \frac{d\xi}{\xi} + H(\xi,0)\frac{dq}{q} 
$$
where $G$ and $H$ are holomorphic functions of $(\xi,q)$, and then setting
$w = e(\xi)$.

Formula~(\ref{eqn:qexp}) implies that
$$
F(\xi,\eta)|_{q=0} = \pi i\bigg(
\frac{e(\xi)+1}{e(\xi)-1} + \frac{e(\eta)+1}{e(\eta)-1}
\bigg)
=\pi i \bigg(
\frac{w+1}{w-1} + \coth(\pi i \eta)
\bigg)
$$
where $w := e(\xi)$ is the parameter in the normalization $\P^1$ of the nodal
cubic $E_0$. From this and the identity (\ref{eqn:derivative}), it follows that
when $q=0$
$$
\frac{1}{\eta} + \eta \frac{\partial F}{\partial\eta}(\xi,\eta)|_{q=0}
= \frac{1}{\eta} - \frac{(\pi i)^2\eta}{\sinh^2(\pi i \eta)}
= \frac{1}{\eta} - \frac{\pi^2\eta}{\sin^2(\pi \eta)},
$$
which is holomorphic at $\eta = 0$.

The restriction of the connection to a first order neighbourhood of $E_0$ is
given by the 1-form:
$$
\w_0 =  \frac{dq}{q}\otimes A\frac{\partial}{\partial T}
+ \psi_0 + \nu_0,
$$
where
\begin{align*}
\psi_0
&=
\sum_{m\ge 1}\bigg(\frac{1}{(2m)!}G_{2m+2}|_{q=0}\frac{dq}{q}
\otimes \sum_{\substack{j+k=2m+1\cr j,k > 0}}
(-1)^{j}[\ad_T^j(A),\ad_T^k(A)]\bigg)
\frac{\partial}{\partial A} \cr
&=-
\sum_{m\ge 1}\bigg(\frac{(2m+1)B_{2m+2}}{(2m+2)!}\frac{dq}{q}
\otimes \sum_{\substack{j+k=2m+1\cr j>k > 0}}
(-1)^j[\ad_T^j(A),\ad_T^k(A)]\bigg)
\frac{\partial}{\partial A}
\end{align*}
and, using the identity (\ref{eqn:coth}),
\begin{align*}
\nu_0 &=
\frac{T}{2}\bigg(\frac{w+1}{w-1} +
\frac{e^T+1}{e^T-1}\bigg)\cdot A\,\frac{dw}{w}
+
\frac{T/4}{\sinh^2(T/2)} \cdot A\,\frac{dq}{q} \cr
&=
[T,A]\, \frac{dw}{w-1} + \bigg(\frac{T}{e^T-1}\bigg)\cdot A\, \frac{dw}{w}
+  \bigg(\frac{1}{T} - \frac{T/4}{\sinh^2(T/2)}\bigg) \cdot A\,\frac{dq}{q}
\cr
&= 
[T,A]\, \frac{dw}{w-1} + \bigg(\frac{T}{e^T-1}\bigg)\cdot A\, \frac{dw}{w}
+
\sum_{m=1}^\infty (2m-1)\frac{B_{2m}}{(2m)!}\ad_T^{2m-1}A \frac{dq}{q}
\end{align*}

At this stage, it is convenient to define $\d_{2m} \in \Der_{2m}\p$ (the
derivations of $\p$ of degree $2m$) by\footnote{These derivations occur in the
work \cite{tsunogai} of Tsunogai on the action of the absolute Galois group on
the fundamental group of a once punctured elliptic curve. They also occur in the
paper of Calaque et al \cite[\S 3.1]{cee}.}
\begin{equation}
\label{eqn:deltas}
\d_{2m} =
\begin{cases}
-A\frac{\partial}{\partial T} & m = 0; \cr
\ad_T^{2m-1}(A)-
\sum_{\substack{j+k=2m-1\cr j>k > 0}}(-1)^j[\ad_T^j(A),\ad_T^k(A)]
\frac{\partial}{\partial A} & m > 0.
\end{cases}
\end{equation}

Assembling the pieces, we see that the restriction of the KZB connection form
to a first order neighbourhood of the Tate curve is
\begin{equation}
\label{eqn:w_0}
\w_0 = [T,A]\,\frac{dw}{w-1} + \bigg(\frac{T}{e^T-1}\bigg)\cdot A\, \frac{dw}{w}
+
\sum_{m=0}^\infty (2m-1)\frac{B_{2m}}{(2m)!}\d_{2m} \frac{dq}{q}.
\end{equation}

\subsection{Monodromy logarithms}

The residue of the connection at $(w,q)$ acts on the fiber $\L(A,T)^\wedge$ of
$\p$ over it as a derivation. At each point $(w,0)$, where $w\neq 0,1$, the
residue of the connection (\ref{eqn:w_0}) is
\begin{equation}
\label{eqn:residue}
N_q = \sum_{m\ge 0}(2m-1)\frac{B_{2m}}{(2m)!}\,\d_{2m}.
\end{equation}
The residue $N_w$ of the connection of each point along the identity section
$w=1$ is
$$
N_w = \ad_{[T,A]}.
$$
Since the KZB connection is flat and since the two boundary components intersect
transversely at $(w,q) = (1,1)$, $[N_q,N_w]=0$. This implies that each
$\d_{2n}$ annihilates $[T,A]$, and therefore lies in
$$
\Der^0\p := \{\delta \in \Der \p : \delta([T,A]) = 0\}.
$$

\begin{proposition}[Tsunogai \cite{tsunogai}]
\label{prop:highest_wt}
For all $m\ge 0$, $\d_{2m} \in \Der^0 \p$. When $m>1$, $\d_{2m}$ is a highest
weight vector of weight $2m-2$ for the natural $\sl_2$-action; that is, it is
annihilated by $A\partial/\partial T$. \qed
\end{proposition}

It can be shown that $\d_{2m}$ is a highest weight vector of the {\em unique}
copy of $S^{2m}H$ in $\Gr^W_{-2m}\Der^0\p$.

The following computation follows from identity (\ref{eqn:coth_alt}) and the
fact that
$$
\d_{2m}(T) = (\ad_T^{2m-1}A) \cdot T = - \ad_T^{2m}(A) = -T^{2m}\cdot A.
$$

\begin{proposition}
\label{prop:N(T)}
The value of $N_q$ on $T$ is
$$
N_q(T) = \frac{1}{4} \bigg(\frac{T^2}{\sinh^2(T/2)}\bigg) \cdot A. \qed
$$
\end{proposition}

\subsection{Pullback to $\P^1-\{0,1,\infty\}$}
\label{sec:Pminus}

We can pullback the connection to $\P^1$ along the ``map'' $E'_0 = \Pminus \to
E_{\partial/\partial q}$ to the fiber of $\E$ over the tangent vector
$\partial/\partial q$. Just set $dq$ to zero to get:
\begin{equation}
\label{eqn:restnKZB}
\w_{E_0'} = [T,A]\, \frac{dw}{w-1} +
\bigg(\frac{T}{e^T-1}\bigg)\cdot A\, \frac{dw}{w}.
\end{equation}
Since
$$
\frac{T}{e^T-1} + \frac{T}{e^{-T}-1} + T = 0,
$$
it follows that
the residues $R_0$, $R_1$ and $R_\infty$ of $\w_{E_0'}$ at $0,1,\infty$ are:
\begin{equation}
\label{eqn:def_Rs}
R_0 = \bigg(\frac{T}{e^T-1}\bigg)\cdot A,\quad
R_1 = [T,A],\quad
R_\infty = \bigg(\frac{T}{e^{-T}-1}\bigg)\cdot A
\end{equation}

\section{Restriction to $\M_{1,\vec{1}}$}
\label{sec:vec}

In this section, we compute the restriction of the universal elliptic KZB
connection to the first order neighbourhood $\Mbar_{1,\vec{1}}$ of the 0-section
of $\cEbar$. In algebraic terms, this restriction map is induced by the
$\O_{\cEbar}$-module homomorphism
$$
\Omega^1_\cEbar(\log \Mbar_{1,1}) \to \Omega^1_\cEbar(\log \Mbar_{1,1})
\otimes_{\O_\cEbar}\O_{\Mbar_{1,1}}
$$ 
Here we are identifying the zero-section of $\cEbar$ with $\Mbar_{1,1}$. This
computation will allow us to see that the restricted connection takes values in
$\Der^0\p$.

In concrete terms the restriction mapping is given by
$$
G(\xi,\tau) \frac{d\xi}{\xi} + H(\xi,\tau)d\tau \mapsto
G(0,\tau) \frac{d\xi}{\xi} + H(0,\tau)d\tau 
$$
where $G$ and $H$ are holomorphic functions of $(\xi,\tau)$. The restricted
connection is thus given by the 1-form
$$
\w' = \frac{dq}{q}\otimes A\frac{\partial}{\partial T} + \psi + \nu'
$$
where
\begin{align*}
\nu' &= [\t,\a]\frac{d\xi}{\xi} + \frac{1}{2\pi i}\bigg(
\frac{1}{\t} + \t\frac{\partial F}{\partial \t}(0,\t,\tau)
\bigg)\cdot \a\, d\tau  \cr
&= [T,A]\bigg(\frac{d\xi}{\xi} - 2G_2(\tau)\frac{dq}{q}\bigg)
- \sum_{m\ge 1}
\frac{2}{(2m)!} G_{2m+2}(\tau)\,\frac{dq}{q}\otimes\ad_T^{2m+1}(A).
\end{align*}
Note that both terms in this last expression are $\SL_2(\Z)$-invariant
and that the term $\psi$ remains unchanged as it does not depend on $\xi$:
$$
\psi = \sum_{m\ge 1}\bigg(\frac{2}{(2m)!}G_{2m+2}(\tau)\frac{dq}{q} \otimes
\sum_{\substack{j+k=2m+1\cr j>k > 0}} (-1)^j[\ad_T^j(A),\ad_T^k(A)]
\frac{\partial}{\partial A}\bigg).
$$

\begin{proposition}
The restriction of the KZB connection to $\M_{1,\vec{1}}$ is given by the
$\SL_2(\Z)$-invariant 1-form
\begin{equation}
\label{eqn:conn}
\w' = -\frac{dq}{q}\otimes \d_0 -
\bigg(2G_2(\tau)\frac{dq}{q} - \frac{d\xi}{\xi}\bigg) \otimes \d_2
- \sum_{m=2}^\infty \frac{2}{(2m-2)!} G_{2m}(\tau)\frac{dq}{q}\otimes\d_{2m}
\end{equation}
on $\Der^0\L(A,T)^\wedge\times\C\times\h\to \C\times \h$.
\end{proposition}

This gives an alternative computation of $N_q$:
$$
N_q = \Res_{q=0}\w' = \sum_{m=0}^\infty (2m-1)\frac{B_{2m}}{(2m)!} \d_{2m}.
$$

\section{Rigidity}
\label{sec:rigidity}

We have not yet proved that $\bP$ with the KZB connection is isomorphic to the
flat bundle $\bP^\top$ defined in Section~\ref{sec:bPtop}. This will be resolved
in this section by proving that both have the same monodromy representation.

The punctured universal elliptic curve $\E'$ is the moduli space $\M_{1,2}$.
Choose a base point $[E_o,x_o]$ of $\M_{1,2}$, where $x_o\neq 0$. There is a
natural isomorphism
$$
\pi_1(\M_{1,2},[E_o,x_o]) \cong \pi_0 \Diff^+(E_o,x_o,0) \cong \Gamma_{1,2},
$$
where $\G_{1,2}$ is the mapping class group of a genus 1 curve with 2 marked 
points.

The restriction of the universal elliptic KZB connection to $E_o$ defines a
homomorphism $\pi_1(E_o',x_o) \to \Aut \p$ whose image lies in the subgroup
$\cP=\exp\p$ which acts on $\p$ via the adjoint action. Corollary~\ref{cor:isom}
implies that it induces an isomorphism $\pi_1^\un(E_o',x_o) \to \cP$.

Identify $\cP$ with $\pi_1^\un(E_o',x_o)$ via this isomorphism. Then one has the
monodromy representations
$$
\rho^\KZB: \G_{1,2} \to \Aut \cP \text{ and }
\rho^\top : \G_{1,2} \to \Aut\cP
$$
of $\bP$ and $\bP^\top$. To prove that $\bP^\top$ and $\bP$ are isomorphic, we
have to prove that $\rho^\KZB=\rho^\top$. Observe that if $\gamma \in
\pi_1(E_o',x_o)$, then $\rho^\top(\gamma)$ and $\rho^\KZB(\gamma)$ are both
conjugation by the image of $\gamma$ in $\cP$ as the restriction of $\bP$ and
$\bP^\top$ to $E'_o$ are isomorphic.

To prove that $\rho^\KZB=\rho^\top$ it is useful to consider a more abstract
situation. Suppose that $N$ is a normal subgroup of a discrete group $\G$.
Denote the unipotent completion of $N$ by $\cN$.\footnote{One can use any field
over char 0, but we will take $\C$.} The homomorphism $\G \to \Aut N$ that takes
$g\in \G$ to $n \mapsto gng^{-1}$ induces a homomorphism $\phi : G \to \Aut
\cN$. The restriction of $\phi$ to $N$ takes $n \in N$ to $\iota_{\theta(n)}$,
where $\theta : N \to \cN$ is the natural homomorphism and $\iota_u$ denotes
conjugation by $u \in \cN$.

\begin{lemma}
\label{lem:extension}
If $\cN$ has trivial center, $\phi$ is the unique homomorphism $\G \to \Aut\cN$
whose restriction to $N$ is $n\mapsto \iota_{\theta(n)}$.
\end{lemma}

\begin{proof}
The condition that $\cN$ have trivial center implies that the centralizer of $\im
\theta$ in $\Aut \cN$ is trivial: if $\sigma \in \Aut \cN$, then
$$
\sigma \iota_{\theta(n)} \sigma^{-1} = \iota_{\sigma(\theta(n))}.
$$
So $\sigma$ centralizers $\im \theta$ if and only if $\iota_{\theta(n)} =
\iota_{\sigma(\theta(n))}$ for all $n\in N$. Since  $\im\theta$ is Zariski dense
in $\cN$ and since $\cN$ has trivial center, this implies that $\sigma = \id$.

Suppose now that $\alpha : \G \to \Aut \cN$ is a homomorphism whose restriction
to $N$ is $n \mapsto \iota_{\theta(n)}$. If $g\in\G$, then
\begin{multline*}
\alpha(g)\iota_{\theta(n)}\alpha(g)^{-1} = \alpha(g)\alpha(n)\alpha(g)^{-1}
= \alpha(gng^{-1})
\cr
= \iota_{gng^{-1}} = \dots = \phi(g)\iota_{\theta(n)}\phi(g)^{-1}
\end{multline*}
for all $n\in N$.
So $\alpha(g)^{-1}\phi(g)$ centralizes $\im \theta$ and is therefore trivial.
\end{proof}

Applying the lemma with $\G = \G_{1,2}$, $N=\pi_1(E_o',x_o)$, $\cN=\cP$ and
$\phi = \rho^\top$ establishes the equality of $\rho^\KZB$ and $\rho^\top$.

\begin{theorem}
\label{thm:rigidity}
The exponential mapping induces an isomorphism of the locally constant sheaf
over $\E'$ of flat sections  of the universal elliptic KZB connection on $\bP$
with the locally constant sheaf $\bP^\top$ over $\E'$. Equivalently, the diagram
$$
\xymatrix{
\pi_1(\E',[E,x_o]) \ar[r]^(.6){\rho^\KZB} \ar@{=}[d] & \Aut \p \ar[d]^{\cong}
\cr
\pi_1(\E',[E,x_o]) \ar[r]_(.6){\rho^\top} & \Aut \cP
}
$$
commutes.
\end{theorem}

\begin{remark}
A similar argument can be used to prove that the local system associated to the
KZB connection over $\M_{1,1+n}$ constructed in \cite{cee} is the canonical
local system whose fiber over $[E,0,x_1,\dots,x_n]$ is the unipotent completion
of the fundamental group of the configuration space of $n$ points on $E'$ with
base point $(x_1,\dots,x_n)$.
\end{remark}

Any other connection on the extension of $\bP$ to $\Mbar_{1,2}$ with regular
singularities with pronilpotent residue and conjugate monodromy representation
differs from the KZB connection by a holomorphic map $\Mbar_{1,2} \to \Aut \p$.
Since $\Aut \p$ is an affine group and since $\Mbar_{1,2}$ is complete, the
change of gauge must be constant.

\begin{corollary}
The universal elliptic KZB connection is the unique meromorphic connection on
$\p \times \C \times \h$ (up to a constant change of gauge) satisfying:
\begin{enumerate}

\item it is $\SL_2(\Z)\ltimes \Z^2$ invariant with the same factors of
automorphy as the elliptic KZB connection,

\item it has poles along $\{(z,\tau)\in \C\times\h : z \in \Lambda_\tau\}$
and is holomorphic elsewhere,\

\item its natural extension to $\bP$ over $\Mbar_{1,2}$ has regular
singularities with pronilpotent residues,

\item for any one $[E,x]\in \M_{1,2}$, the restriction of the connection to
$\E'$ induces an isomorphism $\pi_1(E',x)^\un \to \cP$.

\end{enumerate}
\end{corollary}

One consequence of this rigidity result is that the canonical extension of the
connection constructed by Calaque, Enriquez and Etingof \cite{cee} over
$\M_{1,2}$  equals the connection constructed by Levin and Racinet in
\cite{levin-racinet}.

\section{Hodge Theory}
\label{sec:hodge}

In this section we show that, with appropriate filtrations, $\bP$ is an
admissible variation of mixed Hodge structure (MHS) over $\E'$. We assume that
the reader is familiar with the definition of mixed Hodge structures. We begin
by recalling the definition of an admissible variation of MHS over a smooth
variety. Further details can be found in \cite{steenbrink-zucker} and
\cite{kashiwara}.

\subsection{Admissible variations of mixed Hodge structure}
Suppose that $X$ is a smooth projective variety (or orbifold) and that $U=X-D$
is the complement of a normal crossing divisor $D$ in $X$.
\begin{enumerate}

\item Suppose that $\V$ is a local system of finite dimensional $\Q$-vector
spaces over $U$. For simplicity, we assume that the local monodromy operator at
each smooth point $P$ of $D$ is unipotent. This holds for each finite
dimensional quotient of $\bP$. 

\item Let $\cV$ be Deligne's canonical extension of the flat vector bundle
$\V\otimes_\Q \O_U$ to $X$. It is a holomorphic vector bundle with a connection
$$
\nabla : \cV \to \cV \otimes \Omega^1_X(\log D)
$$
with regular singularities along $D$. It is characterized by the property
that the residue of the connection at each smooth point of $D$ is nilpotent.

\item Suppose that $F^\dot \cV$ is a decreasing filtration of $\cV$ by
holomorphic vector bundles over $X$ and that these satisfy ``Griffiths
transversality":
$$
\nabla : F^p \cV \to  F^{p-1}\cV \otimes \Omega^1_X(\log D) =:
F^p\big(\cV \otimes \Omega^1_X(\log D)\big).
$$
That is, $\nabla$ respects the Hodge filtration.

\item There is an increasing filtration $W_\dot$ of $\V$. It induces a
filtration $W_\dot \cV$ of $\cV$ by flat sub-bundles.

\item Suppose that for each $x \in U$ the restriction of $F^\dot \cV$ and 
$W_\dot \V$ to the fiber $V_x$ of $\V$ over $x$ define a MHS on $V_x$.

\end{enumerate}
When $X$ is a curve, these data (a ``pre-variation'' of MHS) form an {\em
admissible} variation of MHS if for each $P\in D$, there is a {\em relative
weight filtration} $M_\dot$ of the fiber $V_P$ of $\cV$ over $P$ and its
nilpotent endomorphism $N = \Res_P\nabla$.\footnote{Background material on
relative weight filtrations can be found in \cite{steenbrink-zucker} and
\cite{hain:weight}.} This means that:
\begin{enumerate}

\item $M_\dot$ is an increasing filtration of $V_P$ satisfying
$N(M_r V_P) \subseteq M_{r-2}V_P$ and $N(W_mV_P)\subseteq W_m V_P$ for all
$m$ and $r$,

\item for each $m$ and each $k$, $N^k$ induces an isomorphism
$$
N^k : \Gr^M_{m+k}\Gr^W_m V_P \to \Gr^M_{m-k}\Gr^W_m V_P.
$$

\end{enumerate}
In this case, for each $P\in D$ and each choice of non-zero tangent vector $\v
\in T_P X$, there is a canonical MHS on the fiber $V_P$ of $\cV$ over $P$. This
MHS will be denoted $V_\v$. It has weight filtration $M_\dot$; its Hodge
filtration is the restriction of $F^\dot\cV$ to $V_P$. The $\Q$-structure on
$V_P$ is spanned by the elements
$$
\lim_{t\to 0} (t/c)^{-N}v(t) \in V_P,
$$
where $N=\Res_P \nabla$, $t$ is a local holomorphic parameter on $X$ centered at
$P$, $\v=c\partial/\partial t$, and $v(t)$ is a flat section of $\V$ defined in
an angular sector containing a ray tangent to $\v$. Each $W_m V_P$ is a sub-MHS.

When $X$ has dimension $>1$, the pre-variation is admissible if its restriction
to each curve in $X$ is admissible. We will also refer to pro-objects of the
category of admissible variations of MHS as admissible variations.

\subsection{The variation $\bP$}
We now show that the KZB equation gives rise to an admissible variation of MHS.
Denote the maximal ideal $(T,A)$ of $\Q\ll T,A \rr$ by $I$. As in
\cite{levin-racinet}, define Hodge and weight filtrations on $\Q\ll T,A\rr$ in
the natural way by setting
$$
W_{-n}\Q\ll T,A\rr = I^n \text{ and }
F^{-p}\Q\ll T,A \rr = \{x \in \Q\ll T,A\rr : \deg_A(x) \le p\}.
$$
We also define a {\em relative weight filtration} $M_\dot$ on $\Q\ll T,A\rr$ by
$$
M_{-2m}\Q\ll T,A \rr = \{x \in \Q\ll T,A\rr : \deg_A(x) \ge m\}.
$$
These filtrations are multiplicative. By restriction, they induce Hodge, weight
and relative weight filtrations on $\p$ and thus on $\Der \p$. They also induce
filtrations on the bundle $\p \times \C \times \h$ over $\C\times \h$.

\begin{theorem}
\label{thm:hodge}
The Hodge and weight filtrations on $\p \times \C \times \h$ descend to Hodge
and weight filtrations of the local system $\bP \to \E'$. With these
filtrations, the local system $\bP$ over $\E'=\M_{1,2}$ and its restriction to
$\M_{1,\vec{1}}$ are admissible variations of MHS whose weight graded quotients
are direct sums of Tate twists of $S^m\H$. The MHS on the fiber over $[E,x]$ is
the canonical MHS on the Lie algebra of the unipotent completion of
$\pi_1(E',x)$. The relative weight filtration of the limit MHS associated to the
tangent vector $\v = \partial/\partial q + \partial/\partial w$ at the identity
of the nodal cubic is $M_\dot$. In addition, there is a natural homomorphism
$$
\pi_1(\Pminus,\partial/\partial w)^\un \to
\pi_1(E_{\lambda\partial/\partial q}',\partial/\partial w)^\un
$$
whose image is invariant under monodromy and which is a morphism of MHS for all
$\lambda \in \C^\ast$
\end{theorem}

\begin{proof}
It suffices to consider the case where the base is $\E'$. The factor of
automorphy $\Mtilde_\gamma(\xi,\tau)$ (Cf.\ (\ref{eqn:M_tilde})) preserves both
the Hodge and weight filtrations on $\p\times \C\times \h$. They therefore
descend to filtrations of the bundle $\bP$ over $\E'$. Next observe that each
component of the universal elliptic KZB connection is a 1-form that takes values
in $F^{-1}W_0 \Der^0 \p$. This implies that, over $\E'$, the connection
satisfies Griffiths transversality and that the weight bundles are sub-local
systems of $\bP$. The weight filtration is defined over $\Q$ as it is defined in
terms of the lower central series filtration of $\p$.

As explained in Section~\ref{sec:extension}, the natural extension of the
universal elliptic KZB connection to the $q$-disk has regular singular points
along the nodal cubic and along the identity section. The extension of the Hodge
and weight bundles to the $q$-disk are the quotients of the bundles
$$
(F^p \p) \times \C^\ast \times \D \text{ and }
(W_m \p) \times \C^\ast \times \D
$$
over $\C^\ast\times \D$ (with coordinates $(w,q)$) by the factor of automorphy.
These are well defined as $\Mtilde_{(m,n)}(q) = e^{-mT}$, which lies in
$F^0 W_0 M_0 \Der^0\p$.

The residue at each point of the identity section is $N_w := \ad_{[T,A]}$. It
lies in $F^{-1}W_{-2}M_{-2}\Der^0\p$, which implies that the Hodge, weight and
(where relevant) the relative weight filtrations extend across the identity
section.

According to (\ref{eqn:residue}) the residue of the connection at each point of
the nodal cubic is
$$
N_q = \sum_{m\ge 0}(2m-1)\frac{B_{2m}}{(2m)!}\,\d_{2m} \in \Der^0 \p.
$$
It is easy to check that for each $m\ge 0$, $\d_{2m} \in
F^{-1}M_{-2}W_{-2m}\Der^0 \p$ for all $m\ge 0$, so that
$$
N_q \in F^{-1}M_{-2}W_{0}\Der^0\p.
$$
Moreover,
$$
\Gr^W_\dot N_q : \Gr^W_\dot \C\ll T,A\rr \to \Gr^W_\dot \C\ll T,A \rr
$$
is $\d_0$. Set $H=\C A \oplus \C T$. Since $\Gr^W_{-m} \C\ll T,A\rr = S^m H$
placed in weight $-m$ and since $\Gr^W_\dot N_q = -\d_0 = A\partial/\partial T
\in \sl(H)$, it follows easily from the representation theory of $\sl(H)$ that
$N_q^k$ induces an isomorphism
$$
\Gr^M_{-m+k} \Gr^W_{-m}\C\ll T,A\rr \to \Gr^M_{-m-k}\Gr^W_{-m} \C\ll T,A \rr.
$$
This implies that $M_\dot$ is the relative weight filtration of $N_q$ and for
$N_q + N_w$, which completes the proof that $\bP$ is an admissible variation of
MHS over $\M_{1,2}$.

To prove that the MHS on the fiber of $\bP$ over $[E]\in \M_{1,1}$ is its
canonical MHS (as defined in \cite{hain:dht1} or \cite{hain-zucker}) we consider
the restriction $\bP_E$ of $\bP$ to the fiber $E'$. The above discussion implies
that this is an admissible variation of MHS over $E'$. In fact, it is clearly  a
unipotent variation of MHS. Fix a base point $x\in E'$.
Theorem~\ref{thm:rigidity} implies that the fiber $\p(E,x)$ of $\bP_E$ over $x$
is naturally isomorphic to the Lie algebra of the unipotent completion of
$\pi_1(E',x)$. The monodromy representation $\theta_x : \p(E,x) \to \End
\p(E,x)$ of $\bP_E$ is the adjoint action. Since the center of $\p(E,x)$ is
trivial ($\p$ is free of rank $2$), $\theta_x$ is injective. Denote $\p(E,x)$
with its canonical MHS by $\p(E,x)^\canon$ and $\p(E,x)$ with the MHS given by
the elliptic KZB connection via Theorem~\ref{thm:rigidity} by $\p(E,x)^\KZB$.
The main theorem of  \cite{hain-zucker} implies that
$$
\theta_x : \p(E,x)^\canon \to \End \p(E,x)^\KZB
$$
is a morphism of MHS. On the other hand, since $\p(E,x)^\KZB$ is a Lie algebra
in the category of pro-MHS,
$$
\theta_x : \p(E,x)^\KZB \to \End \p(E,x)^\KZB
$$
is also a morphism of MHS. Since $\theta_x$ is injective, this implies that the
MHSs $\p(E,x)^\canon$ and $\p(E,x)^\KZB$ are equal.

The final statement follows from the construction \cite{hain:dht2} of limit
mixed Hodge structures on homotopy groups. It will be explained in greater
detail in \cite{hain:lmhs}.
\end{proof}

We can now prove that $N_q$ annihilates the $R_\alpha$.

\begin{corollary}
\label{cor:N(R)}
The derivation $N_q$ annihilates $R_0$, $R_1$ and $R_\infty$.
\end{corollary}

An elementary proof is given in Appendix~\ref{sec:identity}. Here we sketch a
more conceptual proof.

\begin{proof}[Sketch of Proof]
Since $N_q$ annihilates $R_1 = [T,A]$, and since $R_0 + R_1 + R_\infty = 0$, it
suffices to show that $N_q(R_0) = 0$. 

Denote the limit MHS on $\Q\pi_1(\Pminus,w)^\wedge$ associated to the tangent
vector $\partial/\partial w \in T_1\P^1$ by $\Q\pi_1(\Pminus,\partial/\partial
w)^\wedge$ and the limit MHS on $\Q\pi_1(E'_q,x)^\wedge$ associated to the
tangent vector $\lambda \partial/\partial q + \partial/\partial w$ of $\E$ at
the identity of the nodal cubic. One has the commutative diagram
$$
\xymatrix{
\C\pi_1(\Pminus,\partial/\partial w)^\wedge \ar[r]^(.65){\Theta_{KZ}}\ar[d] &
\C\ll R_0,R_1\rr \ar[d]\cr
\C\pi_1(E_{\lambda\partial/\partial q}',\partial/\partial w)^\wedge
\ar[r]_(.65){\Theta_\lambda} & \C\ll T,A\rr
}
$$
where the left hand vertical mapping is the one given by
Theorem~\ref{thm:hodge}, the right hand vertical map is given by the formulas
for $R_0$ and $R_1$, and where the top horizontal map is the standard
isomorphism given by the KZ-connection. The result follows as the image of
$\Theta_{KZ}$ is invariant under monodromy and as the logarithm of monodromy
acting on $\C\ll T,A\rr$ is $N_q$.
\end{proof}

As noted in Remark~\ref{rem:limit_H1}, the limit MHS on $H_1(E_\tau)$ is an
extension of $\Z$ by $\Z(1)$. The basis $T$, $A$ splits both the Hodge and
monodromy weight filtrations.  The following statement follows directly from
Corollary~\ref{cor:h1}.

\begin{lemma}
\label{lem:limitMHS}
The limit MHS on $H_1(E_\tau)$ associated to the tangent vector
$\lambda\partial/\partial q$ of the origin of the $q$-disk has complex basis $A$
and $T$ and integral basis spanned by $\a$ and $-\b$, where
$$
-\b = T - \log \lambda  A \text{ and } \a = 2\pi i A,
$$
so that the corresponding period matrix is
$$
\begin{pmatrix}
1 & -\log \lambda \cr 0 & 2\pi i
\end{pmatrix}
$$
Denote it by $H_1(E_{\lambda\partial/\partial q})$. \qed
\end{lemma}

\section{Pause for a Picture}
\label{sec:pause}

Suppose that $f : Y \to X$ is a family of varieties which is locally
topologically trivial over the Zariski open subset $X' = X-S$ of $X$, where $S$
is a normal crossing divisor. Suppose that $P\in S$ and that $\v$ is a non-zero
tangent vector at $P$ that is not tangent to $S$. Suppose that $\V$ is an
admissible variation of MHS over $X'$ whose fiber over $x\in X'$ is a given
topological invariant of the fiber $Y_x$ of $f$ over $x$. It is useful to think
of the limit MHS of $\V$ associated to $\v$ as the MHS on that invariant of the
``fiber $Y_\v$ of $f$ over $\v$.'' For example, we would like to think of the
limit MHS of $\H$ associated to the tangent vector $\v =
\lambda\partial/\partial q$ at the origin of the $q$-disk as being
$H_1(E_{\lambda\partial/\partial q})$, and the limit MHS of $\bP$ associated to
the tangent vector $\v = \partial/\partial q + \partial/\partial w$ of the
universal elliptic curve at the identity of the nodal cubic as a MHS on the Lie
algebra $\p(E_{\partial/\partial q},\partial/\partial w)$ of the unipotent
fundamental group of $(E_{\partial/\partial q}',\partial/\partial w)$. The goal
of this section is to explain how to make this more precise using real oriented
blowups.

\subsection{Real oriented blowups and tangential base points}
\label{sec:tang_bps}

The real oriented blowup of a Riemann surface $X$ at a finite subset $S$ will
be denoted by $\Bl^o_S X$. This is a bordered Riemann surface with one boundary
circle for each point of $S$. There is a continuous projection $\pi:\Bl^o_S X
\to X$ that induces a biholomorphism
$$
\Bl^o_S X - \partial\Bl^o_S X \longrightarrow X-S
$$
The fiber of $\pi$ over $P\in S$ is the quotient of $(T_PX)-\{0\}$ by the
multiplicative group of  positive real numbers.

\begin{example}
The real oriented
blowup of the unit disk at the origin is $[0,1)\times S^1$. The projection
to the disk is $(r,\theta)\mapsto re^{i\theta}$.
As explained in Appendix~\ref{sec:blowup}, there is a natural identification
of $\Bl^o_{0,\infty}\P^1$ with $[0,1]\times S^1$.
\end{example}

Each non-zero vector $\v \in T_P X$ determines an element $[\v] \in \Bl^o_P X$.
Set $X' = X-\{P\}$. The fundamental group $\pi_1(X',\v)$ is defined to be
$\pi_1(\Bl^o_P X,[\v])$. (cf.\ \cite{deligne:p1}.) When $t>0$, $\pi_1(X',\v)$
and $\pi_1(X',t\v)$ are canonically isomorphic. It is important to note that the
MHSs on their unipotent completions will not be isomorphic except when the local
monodromy operator associated to $\v$ acts trivially on $\pi_1^\un(X',\v)$.

\subsection{The fiber of $\E$ over $\partial/\partial q$}

We now sketch the construction of the fiber $E_{\partial/\partial q}$ of $\E$
over $\partial/\partial q$. Full details can be found in
Appendix~\ref{sec:blowup}. 

Suppose that $q \in \D^\ast$. The fiber of the universal elliptic curve over $q$
is $E_q := \C^\ast/q^\Z$. For $0<r<1$, set
$$
A_r := \{w\in \C^\ast : r^{1/2}\le |w| \le r^{-1/2}\}
\text{ and }
\Ahat_r := \Bl_1^o A_r.
$$
Denote their outer and inner boundaries by $\partial_+ A_r$ and $\partial_-A_r$,
respectively. The elliptic curve $E_q$ is the quotient of $A_{{|q|}}$ obtained
by identifying $w\in \partial_+A_{{|q|}}$ with $qw\in \partial_-A_{{|q|}}$.
Similarly, $\Bl^o_1 E_q$ is a quotient of $\Ahat_q$. See
Figure~\ref{fig:move}.

\begin{figure}[!ht]
\epsfig{file=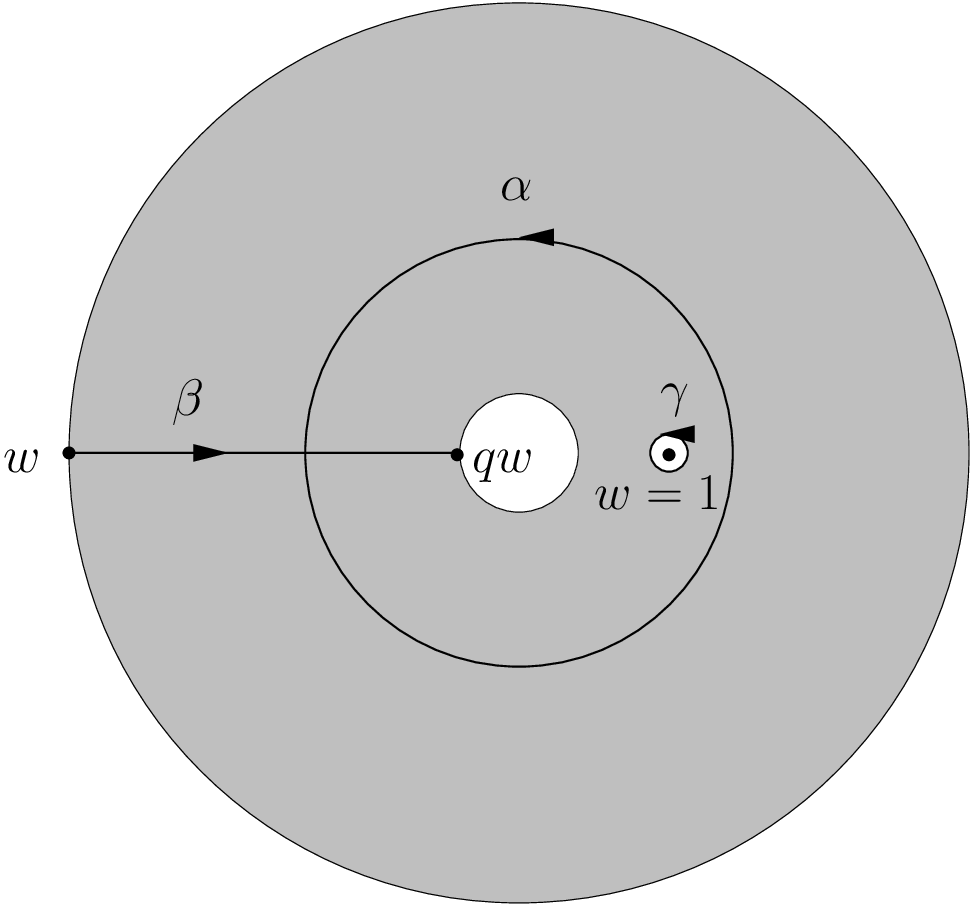, width=1.5in}
\caption{$\Bl_0^o E_q$ as a quotient of $\Ahat_{|q|}$}
\label{fig:move}
\end{figure}

The homology class $\a$ corresponds to the class of the positively oriented unit
circle $\alpha$; the homology class $\b$ corresponds to a path $\beta$ from $w$
to $qw$, where $w\in \partial_+A_{|q|}$. Note that these have intersection
number $+1$.  From
this it is evident that $\a$, the class of $\alpha$, is the vanishing cycle. The
Picard Lefschetz transformation is also evident: as $q$ travels once around
$\D^\ast$ in the positive direction, $a$ remains invariant, but $\b$, the class
of $\beta$ changes to $\b+\a$. This can also be seen from the formula
$$
T =
\tau \a - \b = \frac{\a}{2\pi i} \log q - \b.
$$
Let $\gamma$ be the boundary circle at $w=1$ of $\Ahat_r$.

As $r\to 0$, $A_r$ ``converges to'' $\Bl^o_{0,\infty}\P^1$ and $\Ahat_r$
``converges to'' $\Bl_{0,1,\infty}\P^1$. So $E_{re^{i\theta}}$ and
$\Bl_1E_{re^{i\theta}}$ ``converge'' to the surfaces obtained by identifying the
boundary circles of $\Bl_{0,\infty}^o\P^1$ and $\Bl_{0,1,\infty}\P^1$ at $0$ and
$\infty$ by multiplication $e^{i\theta}$. The resulting surface will be denoted
by $E_{e^{i\theta}\partial/\partial q}$. See Figure~\ref{fig:E_v}.
\begin{figure}[!ht]
\epsfig{file=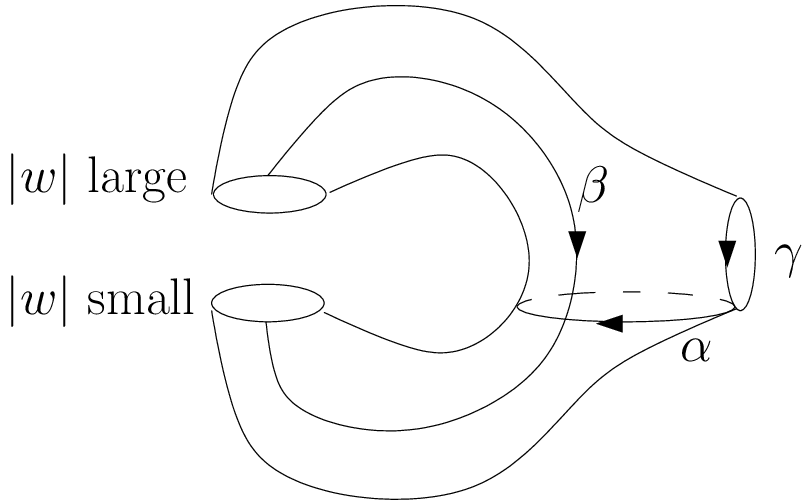, width=2in}
\caption{The ``nearby fiber'' $\Bl_1^oE_{\partial/\partial q}'$}
\label{fig:E_v}
\end{figure}
In Appendix~\ref{sec:blowup} we show that the universal family of elliptic
curves over $\D^\ast$ extends to a family of tori over $\Bl^o_0 \D$ whose fiber
over $[e^{i\theta}\partial/\partial q]$ is $E_{e^{i\theta}\partial/\partial q}$.

The curve $\gamma$ represents the commutator of two generators of
$\pi_1(\Bl_1(E_{\partial/\partial q}, [\partial/\partial w]))$. This is
consistent with the fact that
$$
\Res_{w=1}\w_{E_0'} = [T,A] = \frac{1}{2\pi i}[T,\a] = 
\frac{1}{2\pi i}[\tau \a - \b,\a] = \frac{1}{2\pi i}[\a,\b].
$$

\section{The KZ-equation and the Drinfeld Associator}
The quotient map
$$
\C\ll X_0,X_1,X_\infty\rr/(X_0+X_1+X_\infty) \to \C\ll X_0,X_1\rr
$$
is an isomorphism. We will identify these two rings. Recall that the
KZ-connection on $\P^1-\{0,1,\infty\}$ is given by
$$
\w_{KZ} = \frac{dw}{w} X_0 + \frac{dw}{w-1} X_1
\in H^0\big(\Omega^1_{\P^1}(\log\{0,1,\infty\})\big)
\comptensor \C\ll X_0,X_1\rr.
$$

The form $\w_{KZ}$ defines a flat connection on the trivial bundle
$$
\C\ll X_0,X_1\rr \times \P^1-\{0,1,\infty\} \to \P^1-\{0,1,\infty\}
$$
by the formula
$$
\nabla f = df - f\w_{KZ}.
$$
Its transport function induces a transport function
$$
\{\text{paths in }\P^1-\{0,1,\infty\}\} \to
\{\text{group-like elements of }\C\ll X_0,X_1\rr\}
$$
in $\P^1-\{0,1,\infty\}$.
It takes the path $\gamma$ in $\P^1-\{0,1,\infty\}$ to
\begin{equation}
\label{eqn:transport}
T(\gamma) = 1 + \int_\gamma \w_{KZ} + \int_\gamma \w_{KZ}\w_{KZ}
+ \int_\gamma \w_{KZ}\w_{KZ}\w_{KZ} + \cdots
\end{equation}
Since $\w_{KZ}$ is clearly integrable, the connection is flat and $T(\gamma)$
depends only on the homotopy class of $\gamma$ relative to its endpoints.

There are six standard tangent vectors of $\P^1-\{0,1,\infty\}$. Two are
anchored at each of $0,1,\infty$. They lie in one orbit under the action of the
symmetric group $S_3$ on $\P^1-\{0,1,\infty\}$ and are thus determined by the
two vectors at $w=0$, which are $\pm \partial/\partial w$.
\begin{center}
\begin{figure}[!ht]
\epsfig{file=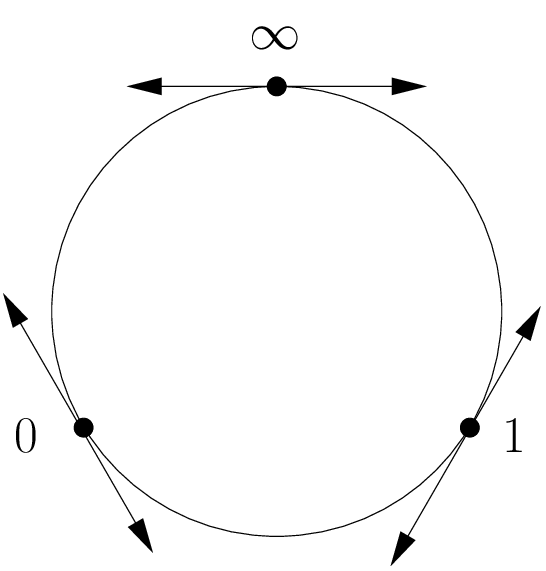, width=1.25in}
\caption{The $6$ tangent vectors}
\end{figure}
\end{center}
These have the property that their reduction mod $p$ is non-zero for all
prime numbers $p$.

The (KZ/de~Rham) version of the Drinfeld associator is the invertible power
series $\Phi(X_0,X_1) \in \C\ll X_0,X_1 \rr$ obtained by taking the regularized
value of the transport (\ref{eqn:transport}) above on the path $[0,1]$. It
begins:
\begin{multline*}
\Phi(X_0,X_1) = 1 - \zeta(2)[X_0,X_1] + \zeta(3)[X_0,[X_0,X_1]]
+ \zeta(1,2)[[X_0,X_1],X_1]
\cr
\phantom{xxxxxxxxxx} - \zeta(4)[X_0,[X_0,[X_0,X_1]]]
- \zeta(1,3)[X_0,[[X_0,X_1],X_1]]
\cr
-\zeta(1,1,2)[[[X_0,X_1],X_1],X_1]
+ \frac{1}{2}\zeta(2)^2[X_0,X_1]^2 + \cdots
\end{multline*}
where, for positive integers $n_1,\dots,n_r$, where $n_r>1$,
$$
\zeta(n_1,\dots,n_r) =
\sum_{0<k_1<\dots <k_r} \frac{1}{k_1^{n_1}k_2^{n_2}\dots k_r^{n_r}}.
$$
These are the {\em multiple zeta numbers}. They generalize the values of the
Riemann zeta function at positive integers. An explicit formula for
$\Phi(X_0,X_1)$ is given in \cite{le-murakami,furusho}.\footnote{To get this
formula for $\Phi(X_0,X_1)$, one has to reverse the order of all monomials ---
equivalently, replace each bracket $[U,V]$ by its negative $-[U,V]$. This is
because Furusho uses the opposite convention for path multiplication.} All
coefficients are rational multiplies of multiple zeta values.

Several (not all) of its basic properties are summarized in the following
result:

\begin{theorem}[Drinfeld]
The Drinfeld associator $\Phi$ satisfies:
\begin{enumerate}

\item $\Phi(X_0,X_1)\Phi(X_1,X_0) = 1$

\item In the ring $\C\ll X_0,X_1,X_\infty\rr/(X_0+X_1+X_\infty)$ we have
$$
\Phi(X_0,X_1) e^{i\pi X_1} \Phi(X_1,X_\infty) e^{i\pi X_\infty}
\Phi(X_\infty,X_0) e^{i\pi X_0} = 1.
$$

\end{enumerate}
\end{theorem}

The normalized value of $T$ on the unique real path from $w=0$ to $w=1$ is
$\Phi(X_0,X_1)$. View the symmetric group $S_3$ as $\Aut \{0,1,\infty\}$.  The
action of the automorphisms of $\P^1-\{0,1,\infty\}$ on the cusps determines an
isomorphism
$$
\Aut(\P^1,\{0,1,\infty\}) \to \Aut\{0,1,\infty\}.
$$
Let it act on $\{X_0,X_1,X_\infty\}$ by permuting the indices. Since the
connection is invariant under the $S_3$-action on $\P^1-\{0,1,\infty\}$, we
have, for example, the following values of the normalized transport on the real
paths:
$$
T^\norm([1,\infty]) = \Phi(X_1,X_\infty),\quad
T^\norm([0,\infty]) = \Phi(X_0,X_\infty),
$$
where $[0,\infty]$ is the path from $0$ to $\infty$ along the negative real
axis.

\subsection{The fundamental groupoid of $\P^1-\{0,1,\infty\}$}

Consider the category whose objects are the 6 tangent vectors of $\P^1$ defined
above and whose morphisms are homotopy classes from one tangent vector to
another.\footnote{That is, the path starts with one tangent vector and ends with
the negative of the second. Such a path $\gamma$ must also satisfy $\gamma(t)
\notin \{0,1,\infty\}$ when $0<t<1$. Composition of two such homotopy classes of
paths can be defined when the second path begins at the tangent vector where the
first ends.} Denote it by $\Pi(\P^1,V)$. It is generated by the paths shown in
the diagram. As above, a good topological model is to replace $\Pminus$ by the
real oriented blow-up of $\P^1$ at $\{0,1,\infty\}$, which is a 3-holed sphere,
and represent the tangent directions by the corresponding points on the boundary
of the blown up sphere.

\begin{center}
\begin{figure}[!ht]
\epsfig{file=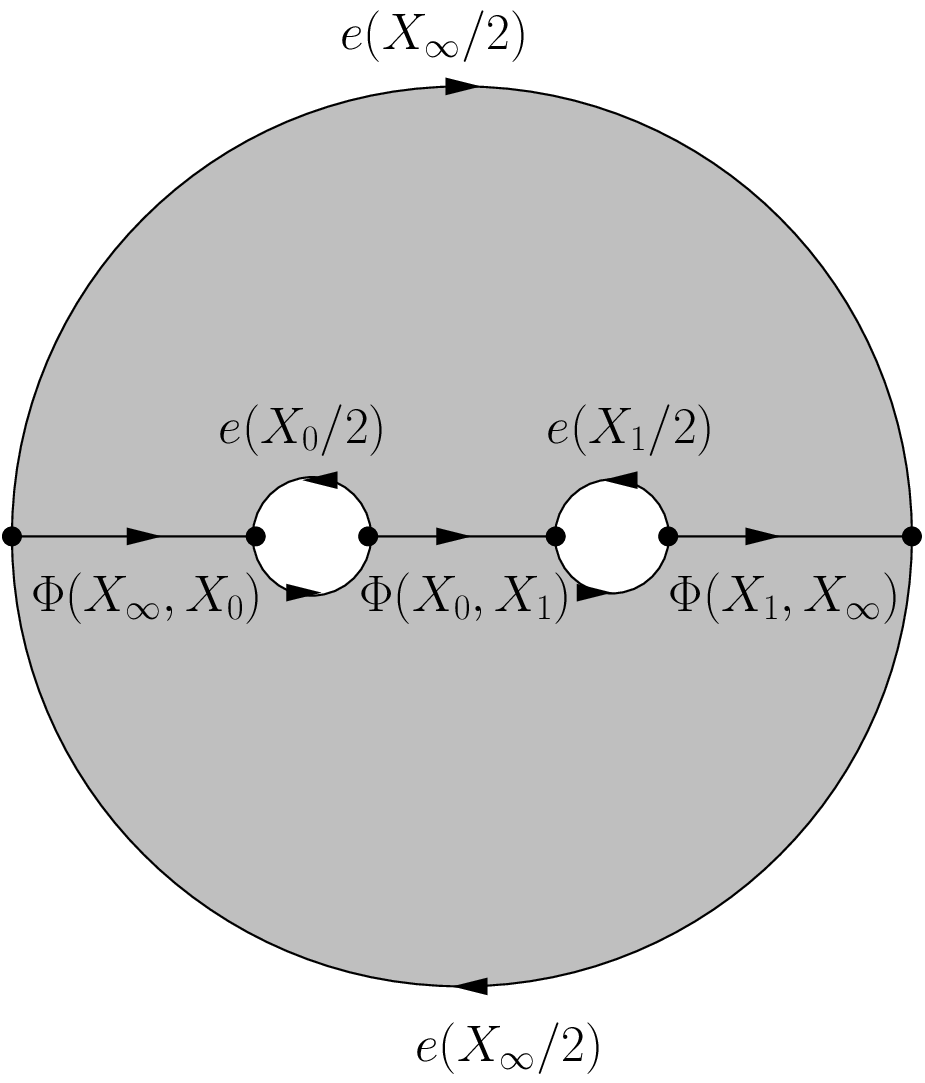, width=1.75in}
\caption{The fundamental groupoid of $\P^1-\{0,1,\infty\}$}
\end{figure}
\end{center}

Define a functor
$$
\Theta : \Pi(\P^1,V) \to
\big\{\text{group-like elements of } \C\ll X_0,X_1\rr \big\}
$$
by taking the positively oriented semi-circle about $a \in \{0,1,\infty\}$ to
$e(X_a/2)$ and the real interval from $a$ to $b$ to $\Phi(X_a,X_b)$. Drinfeld's
relations imply that $\Theta$ is well-defined.

This restricts to a group homomorphism $\Theta_\v :
\pi_1(\P^1-\{0,1,\infty\},\v) \to \cP$ for each of the 6 distinguished tangent
vectors $\v$.

\section{The Limit MHS on
$\pi_1(E_{\partial/\partial q}',\partial/\partial w)^\un$}
\label{sec:limit_mhs}

The computations of Section~\ref{sec:Pminus} imply that the restriction
$\w_{E_0'}$ of the KZB connection to $E_0' =\Pminus$ is obtained from $\w_{KZ}$
by composing it with the ring homomorphism
$$
\C\ll X_0,X_1,X_\infty \rr/(X_0+X_1+X_\infty) \hookrightarrow \C\ll T, A\rr
$$
defined by
\begin{align}
\label{eqn:homom}
X_0 &\mapsto R_0 = \bigg(\frac{T}{e^T-1}\bigg)\cdot A, \cr
X_1 &\mapsto R_1 = [T,A], \cr
X_\infty &\mapsto R_\infty = \bigg(\frac{T}{e^{-T}-1}\bigg)\cdot A,
\end{align}
which is well defined as $R_0+R_1+R_\infty = 0$.

Since the periods of $\w_{KZ}$ are understood (they are multiple zeta numbers),
this formula will allow us to compute the periods of $\w_{E_0'}$ in terms of
multiple zeta numbers.

\subsection{The Cylinder Relation}

To construct a well-defined homomorphism
$$
\pi_1(\E_{\partial/\partial q}',\partial/\partial w) \to \cP,
$$
we need to find all solutions $U\in \L(T,A)^\wedge$ of the equation
\begin{equation}
\label{eqn:cylinder}
e^{-U} e^{\lambda R_0} e^U e^{\lambda R_\infty}  = 1
\text{ in } \Q\ll T,A \rr.
\end{equation}
for all $\lambda \in \Q^\times$. This relation will be called the {\em cylinder
relation}. Note that a solution to the equation with $\lambda=1$ will be a
solution for all $\lambda$.

\begin{lemma}
For all $\lambda \in \C$, the relation $e^T e^{\lambda R_0} e^{-T} e^{\lambda
R_\infty} = 1$ holds in $\cP$. That is, $U=-T$ is a solution of the cylinder
equation.
\end{lemma}

\begin{proof}
It suffices to prove the relation $e^T R_0 e^{-T} = - R_\infty$.  Since
$\phi\exp(u) \phi^{-1} = \exp(\phi u \phi^{-1})$, we have
\begin{multline*}
e^T R_0 e^{-T}
= e^T \bigg(\bigg[\frac{T}{e^{T}-1}\bigg]\cdot A\bigg) e^{-T}
= e^{\ad_T}\bigg(\bigg[\frac{T}{e^{T}-1}\bigg]\cdot A\bigg)
\cr
= \bigg[\frac{Te^T}{e^{T}-1}\bigg]\cdot A
= \bigg[\frac{-T}{e^{-T}-1}\bigg]\cdot A
= - R_\infty.
\end{multline*}
\end{proof}

\begin{proposition}
Every solution of the cylinder relation (\ref{eqn:cylinder}) is of the form
$$
e^U  = e^{\lambda R_0} e^{-T}
$$
for some $\lambda \in \C$.
\end{proposition}

\begin{proof}
Suppose that $U$ is a solution of the cylinder equation. Set $V = \log(e^U
e^T)$. Then $V\in \L(T,A)^\wedge$ and $e^U = e^V e^{-T}$. The cylinder relation
implies that
$$
e^T e^{-V} e^{R_0} e^V e^{-T} = e^{-U} e^{R_0} e^U = e^{-R_\infty}
= e^T e^{R_0} e^{-T}
$$
so that $e^{-V} e^{R_0} e^V = e^{R_0}$. The result follows as the centralizer of
$R_0$ in $\L(T,A)^\wedge$ is $\C R_0$.
\end{proof}

\subsection{The homomorphisms
$\pi_1(E_{\partial/\partial q}',\partial/\partial w) \to \cP$}

The positive real axis determines two points $v_0$ and $v_\infty$ on the real
oriented blowup of $\P^1$ at $\{0,1,\infty\}$ --- the point $v_0$ lies on the
circle at $0$ and $v_\infty$ lies on the circle at $\infty$. There is a natural
$U(1)$ action on each of these circles.

Suppose that $\lambda \in \C^\ast$. Write it in the form $re^{i\theta}$. View
$E_{\lambda\partial/\partial q}'$ as the quotient of the real oriented blow-up
of $\P^1$ at $\{0,1,\infty\}$ by
$$
e^{i\phi}v_\infty \sim e^{i(\theta-\phi)}v_0.
$$
Denote the image of the two identified circles in $E_{\lambda\partial/\partial
q}$ by
$C$.
One can check that as $\lambda$ moves around the unit circle in the positive
direction, the identification changes by a positive Dehn twist about $C$. Note
that for each $\lambda$ there is a natural inclusion
$$
\iota : (\Pminus,\partial/\partial w) \to
(E_{\lambda\partial/\partial q}',\partial/\partial w)
$$
where in both cases $\partial/\partial w$ is in element of the tangent space
of $1\in \P^1$.

To define a homomorphism $\Theta_{\lambda} : \pi_1(E_{\lambda\partial/\partial
q}',\partial/\partial w) \to \cP$ such that the diagram
$$
\xymatrix{
\pi_1(\Pminus,\partial/\partial w) \ar[r]\ar[d]_{\iota_\ast} &
\exp\L(X_0,X_1)^\wedge \ar[d] \cr
\pi_1(E_{\lambda\partial/\partial q}',\partial/\partial w)
\ar[r]^{\Theta_\lambda} & \cP
}
$$
commutes, where the right-hand vertical map is defined by (\ref{eqn:homom}). We
need to give a ``factor of automorphy'' for the identification. This is the
monodromy along a ``path'' from $e^{i\phi}v_\infty$ to $e^{i(\theta-\phi)}v_0$.
In order that $\Theta_\lambda$ be well defined, this factor of automorphy has
to satisfy the cylinder relation.

For $\lambda \in \C^\ast$, define this factor of automorphy for $\Theta_\lambda$
to be
$$
\lambda^{R_0}e^{-T} := e^{\log\lambda R_0} e^{-T}.
$$
That is, the (inverse) monodromy in going from the tangent vector $\lambda v_0$
at $0 \in \P^1$ to the tangent vector $v_\infty$ at $\infty$ is
$\lambda^{R_0}e^{-T}$.

\begin{figure}[!ht]
\epsfig{file=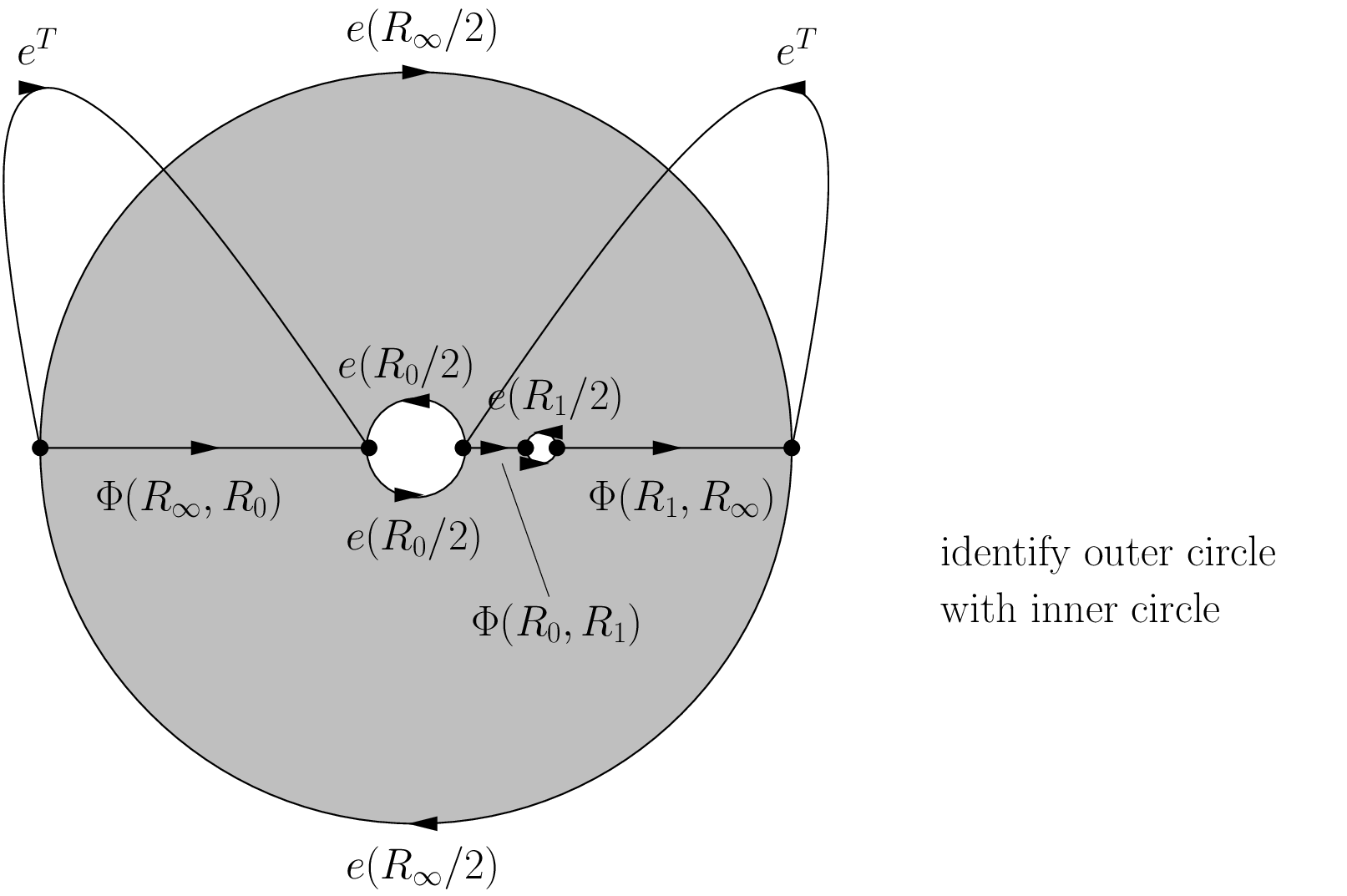, width=3.25in}
\caption{The path torsor of $E_{\partial/\partial q}'$} 
\end{figure}

Give $\C\ll A, T\rr$ the Hodge, weight and relative weight filtrations defined
in Section~\ref{sec:hodge}.

\begin{proposition}
The (complete Hopf algebra) homomorphism
$$
\Theta_\lambda :
\Q\pi_1(E_{\lambda\partial/\partial q}',\partial/\partial w)^\wedge
\to \C\ll T,A\rr
$$
is an isomorphism after tensoring the source with $\C$. This and the Hodge
and weight filtrations on $\C\ll A,T\rr$ defined in Section~\ref{sec:hodge}
define a MHS on $\Q\pi_1(E_{\lambda\partial/\partial q}',\partial/\partial
w)^\wedge$. This is the canonical limit MHS on the fiber of the universal
enveloping algebra of $\bP$ corresponding to the tangent vector
$\lambda\partial/\partial q + \partial/\partial w$ at the identity of the nodal
cubic. Its relative weight filtration is the one defined in
Section~\ref{sec:hodge}.
\end{proposition}

\begin{proof}
Observe that $\Theta_\lambda$ induces a homomorphism
$H_1(E_{\lambda\partial/\partial q}',\Z) \to \C T \oplus \C A$. It takes $\a$ to
$2\pi i R_0 \bmod I^2 = 2\pi i A$ and $\b$ to $\log(\lambda^{R_0} e^{-T}) \bmod
I^2 = \log \lambda A - T$. The homomorphism $\Theta_\lambda$ is an isomorphism
after complexifying as both its source and target are free and as it induces an
isomorphism on $I/I^2$. 

The homomorphism $\Theta_\lambda$ defines a MHS on 
$\Q\pi_1(E_{\lambda\partial/\partial q}',\partial/\partial w)^\wedge$ by pulling
back the Hodge, weight and relative weight filtrations of $\C\ll T,A\rr$. To
check that this is the limit MHS associated to $\lambda\partial/\partial q$, it
suffices to check that the induced MHS on $H_1(E_{\lambda\partial/\partial q})$
by $\Theta_\lambda$ agrees with the canonical limit MHS. This follows from the
discussion above and the fact that the MHS induced on the image of
$\Q\pi_1(\Pminus,\partial/\partial w)^\wedge$ by $\Theta_\lambda$ is its
canonical MHS, which follows from the computations in Section~\ref{sec:Pminus}.
The point being that the limit MHS on $H_1(E_{\lambda\partial/\partial q})$
corresponding to $\lambda\partial/\partial q$ determines the factor of
automorphy. The computations at the beginning of the proof imply that the MHS on
$H_1(E_{\lambda\partial/\partial q})$ induced by $\Theta_\lambda$ agrees with
the limit MHS that was computed in Lemma~\ref{lem:limitMHS}.
\end{proof}


\part{The $\Q$-de~Rham Structure}
\label{part:Q-DR}

Levin and Racinet \cite[\S5]{levin-racinet} sketch an argument to show that the
elliptic KZB connection is defined over $\Q$. This part is an expanded
exposition of a special case of their computation. In particular, we explicitly
compute (Thm.~\ref{thm:QDR}) the restriction of the canonical extension of the
universal elliptic KZB connection to $\M_{1,\vec{1}}$ in terms of the
$\Q$-algebraic coordinates on $\M_{1,\vec{1}}$. When reading
\cite[\S5]{levin-racinet}, it is important to note that that the restriction of
$\bPbar$ to an elliptic curve $E$ is {\em not} algebraically trivial. Levin and
Racinet trivialize the restriction of $\bP$ to $E'$. The $\Q$-connection they
write down is on the corresponding {\em trivial} extension of $\bP|_{E'}$ to
$E$. It does not have a regular singular point at 0. However, it is important in
applications, such as those in \cite{hain-matsumoto:mem}, to know that the
canonical extension $\bPbar$  and its connection are defined over $\Q$. This has
been verified by Ma Luo and will appear in his Duke PhD thesis.

\section{The $\Q$-DR Structure on $\Hbar$ over $\Mbar_{1,\vec{1}}$}

The first step is to compute the $\Q$-DR structure on $\H$ and its canonical
extension $\Hbar$. Since $\M_{1,\vec{1}}$ is the moduli space of elliptic curves
endowed with a non-zero abelian differential, the Hodge bundle $F^1\cH$ is
trivialized by its tautological section. We show that the canonical extension of
$\cH$ over $\Mbar_{1,\vec{1}}$ is trivial and that it and its connection are
defined over $\Q$. Material in this section must surely be well known and
classical (19th C).

\subsection{$\M_{1,\vec{1}}$ as a $\Q$-scheme}

As explained in Section~\ref{sec:univ_curve} (also see \cite{hain:elliptic}),
$\M_{1,\vec{1}}$ is the quotient $\cL_{-1}'$ of $\C\times\h$ by the action of
$\SL_2(\Z)$ which acts with factor of automorphy $(c\tau+d)^{-1}$. It is also
the complement in $\C^2$ of the discriminant locus $\Delta = 0$, where
$$
\Delta = u^3 - 27v^2.
$$
The quotient mapping $\C\times\h \to \C^2 - \Delta^{-1}(0)$ is
$$
(\xi,\tau) \mapsto \big(\xi^4 g_2(\tau),\xi^6 g_3(\tau)\big),
$$
where
$$
g_2(\tau) = 20(2\pi i)^4 G_4(\tau)
\text{ and }
g_3 = \frac{7}{3}(2\pi i)^6 G_6(\tau).
$$
The point $(u,v)$ corresponds to the pair $(E_{u,v},\w_{u,v})$ where $E_{u,v}$
is the elliptic curve $y^2 = 4 x^3 - ux - v$ and $\w_{u,v}$ is the  abelian
differential $dx/y$. This elliptic curve has discriminant (divided by 16) equal
to
$$
\Delta := u^3 - 27 v^2 = \xi^{12}(g_2^3-27g_3^2)
= (2\pi i\xi)^{12} \Delta_0,
$$
where $\Delta_0 = q\prod_{n\ge 1}(1-q^n)^{24}$ is the Ramanujan $\tau$-function.
We will view $\M_{1,\vec{1}}$ as the $\Q$-scheme $\Spec\Q[u,v,\Delta^{-1}]$.

\subsection{Trivializing $\cH$ over $\M_{1,\vec{1}}$}
\label{sec:trivializing}

To trivialize $\cH$, we need two linearly independent sections. The first is
given by the abelian differential $dx/y$. The second by $xdx/y$, a differential
of the second kind.

Set
$$
\eta_\tau = \wp(z,\tau)dz = \bigg(1 + 2\sum_{m=1}^\infty
\frac{G_{2m+2}(\tau)}{(2m)!}(2\pi i z)^{2m+2}\bigg)\frac{dz}{z^2}.
$$
This is a differential of the second kind on $E_\tau$.

\begin{proposition}
If $\gamma\in \SL_2(\Z)$, then $\eta_{\gamma\tau} = (c\tau + d)\eta_\tau$ and
$$
\int_{E_\tau} \w_\tau \smile \eta_\tau = 2\pi i.
$$
In particular, $H^1(E_\tau;\C) = \C\w_\tau \oplus \C\eta_\tau$ for all $\tau$.
\end{proposition}

\begin{proof}
The first assertion follows easily from the definition of $\eta_\tau$. The
second formula follows from a routine residue computation:

Choose a closed disk $D = \{z: |z|\le R\}$ in $E_\tau$ about the origin. Let $F$
be a holomorphic function on $D$ satisfying $F'(z)=\wp(z)$.  Let $\varphi :
E_\tau \to \R$ be a smooth function that vanishes outside the annulus $A = \{z :
R/3 < |z| < R/2\}$ and is identically 1 when $|z|<R/3$. The $1$-form
$$
\psi := \eta_\tau - d(\varphi F(z))
$$
is smooth and closed. Since it agrees with $\eta_\tau$ outside $A$, it has the
same periods as $\eta_\tau$ and thus represents the same cohomology class. Since
$\w_\tau \wedge \psi$ is supported in $D$, we have
\begin{multline*}
\langle \w_\tau \smile \eta_\tau, E_\tau \rangle
= \int_{E_\tau} \w_\tau \wedge \psi
= \int_D dz\wedge \psi
= \int_{\partial D} z\psi
= \int_{\partial D} z\wp(z)dz
= 2\pi i.
\end{multline*}
\end{proof}

\begin{remark}
\label{rem:splitting}
This implies that the exact sequence
$$
0 \to \cL \to \cH \to \cL_{-1} \to 0
$$
over $\M_{1,1}$ splits; the copy of $\cL_{-1}$ in $\cH$ is spanned locally by
$\eta_\tau$. We will see below that this sequence also splits over
$\Mbar_{1,1}$. This splitting also follows from the vanishing of
$H^1(\Mbar_{1,1},\cL_2)$ as there are no modular forms of weight $2$ and level
$1$.
\end{remark}

\begin{corollary}
The sections $\xi^{-1}\w_\tau$ and $\xi\eta_\tau$ of $\cH$ over $\C\times\h$
are $\SL_2(\Z)$-invariant. \qed
\end{corollary}

For a lattice $\Lambda$ in $\C$, set
$$
\wp_\Lambda(z) := \frac{1}{z^2} +
\sum_{\substack{\lambda \in \Lambda\cr \lambda \neq 0}}
\bigg[\frac{1}{(z-\lambda)^2} - \frac{1}{\lambda^2} \bigg].
$$
The Weierstrass $\wp$-function $\wp(z,\tau)$ defined in Section~\ref{sec:wp} is
$\wp_{\Lambda_\tau}(z)$. One checks easily that
$$
\wp_{\xi^{-1}\Lambda}(\xi^{-1}z) = \xi^2\wp_\Lambda(z).
$$

Multiplication by $\xi^{-1}$ induces an isomorphism $E_\tau \to
\C/\xi^{-1}\Lambda_\tau$ under which $dz$ and $\wp_{\xi^{-1}\Lambda_\tau}dz$
pull back to $\xi^{-1}\w_\tau$ and $\xi\eta_\tau$, respectively.

\begin{proposition}
\label{prop:diffs}
If $(u,v) = (\xi^4 g_2(\tau),\xi^6 g_3(\tau)\big)$, then
\begin{enumerate}

\item the map
$$
z\mapsto
\big[\wp_{\xi^{-1}\Lambda_\tau}(z),\wp_{\xi^{-1}\Lambda_\tau}'(z),1\big]
$$
from $\C/\xi^{-1}\Lambda_\tau$ to $\P^2$ induces an isomorphism
$\C/\xi^{-1}\Lambda_\tau \to E_{u,v}$;

\item under this isomorphism
$$
dx/y = dz = \xi^{-1}\w_\tau \text{ and }
xdx/y = \wp_{\xi^{-1}\Lambda_\tau}(z)dz;
$$

\item under the isomorphism $E_\tau \to \C/\xi^{-1}\Lambda_\tau \to E_{u,v}$,
$dx/y$ and $xdx/y$ pull back to $\xi^{-1}\w_\tau$ and $\xi\eta_\tau$,
respectively.

\end{enumerate}\qed
\end{proposition}


\begin{corollary}
For each $(u,v) \in \M_{1,\vec{1}}$, the elements $dx/y$ and $xdx/y$ of the
fiber $H^1(E_{u,v})$ of $\cH$ are linearly independent. \qed
\end{corollary}

Denote these sections of $\cH$ over $\M_{1,\vec{1}}$ by $\That$ and $\Shat$,
respectively. They trivialize $\cH$ and determine the extension
$$
\Hbar := \O_{\Mbar_{1,\vec{1}}}\Shat \oplus \O_{\Mbar_{1,\vec{1}}}\That
$$
of $\cH$ to $\Mbar_{1,\vec{1}}$.

\begin{proposition}
\label{prop:nabla_0}
The connection on $\Hbar$ with respect to this trivialization is
$$
\nabla_0 = d +
\Big(
-\frac{1}{12}\frac{d\Delta}{\Delta}\otimes\That
+ \frac{3}{2}\frac{\alpha}{\Delta}\otimes\Shat
\Big)\frac{\partial}{\partial\That}
+
\Big(
-\frac{u}{8}\frac{\alpha}{\Delta}\otimes\That
+ \frac{1}{12}\frac{d\Delta}{\Delta}\otimes\Shat
\Big)\frac{\partial}{\partial\Shat}
$$
where $\alpha = 2udv-3vdu$ and $\Delta = u^3 - 27 v^2$. This form is logarithmic
on $\Mbar_{1,\vec{1}}$ with nilpotent residue along $\Delta=0$ and is therefore
the canonical extension of $\cH$ over $\Mbar_{1,\vec{1}}$. It is defined over
$\Q$.
\end{proposition}

\begin{proof}[Sketch of Proof]
We need to understand how the classes $dx/y$ and $xdx/y$ depend on $(u,v)$.
Each of the $1$-forms
$$
\frac{\partial}{\partial u}\Big(\frac{dx}{y}\Big)=\frac{1}{2}\frac{xdx}{y^3},\
\frac{\partial}{\partial v}\Big(\frac{dx}{y}\Big)=\frac{1}{2}\frac{dx}{y^3},\
\frac{\partial}{\partial u}\Big(\frac{xdx}{y}\Big)=\frac{1}{2}\frac{x^2dx}{y^3}
,\ \frac{\partial}{\partial v}\Big(\frac{xdx}{y}\Big)=\frac{1}{2}\frac{xdx}{y^3}
$$
is a differential of the second kind on each $E_{u,v}$. So the cohomology
class of each is a linear combination of the classes of $dx/y$ and $xdx/y$.

The differentials $d(1/y),\ d(x/y)\text{ and }d(x^2/y)$, the relation $2ydy =
(12x^2-u)dx$, and some linear algebra give
$$
\begin{pmatrix}
\frac{dx}{y^3} & \frac{xdx}{y^3} & \frac{x^2dx}{y^3}
\end{pmatrix}
\equiv
\frac{3}{\Delta}
\begin{pmatrix}
\frac{dx}{y} & \frac{xdx}{y}
\end{pmatrix}
\begin{pmatrix}
3v & - u^2/6 & uv/4 \cr 2u & -3v & u^2/6
\end{pmatrix}
$$
where $\equiv$ means congruent mod exact forms of the second kind, and thus
equal in cohomology.

Now
\begin{align*}
\nabla_0
\begin{pmatrix}
\That &\Shat
\end{pmatrix}
&= \nabla_0
\begin{pmatrix}
\frac{dx}{y} & \frac{xdx}{y}
\end{pmatrix}
\cr
&=
\frac{1}{2}
\begin{pmatrix}
\frac{xdx}{y^3} & \frac{x^2dx}{y^3}
\end{pmatrix}
du
+
\frac{1}{2}
\begin{pmatrix}
\frac{dx}{y^3} & \frac{xdx}{y^3}
\end{pmatrix}
dv
\cr
&=
\frac{3}{2\Delta}
\begin{pmatrix}
\frac{dx}{y} & \frac{xdx}{y}
\end{pmatrix}
\begin{pmatrix}
-u^2du/6+3vdv & uvdu/4-u^6dv/6 \cr -3vdu+2udv & u^2du/6-3vdv
\end{pmatrix}
\cr
&=
\begin{pmatrix}
\That &\Shat
\end{pmatrix}
\begin{pmatrix}
-\frac{1}{12}\frac{d\Delta}{\Delta} & -\frac{u\alpha}{8\Delta} \cr
\frac{3\alpha}{2\Delta} & \frac{1}{12}\frac{d\Delta}{\Delta}
\end{pmatrix}
\end{align*}

The forms $\alpha/\Delta$ and $u\alpha/\Delta$ are logarithmic. One can prove
this directly. Alternatively, we can use the fact \cite{deligne:ode} that a
meromorphic form $\varphi$ on $\C^2$ has logarithmic singularities along
$\Delta=0$ if and only if $\Delta\varphi$ and $\Delta d\varphi$ are both
holomorphic along $\Delta=0$. This holds in our case as
$$
\Delta d(\alpha/\Delta) = - du\wedge dv \text{ and }
\Delta d(u\alpha/\Delta) = u du\wedge dv.
$$
Similarly, one checks that $w\varphi$ and $wd\varphi$ are holomorphic along the
line at infinity, where $w=0$ is a local defining equation of the line at
infinity, then $\varphi$ is logarithmic along the line at infinity. This is
easily checked when $\varphi$ is $\alpha/\Delta$ and $u\alpha/\Delta$.
\end{proof}

This implies that the sequence
$$
0 \to \cLbar \to \Hbar \to \cLbar_{-1} \to 0
$$
splits over $\Mbar_{1,\vec{1}}$. The lift of $\cLbar_{-1}$ in $\Hbar$ is
$\O_{\Mbar_{1,\vec{1}}}\Shat$.

\subsection{Transcendental version}

We re-derive the formula for the connection in terms of the coordinates
$(\xi,\tau)\in \C\times\h$. This will yield some formulas that are useful
in computing the algebraic version of the universal elliptic KZB connection.

As observed above, if $(u,v) = (\xi^4 g_2(\tau),\xi^6 g_3(\tau))$, then
$dx/y = \xi^{-1}\w_\tau$. Since $\That$ is the class of $dx/y$, we have
$$
\That = \xi^{-1} T.
$$
\begin{proposition}
We have $\eta_\tau = (2\pi i)^2 (A - 2G_2(\tau)T)$ so that
$$
\Shat = \xi\eta_\tau = (2\pi i)^2 \xi (A - 2G_2(\tau)T).
$$
\end{proposition}

\begin{proof}
Using the notation of Example~\ref{ex:hodge}, we have
\begin{align*}
\begin{pmatrix}
\a & \t
\end{pmatrix}
\begin{pmatrix}
1 \cr 8 \pi^2 G_2(\tau)
\end{pmatrix}
&= 
\begin{pmatrix}
\a' & \t'
\end{pmatrix}
\begin{pmatrix}
(c\tau + d)^{-1} & 0 \cr 2\pi i c & c\tau + d
\end{pmatrix}
\begin{pmatrix}
1 \cr 8 \pi^2 G_2(\tau)
\end{pmatrix}
\cr
&=
(c\tau + d)^{-1}
\begin{pmatrix}
\a' & \t'
\end{pmatrix}
\begin{pmatrix}
1 \cr 8\pi^2 G_2(\gamma \tau)
\end{pmatrix}.
\end{align*}
From this it follows that
$$
A - 2G_2(\tau)T = (c\tau+d)^{-1}\big(A'- 2G_2(\gamma\tau)T').
$$
Consequently, $A - 2G_2(\tau)T$ is a section of $\cH$ that spans a copy of
$\cL_{-1}$ over $\h$. But $\eta_\tau$ is another such section. It follows that
$\eta_\tau$ is a holomorphic multiple of $A - 2G_2(\tau)T$. This multiple can be
determined by pairing with $\w_\tau$. Since
$$
\langle T,A - 2G_2(\tau)T\rangle = \langle \w_\tau, (2\pi i)^{-1}\a\rangle
= (2\pi i)^{-1} \text{ and }
\langle T,\eta_\tau\rangle = \int_{E_\tau}\w_\tau \smile \eta_\tau = 2\pi i,
$$
it follows that
$$
\eta_\tau = (2\pi i)^2 (A-2G_2(\tau)T).
$$
There are several ways to prove the second assertion. One is to observe that
$$
\langle \That, \Shat \rangle =
\langle \frac{dx}{y}, \frac{xdx}{y}\rangle = 2\pi i = \langle T,S\rangle
= \langle \xi^{-1}T,\xi S\rangle = \langle \That,\xi S\rangle.
$$
\end{proof}

We've already seen in Example~\ref{ex:connection_H} that the connection
$\nabla_0$ on $\cH$ over $\h$ with respect to the framing $A$, $T$ is given by
$$
\nabla_0 = d + 2\pi i A\frac{\partial}{\partial T}\otimes d\tau.
$$

\begin{proposition}
With respect to the framing $\Shat$ and $\That$ of $\Hbar$ over
$\M_{1,\vec{1}}$, the connection on $\cH$ is
\begin{multline*}
\nabla_0 = d +
\Big(
\big(4\pi i\, G_2\,d\tau - \frac{d\xi}{\xi}\big)\otimes\That
+\frac{2\pi i}{(2\pi i \xi)^2}d\tau\otimes\Shat
\Big)\frac{\partial}{\partial \That}
\cr
-
\Big(
(2\pi i\xi)^2(8\pi i\, G_2^2 + 2 G_2')\,d\tau\otimes\That
+\big(4\pi i\, G_2\,d\tau - \frac{d\xi}{\xi}\big)\otimes\Shat
\Big)\frac{\partial}{\partial \Shat}
\end{multline*}
\end{proposition}

\begin{proof}
Since $S/(2\pi i)^2 = A - 2G_2 T$,
\begin{align*}
\nabla_0
\begin{pmatrix}
T & S/(2\pi i)^2
\end{pmatrix}
&=
\nabla_0
\begin{pmatrix}
T & A
\end{pmatrix}
\begin{pmatrix}
1 & -2G_2 \cr 0 & 1
\end{pmatrix}
+
\begin{pmatrix}
T & A
\end{pmatrix}
\begin{pmatrix}
0 & -2G_2' \cr 0 & 0
\end{pmatrix}
d\tau
\cr
&=
\begin{pmatrix}
T & A
\end{pmatrix}
\bigg(
\begin{pmatrix}
0 & 0 \cr 2\pi i & 0
\end{pmatrix}
\begin{pmatrix}
1 & -2G_2 \cr 0 & 1
\end{pmatrix}
+
\begin{pmatrix}
0 & -2G_2' \cr 0 & 0
\end{pmatrix}
\bigg) d\tau
\cr
&=
\begin{pmatrix}
T & S/(2\pi i)^2
\end{pmatrix}
\begin{pmatrix}
1 & 2G_2 \cr 0 & 1
\end{pmatrix}
\begin{pmatrix}
0 & -2G_2' \cr 2\pi i & -4\pi i G_2
\end{pmatrix}
d\tau
\cr
&=
\begin{pmatrix}
T & S/(2\pi i)^2
\end{pmatrix}
\begin{pmatrix}
4\pi i G_2 & -(8\pi i G_2^2 + 2G_2') \cr 2\pi i & -4\pi i G_2
\end{pmatrix}
d\tau
\end{align*}
Rescaling, we have
$$
\nabla_0
\begin{pmatrix}
T & S
\end{pmatrix}
=
\begin{pmatrix}
T & S
\end{pmatrix}
\begin{pmatrix}
4\pi i G_2 & -(2\pi i)^2(8\pi i G_2^2 + 2G_2') \cr (2\pi i)^{-1} & -4\pi i G_2
\end{pmatrix}
d\tau
$$
Denote the $2\times 2$ matrix of 1-forms in this expression by $Bd\tau$. Then
\begin{align*}
\nabla_0
\begin{pmatrix}
\That & \Shat
\end{pmatrix}
&=
\nabla_0
\begin{pmatrix}
T & S
\end{pmatrix}
\begin{pmatrix}
\xi^{-1} & 0 \cr 0 & \xi
\end{pmatrix}
+
\begin{pmatrix}
T & S
\end{pmatrix}
\begin{pmatrix}
-\xi^{-2} & 0 \cr 0 & 1
\end{pmatrix}
d\xi
\cr
&=
\begin{pmatrix}
\That & \Shat
\end{pmatrix}
\bigg[
\begin{pmatrix}
\xi & 0 \cr 0 & \xi^{-1}
\end{pmatrix}
B
\begin{pmatrix}
\xi^{-1} & 0 \cr 0 & \xi
\end{pmatrix}
d\tau
+
\begin{pmatrix}
\xi & 0 \cr 0 & \xi^{-1}
\end{pmatrix}
\begin{pmatrix}
-\xi^{-2} & 0 \cr 0 & 1
\end{pmatrix}
d\xi
\bigg]
\cr
&=
\begin{pmatrix}
\That & \Shat
\end{pmatrix}
\begin{pmatrix}
-\frac{d\xi}{\xi} + 4\pi i G_2(\tau)d\tau &
-2(2\pi i\xi)^2(4\pi i G_2(\tau)^2 + G_2'(\tau))d\tau \cr
\frac{2\pi i}{(2\pi i \xi)^2}d\tau &
\frac{d\xi}{\xi} - 4\pi i G_2(\tau)d\tau
\end{pmatrix}
\end{align*}
\end{proof}

Comparing this formula with that in Proposition~\ref{prop:nabla_0}, we conclude:
\begin{equation}
\label{eqn:identities}
\frac{d\xi}{\xi} - 4\pi i G_2(\tau)d\tau = \frac{1}{12}\frac{d\Delta}{\Delta}
\text{ and }
\frac{2\pi i}{(2\pi i \xi)^2}d\tau = \frac{3\alpha}{2\Delta}.
\end{equation}.
Taking the quotient of the two off diagonal entries of the connection matrix,
we conclude that
$$
G_4(\tau) = \frac{6}{5}\big(2 G_2(\tau)^2 + G_2'(\tau)/2\pi i\big).
$$
This can also be verified by observing that the RHS is a modular form of
weight 4 and then computing value of both sides at $q=0$.

\section{The $\Q$-DR Structure on $\bPbar$ over $\Mbar_{1,\vec{1}}$}

The bundle $\bPbar$ over $\Mbar_{1,\vec{1}}$ is the trivial bundle whose fiber
is $\L(\Shat,\That)^\wedge$. We will define a $\Q$ structure on it --- that is,
a $\Q$ structure on its truncations by the terms of its lower central series.

Some preliminary observations will be helpful. Since the cup product of the
rational differentials $dx/y$ and $xdx/y$ is $2\pi i$, it is natural to multiply
their Poincar\'e duals by $(2\pi i)^{-1}$ to obtain a $\Q$-de Rham basis of the
first homology. Motivated by this, we define
$$
\That_0 = \That/2\pi i \text{ and } \Shat_0 = \Shat/2\pi i.
$$
Since both basis elements of $\Hbar$ are multiplied by the same constant, the
formula for the connection on $\Hbar$ given in Proposition~\ref{prop:nabla_0}
remains valid when we replace $\Shat$ by $\Shat_0$ and $\That$ by $\That_0$.

Set $\p_\Q = \L_\Q(\Shat_0,\That_0)^\wedge$. Define the $\Q$-structure on
$\bPbar$ to be $\Mbar_{1,\vec{1}/\Q}\times \p_\Q$.

Define derivations $\ehat_{2m}$ of $\p_\Q$ by
$$
\ehat_{2m} =
\begin{cases}
-\Shat_0\frac{\partial}{\partial \That_0} & m = 0; \cr
\That_0^{2m-1}\cdot\Shat_0 -
\sum_{\substack{j+k=2m-1\cr j>k > 0}}(-1)^j
[\That_0^j\cdot\Shat_0,\That_0^k\cdot\Shat_0]
\frac{\partial}{\partial \Shat_0} & m > 0.
\end{cases}
$$

\begin{lemma}
For all $m\ge 0$, we have $(2\pi i\xi)^{2m-2}\ehat_{2m} = \d_{2m}$ in $\Der\p$.
\end{lemma}

\begin{proof}
First observe that
$$
\d_{2m}(T) = -T^{2m}\cdot A \text{ and } \d_{2m}(A) =
\sum_{\substack{j>k\ge 0\cr j+k = 2m-1}}(-1)^{j+1}[T^j\cdot A,T^k\cdot A]
$$
and 
$$
\ehat_{2m}(\That_0) = -\That_0^{2m}\cdot \Shat_0 \text{ and }
\ehat_{2m}(\Shat_0) = \sum_{\substack{j>k\ge 0\cr j+k= 2m-1}}
(-1)^{j+1}[\That_0^j\cdot \Shat_0,\That_0^k\cdot \Shat_0].
$$
One checks easily that, when $m\ge 1$, $\d_{2m}(T) = -T^{2m}\cdot (A-2G_2T)$
and
$$
\d_{2m}(A-2G_2T) =
\sum_{\substack{j>k\ge 0\cr j+k = 2m-1}}
(-1)^{j+1}[T^j\cdot (A-2G_2T),T^k\cdot (A-2G_2T)].
$$
The result follows by rescaling as $\That_0 = T/(2\pi i\xi)$ and $\Shat_0 = 2\pi
i\xi(A-2G_2T)$.
\end{proof}

The connection $\nabla_0$ on $\cH$ defines, and will be viewed as, a
$\Q$-rational connection on each graded quotient of $\bP$.

\begin{theorem}
\label{thm:QDR}
With respect to the framing of $\bP$ over $\M_{1,\vec{1}}$ described above, the
universal elliptic KZB-connection $\nabla$ is given by
$$
\nabla = \nabla_0 + \frac{1}{12}\frac{d\Delta}{\Delta}\otimes \ehat_2
+\sum_{m\ge 2}\frac{3}{(2m-2)!} \frac{p_{2m}(u,v)(3vdu - 2udv)}{\Delta}
\otimes \ehat_{2m}
$$
where $\Delta = u^3 - 27 v^2$ is the discriminant and where $p_{2m}(u,v)\in
\Q[u,v]$ is the polynomial characterized by $(2\pi i\xi)^{2m} G_{2m}(\tau) =
p_{2m}(u,v)$. The Hodge bundles $F^p\bP$ are all defined over $\Q$.
\end{theorem}

The polynomial $p_{2m}(u,v)$ is weighted homogeneous of weight $2m$ in $u$ and
$v$, where $u$ is given weight 4 and $v$ is given weight 6. The polynomials
of weight up to 24 are:
\begin{align*}
 p_{4}(u,v) &= \frac{1}{20}u \cr
 p_{6}(u,v) &= \frac{3}{7} v \cr
 p_{8}(u,v) &= \frac{3}{10} u^2 \cr
 p_{10}(u,v) &= \frac{108}{11} uv \cr  
 p_{12}(u,v) &= \frac{756}{65} u^3 + \frac{16200}{91} v^2 \cr
 p_{14}(u,v) &= 1296\, uv^2 \cr
 p_{16}(u,v) &= \frac{174636}{85} u^4 + \frac{1166400}{17} uv^2 \cr     
 p_{18}(u,v) &= \frac{9471168}{19} u^3v + \frac{256608000}{133} v^3 \cr   
 p_{20}(u,v) &= \frac{25147584}{25} u^5 + \frac{678844800}{11} u^2 v^2 \cr
 p_{22}(u,v) &= \frac{10671720192}{23} u^4 v + \frac{103296384000}{23} u v^3\cr
 p_{24}(u,v) &= \frac{73581830784}{65} u^6 + \frac{1410877440000}{13} u^3 v^2
 + \frac{15547365504000}{91} v^4    
\end{align*}

\begin{proof}
With respect to the framing $A$, $T$ of $\bP$, the connection is $\nabla = d +
\w'$ where $\w'$ is the form (\ref{eqn:conn}). Since the change of frame is
homogeneous, the transformed connection is of the form $\nabla = \nabla_0 +
\w'$. We just need to express $\w'$ in the frame given by Lie words in
$\Shat_0$, $\That_0$. Using the identities (\ref{eqn:identities}) and the
preceding lemma we have
\begin{align*}
\w' &=
-
\bigg(2G_2(\tau)\frac{dq}{q} - \frac{d\xi}{\xi}\bigg) \otimes \d_2
- \sum_{m=2}^\infty \frac{2}{(2m-2)!} G_{2m}(\tau)\frac{dq}{q}\otimes\d_{2m}
\cr
&= \frac{1}{12}\frac{d\Delta}{\Delta}\otimes\ehat_2 - \sum_{m\ge 2}
\frac{2}{(2m-2)!}(2\pi i \xi)^{2m}G_{2m}(\tau)
\frac{2\pi i}{(2\pi i\xi)^2}\,d\tau\otimes\ehat_{2m}\cr
&= \frac{1}{12}\frac{d\Delta}{\Delta}\otimes\ehat_2 - \sum_{m\ge 2}
\frac{3}{(2m-2)!}\frac{p_{2m}(u,v)\alpha}{\Delta}\otimes\ehat_{2m}
\end{align*}
The last assertion follows from the fact that $F^p\bP$ is trivial and
consists of those Lie words whose degree in $\That_0$ is $\ge -p$.
\end{proof}

\section{The $\Q$-de~Rham Structure on
$F^{2n+1}H^1(\M_{1,1},S^{2n}\cH)$}

Here we compute the $\Q$-structure on $F^{2n+1}H^1_\dR(\M_{1,1},S^{2n}\cH)$. The
computation of the $\R$-de~Rham structure on all of $H^1(\M_{1,1},S^{2n}\cH)$
can be found in \cite[\S17.2]{hain:vancouver}.

The starting point is the isomorphisms
$$
H^0(\Mbar_{1,1}, \cL_{2n+2}) \to 
H^0(\Omega^1_{\Mbar_{1,1}}(P)\otimes F^{2n}S^{2n}\Hbar)
\to F^{2n+1}H^1(\M_{1,1},S^{2n}\cH),
$$
where $P$ denotes the cusp $q=0$. The second isomorphism takes a 1-form to its
cohomology class. The first follows from the isomorphisms $\cL\cong F^1\Hbar$
and $\Omega^1_{\Mbar_{1,1}}(P) \cong \cL_2$, which together induce an
isomorphism
$$
\cL_{2n+2} \cong \Omega^1_{\Mbar_{1,1}}(P)\otimes F^{2n}S^{2n}\Hbar.
$$

We now explain the $\Q$-structure. Recall that if $f$ is a modular form of
weight $2n+2$ of $\SL_2(\Z)$, then 
$$
\w_f := f(\tau)\bw^{2n}d\tau \in E^1(\h)\otimes S^{2n}H
$$
is an $\SL_2(\Z)$-invariant 1-form on $\h$, where $\bw := 2\pi i T$ is the
section of $H\times\h \to \h$ that takes the value $2\pi i \w_\tau$ at
$\tau\in\h$ and $\SL_2(\Z)$ acts on $H=\C A \oplus \C T$ via the factor of
automorphy (\ref{eqn:automorphy}).\footnote{Recall the definitions
(\ref{eqn:def_AT}).} It gives a framing of $F^1\cH$ over $\h$. The section
$\bw$ extends to a framing of $F^1\Hbar$ over the $q$-disk.

Denote the space of modular forms of $\SL_2(\Z)$ of weight $m$ whose Fourier
coefficients lie in the subfield $F$ of $\C$ by $M_{m,F}$. These form a graded
ring $M_{\ast,F}$ isomorphic to $F[G_4,G_6]$.

\begin{proposition}
The $\Q$-structure on $F^{2n+1}H^1_\dR(\M_{1,1},S^{2n}\cH)$ is
$$
\{f(q)\bw^{2n}\frac{dq}{q} : f(q)  \in M_{2n+2,\Q}\}.
$$
\end{proposition}

\begin{proof}
Embed $E_\tau$ into $\P^2$ via the mapping
$$
z + \Lambda_\tau \mapsto [\wp_\tau(z)/(2\pi i)^2,\wp_\tau'(z)/(2\pi i)^3,1].
$$
The image is the plane cubic $y^2 = 4x^3 -ux -v$ where
$$
u=g_2(\tau)/(2\pi i)^4 = 20 G_4(\tau) \text{ and }
v=g_3(\tau)/(2\pi i)^6 = \frac{7}{3}G_6(\tau).
$$
This curve has discriminant $\Delta_0(\tau)$, where $\Delta_0$ denotes the
normalized cusp form of weight 12.

With this normalization $dx/y = 2\pi i \w_\tau=\bw(\tau)$.\footnote{This would
have been a better normalization to use in Part~\ref{part:Q-DR}.} We choose it
because the value of the section $\bw$ at $q=0$ is $dw/w$, where $w$ is the
parameter on the nodal cubic that maps $0$ and $\infty$ to the node and $1$ to
the identity. (Cf.\ Exercise~47 in \cite{hain:elliptic}.)

We regard $\bw$ as the section $dx/y$ of $F^1\cH$ over $\M_{1,\vec{1}}= \A^2_\Q
- \{u^3-27v^2=0\}$. It is
defined over $\Q$. Since $x$ has weight $2$ and $y$ weight $3$, it has weight
$-1$ under the $\Gm$ action. If $h(u,v) \in \Q[u,v]$ is a polynomial of weight
$2n+2$ (where $u$ has weight $4$ and $v$ weight $6$), then
$$
\frac{h(u,v)}{u^3-27v^2}\, (2udv-3vdu)\bw^{2n} 
$$
is a $\Q$-rational, $\Gm$-invariant section of
$\Omega^1_{\M_{1,\vec{1}/\Q}}(\log D)\otimes F^{2n}S^{2n}\cH$ over
$\M_{1,\vec{1}/\Q}$, where $D$ denotes the discriminant locus $u^3-27v^2=0$.
Since $\M_{1,1/\Q} = \Gm \bbs \M_{1,\vec{1}/\Q}$ it descends to a section of
$\Omega^1_{\M_{1,1/\Q}}(P)\otimes F^{2n}S^{2n}\cH$.

The identity (\ref{eqn:identities}) implies that the pullback of this form along
the map $\h \to \M_{1,\vec{1}}$ defined by $\tau \mapsto (20
G_4(\tau),7G_6(\tau)/3)$ is
$$
\frac{2}{3} h\big(20 G_4(\tau),7G_6(\tau)/3\big) \bw^{2n} \frac{dq}{q}.
$$
The result follows as $M_{2n+2,\Q}$ is isomorphic to the polynomials in $G_4$
and $G_6$ with rational coefficients.
\end{proof}

\appendix

\section{Vanishing of $N(R_0)$}
\label{sec:identity}

Here we give an elementary proof of the vanishing of $N_q(R_0)$ established in
Corollary~\ref{cor:N(R)}. It does not use limit mixed Hodge structures. Instead
we deduce the vanishing from an identity involving Bernoulli numbers.

For non-negative integers, define polynomials
\begin{align*}
h_{a,b}(x,y) &= x^{2a-1}y^{2b}-x^{2b}y^{2a-1}+xy(x+y)^{2b-1}(y^{2a-2}-x^{2a-2})
\cr
&= xy(y-x)\Big(x^{2a-2}y^{2b-2}+(x+y)^{2b-1}
\sum_{\substack{i+j=2a-3\cr i,j\ge 0}} x^iy^j\Big).
\end{align*}
in commuting indeterminants $x$ and $y$. Note that
$$
h_{0,n}(x,y) = - \sum_{\substack{i+j=2n-1\cr i,j \ge 0}}
\Big[\binom{2n}{i+1}-\binom{2n}{j+1}\Big] x^i y^j
$$
for all $n\ge 1$.

\begin{theorem}
\label{thm:identity}
For all $n\ge 1$,
$$
\sum_{\substack{a+b=n\cr a>0}} (2a-1)\binom{2n}{2a}\frac{B_{2a}B_{2b}}{B_{2n}}
\, h_{a,b}(x,y)
= \sum_{\substack{i+j=2n-1\cr i,j \ge 0}}
\Big[\binom{2n}{i+1}-\binom{2n}{j+1}\Big] x^i y^j \in \Z[x,y].
$$
Equivalently,
$$
\sum_{\substack{a+b=n\cr a,b\ge 0}}
(2a-1)\frac{B_{2a}}{(2a)!}\frac{B_{2b}}{(2b)!} \, h_{a,b}(x,y) = 0.
$$
\end{theorem}

\begin{proof}
It suffices to show that
\begin{equation}
\label{eqn:series}
\sum_{n\ge 0} \sum_{\substack{a+b=n\cr a,b\ge 0}}
(2a-1)\frac{B_{2a}}{(2a)!}\frac{B_{2b}}{(2b)!} \, h_{a,b}(x,y) = 0.
\end{equation} 
Observe that
\begin{align*}
\sum_{n\ge 0} \sum_{\substack{a+b=n\cr a,b\ge 0}}
(2a-1)\frac{B_{2a}}{(2a)!}\frac{B_{2b}}{(2b)!} \, u^{2a-1}v^{2b}
&= \bigg(\sum_{a\ge 0}(2a-1)\frac{B_{2a}}{(2a)!} u^{2a-1}\bigg)
\bigg(\sum_{b\ge 0}\frac{B_{2b}}{(2b)!} v^{2b}\bigg)
\cr
&= \frac{4u}{\sinh^2(u/2)}\bigg(\frac{v}{e^v-1}+\frac{v}{2}\bigg)
\cr
&= \frac{u e^u}{(e^u-1)^2}\bigg(\frac{v}{e^v-1}+\frac{v}{2}\bigg)
.
\end{align*}
Denote this function of $(u,v)$ by $F(u,v)$. The series (\ref{eqn:series})
is then
$$
F(x,y)-F(y,x)+\frac{x}{x+y}F(y,x+y)-\frac{y}{x+y}F(x,x+y)
$$
which is easily to vanish by elementary algebraic manipulations.
\end{proof}

We finish by showing that this identity is equivalent to the vanishing of
$N_q(R_0)$. For this we need to relate polynomials to the free Lie algebra
$\L(T,A)$. For this we use the Levin-Racinet calculus \cite[\S
3.1]{levin-racinet}. Recall from Section~\ref{sec:calc} that for $U,V \in
\L(A,T)$,
$$
x^r y^s \circ (U,V) = [T^r\cdot U, T^s\cdot V].
$$
This extends linearly to an action $f(x,y)\circ (U,V)$ of polynomials $f(x,y)$
in commuting indeterminants on ordered pairs of elements of $\L(T,A)$. When $U$
and $V$ are equal, one has the identity $f(x,y)\circ (U,U) = -f(y,x)\circ
(U,U)$, so that
$$
2f(x,y)\circ (U,U) = (f(x,y)-f(y,x))\circ(U,U).
$$
In this case we need only consider polynomials $f(x,y)$ satisfying
$f(x,y)+f(y,x) = 0$.

The significance of the polynomials $h_{a,b}(x,y)$ is given by:

\begin{lemma}
For all $a\ge 0$ and $b\ge 0$ with $a+b>0$,
$$
2\,\d_{2a}(T^{2b}\cdot A) = h_{a,b}(x,y)\circ (A,A).
$$
\end{lemma}

\begin{proof}
Observe that $h_{a,b}(x,y) = f_{a,b}(x,y) - f_{a,b}(y,x)$ where
$$
f_{a,b}(x,y) = 
x^a(x+y)^{2b-1}\big(x^a-(-y)^a\big) - x^{2a-1}\big((x+y)^{2b}-y^{2b}\big).
$$
The result now follows from the easily verified identity
$$
\d_{2a}(T^{2b}\cdot A) = f_{a,b}(x,y)\circ (A,A).
$$
\end{proof}

Theorem~\ref{thm:identity} implies the vanishing of $N_q(R_0)$:
\begin{align*}
N_q(R_0)
&= \sum_{a\ge 0}\sum_{b\ge 0}(2a-1)
\frac{B_{2a}}{(2a)!}\frac{B_{2b}}{(2b)!}\, \d_{2a}(T^{2b}\cdot A)
\cr
&= \frac{1}{2}\Big(\sum_{a\ge 0}\sum_{b\ge 0}(2a-1)
\frac{B_{2a}}{(2a)!}\frac{B_{2b}}{(2b)!}\, h_{a,b}(x,y)\Big)\circ(A,A)
\cr
&= 0.
\end{align*}

\section{The Universal Elliptic Curve over $\Bl_0^+\D$}
\label{sec:blowup}

Here we justify the claim, made in Section~\ref{sec:pause}, that the fiber of
the universal elliptic curve over $e^{i\theta}\partial/\partial q$ is obtained
from the real oriented blowup of $\P^1$ at $\{0,\infty\}$ by identifying its
two boundary components with a suitable twist.


The map $[0,1) \times S^1 \to \D$ that takes $(r,\theta)$ to $e^{i\theta}$ is
the real oriented blowup $\Bl^o_0 \D \to \D$ of the disk. More generally, the
map
$$
S^1\times [0,1] \to \P^1 \text{ defined by } (\phi,t) \mapsto [te^{i\phi},1-t]
$$
is $\Bl^o_{0,\infty}\P^1 \to \P^1$. With this identification, the inclusion
$\C^\ast \to \Bl^o_{0,\infty}\P^1$ takes $se^{i\phi}$ to
$\big(\phi,s/(1+s)\big)$.


The fiber of the universal elliptic curve over $q = re^{i\theta}$ is the
quotient of
$$
A := \{(w,q)\in \C^\ast\times \D^\ast : \sqrt{|q|}\le |w| \le 1/\sqrt{|q|}\}
$$
obtained by glueing $w$ to $qw$ when $|w| = 1/\sqrt{|q|}$. Write $w=s e^{i\phi}$
so that we can identify $A$ with
$$
\{(s,\phi,re^{i\theta}) \in \R\times S^1 \times \D^\ast :
\sqrt{r} \le s \le 1/\sqrt{r}\}.
$$
With this identification, $(1/\sqrt{r},\phi,re^{i\theta})$ is glued to
$(\sqrt{r},\phi+\theta,re^{i\theta})$.

The function
$$
h(r,s) =
\Big(\frac{s}{1+s} - \frac{\sqrt r}{1+\sqrt r}\Big)
\Big(\frac{1}{1+\sqrt r} - \frac{\sqrt r}{1+\sqrt r}\Big)^{-1}
$$
induces homeomorphisms $h(r,\blank) : [\sqrt{r}, 1/\sqrt{r}] \to [0,1]$ for all
$r\ge 0$. It has inverse $k(r,\blank)$, where
$$
k(r,t) = \frac{\sqrt{r}-(\sqrt{r}-1)t}{1+(\sqrt{r}-1)t}
$$

Define an equivalence relation on $B := S^1\times S^1 \times [0,1) \times
[0,1)$
by
$$
(\phi,\theta,1,r) \sim (\phi+\theta,\theta,0,r).
$$
Set $\B = B/\hspace{-2.5pt}\sim$. The map $(\phi,\theta,t,r) \to (r,\theta)$
defines a projection $\pi: \B \to \Bl^o_0 \D$. This is a torus bundle over
$\Bl^o_0 \D$. Its fiber $B_\theta$ over $(0,\theta) \in \Bl^o_0 \D$ is the
quotient of $\Bl^o_{0,\infty}\P^1$ obtained by identifying the two boundary
components by a twist by $\theta$. The inclusion $A \hookrightarrow  B$ defined
by $(s,\phi,re^{i\theta}) \mapsto (\phi,\theta,h(r,s),r)$ induces a map
$\E_{\D^\ast} \to \B$ that commutes with the projections to $\D$ and is a
homeomorphism into its image.

The map $B \to \C^\ast \times \D$ that takes $(\phi,\theta,t,r)$ to
$(k(r,t)e^{i\phi},re^{i\theta})$ induces a map $\B \to \E_\D$ such that the
diagram
$$
\xymatrix{
\B \ar[r]\ar[d] & \E_\D \ar[d] \cr
\Bl^o_0\D \ar[r] & \D
}
$$
commutes. For each $\theta \in S^1$, the composite
$
\C^\ast \hookrightarrow \Bl^o_{0,\infty}\P^1 \to B_\theta \to E_0
$
is the natural inclusion of the smooth locus of the nodal cubic $E_0$ given by
the parameter $w=se^{i\phi}$. The map $\Bl^o_{0,\infty}\P^1 \to B_\theta \to
E_0$ collapses the boundary of $\Bl^o_{0,\infty}\P^1$ to the double point of
$E_0$.


\begin{thebibliography}{99}


\bibitem{kzb}
D.~Bernard:
{\em On the Wess-Zumino-Witten models on the torus},
Nuclear Phys.\ B 303 (1988), 77--93. 

\bibitem{brown}
F.~Brown:
{\em Mixed Tate motives over $\Z$}, Ann.\ of Math.\ 175 (2012), 949--976.

\bibitem{cee}
D.~Calaque, B.~Enriquez, P.~Etingof:
{\em Universal KZB equations: the elliptic case}, in Algebra, arithmetic, and
geometry: in honor of Yu.~I.~Manin. Vol.~I, 165--266, Progr.\ Math., 269,
Birkh\"auser, Boston, 2009, \comment{arXiv:math/0702670}

\bibitem{carlson}
J.~Carlson:
{\em Extensions of mixed Hodge structures}, Journ\'ees de G\'eometrie
Alg\'ebrique d'Angers, Juillet 1979, Sijthoff \& Noordhoff, 1980, 107--127.

\bibitem{chen}
K.-T.~Chen:
{\em Iterated path integrals}, Bull.\ Amer.\ Math.\ Soc.\ 83 (1977), 831--879.

\bibitem{deligne:ode}
P.~Deligne:
{\em \'Equations diff\'erentielles \`a  points singuliers r\'eguliers}, Lecture
Notes in Mathematics, Vol.~163. Springer-Verlag, 1970.

\bibitem{deligne:hodge2}
P.~Deligne:
{\em Th\'eorie de Hodge, II}, Inst.\ Hautes \'Etudes Sci.\ Publ.\ Math.\ No.~40
(1971), 5--57.

\bibitem{deligne:p1}
P.~Deligne:
{\em Le groupe fondamental de la droite projective moins trois points},  Galois
groups over $\Q$ (Berkeley, CA, 1987), 79--297,  Math.\ Sci.\ Res.\ Inst.\
Publ., 16, Springer, 1989.

\bibitem{enriquez}
B.~Enriquez:
{\em Elliptic associators}, Selecta Math.\ (N.S.) 20 (2014), 491--584.

\bibitem{furusho}
H.~Furusho:
{\em The multiple zeta value algebra and the stable derivation algebra}, Publ.\
Res.\ Inst.\ Math.\ Sci.~39  (2003), 695--720. \comment{arXiv:math/0011261}

\bibitem{hain:bowdoin}
R.~Hain:
{\em The geometry of the mixed Hodge structure on the fundamental group},
Algebraic Geometry, 1985, Proc.\ Symp.\ Pure Math.\ 46 (1987), 247--282.

\bibitem{hain:dht1}
R.~Hain:
{\em The de Rham homotopy theory of complex algebraic varieties, I}.\ K-Theory
1 (1987), 271--324.

\bibitem{hain:dht2}
R.~Hain:
{\em The de Rham homotopy theory of complex algebraic varieties, II}.\ K-Theory
1 (1987), 481--497.


\bibitem{hain:weight}
R.~Hain:
{\em Relative weight filtrations on completions of mapping class groups}, in
Groups of Diffeomorphisms, Advanced Studies in Pure Mathematics, vol.~52 (May,
2008), 309--368, Mathematical Society of Japan, \comment{arXiv:0802.0814}

\bibitem{hain:elliptic}
R.~Hain:
{\em Lectures on Moduli Spaces of Elliptic Curves}, in Transformation Groups and
Moduli Spaces of Curves, Advanced Lectures in Mathematics, edited by Lizhen Ji,
S.-T.~Yau no.~16 (2010), 95--166, Higher Education Press, Beijing,
\comment{arXiv:0812.1803}

\bibitem{hain:lmhs}
R.~Hain:
{\em Unipotent Path Torsors of Ihara Curves}, in preparation.

\bibitem{hain:vancouver} R.~Hain: {\em The Hodge-de~Rham Theory of Modular
Groups} in Recent Advances in Hodge Theory Period Domains, Algebraic Cycles, and
Arithmetic, Matt Kerr and Greogory Pearlstein, editors, 422--514. Cambridge
University Press, 2016. \comment{arXiv:1403.6443}

\bibitem{hain-matsumoto:mem}
R.~Hain, M.~Matsumoto:
{\em Universal mixed elliptic motives}, J.\ Inst.\ Math.\ Jussieu 19 (2020), 663--766. \comment{arXiv:1512.03975}

\bibitem{hain-zucker}
R.~Hain, S.~Zucker:
{\em Unipotent variations of mixed Hodge structure}, , Invent.\ Math.\ 88
(1987), 83--124.

\bibitem{kashiwara}
M.~Kashiwara:
{\em A study of variation of mixed Hodge structure},  Publ.\ Res.\ Inst.\ Math.\
Sci.\ 22 (1986),  no.~5, 991--1024.

\bibitem{kronecker}
Leopold Kronecker:
{\em Zur theorie der elliptischen Funktionen} (1881). Reproduced in:
{\em Leopold Kronecker's Werke Volume IV}, pp.~313--318. Herausgegeben auf
Veranlassung der K\"oniglich Preussischen Akademie der Wissenschaften von
K.~Hensel, Chelsea, New York 1968.

\bibitem{kz}
V.~Knizhnik, A.~Zamolodchikov:
{\em Current algebra and Wess--Zumino model in two dimensions}, Nuclear Phys.\ B
247 (1984), 83--103.

\bibitem{le-murakami}
T.~Le, J.~Murakami:
{\em Kontsevich's integral for the Kauffman polynomial}, 
Nagoya Math.\ J.\ 142 (1996), 39--65.

\bibitem{levin-racinet} A.~Levin, G.~Racinet: {\em Towards multiple elliptic
polylogarithms}, unpublished preprint, \comment{arXiv:math/0703237}

\bibitem{pollack}
A.~Pollack:
{\em Relations between derivations arising from modular forms}, undergraduate
thesis, Duke University, 2009.\footnote{Available at: {\sf
http://dukespace.lib.duke.edu/dspace/handle/10161/1281}}

\bibitem{schmid}
W.~Schmid:
{\em Variation of Hodge structure: the singularities of the period mapping},
Invent.\ Math.\ 22 (1973), 211--319. 

\bibitem{serre}
J.-P.~Serre:
{\em A course in arithmetic}, Translated from the French. Graduate Texts in
Mathematics, No. 7. Springer-Verlag, New York-Heidelberg, 1973.

\bibitem{serre:la-lg}
J.-P.~Serre:
{\em Lie algebras and Lie groups, Lectures given at Harvard University, 1964},
Benjamin, 1965. (Second edition, Springer Verlag, 1992.)

\bibitem{silverman}
J.~Silverman:
{\em Advanced Topics in the Arithmetic of Elliptic Curves}, Graduate Texts in
Mathematics, 151. Springer-Verlag, New York, 1994.

\bibitem{steenbrink-peters}
C.~Peters, J.~Steenbrink:
{\em Mixed Hodge structures}, Ergebnisse der Mathematik und ihrer Grenzgebiete,
volume 52. Springer-Verlag, 2008. 

\bibitem{steenbrink-zucker}
J.~Steenbrink, S.~Zucker:
{\it Variation of mixed Hodge structure, I},  Invent.\ Math.\ 80 (1985),
489--542.

\bibitem{tsunogai}
H.~Tsunogai:
{\em On some derivations of Lie algebras related to Galois representations},
Publ.\ Res.\ Inst.\ Math.\ Sci.\ 31 (1995), 113--134.


\bibitem{zagier}
D.~Zagier:
{\em Periods of modular forms and Jacobi theta functions},  Invent.\ Math.\ 
104  (1991),  449--465.

\end{thebibliography}
\end{document}